\documentclass{sfuthesis}

\title{Graph Invariants with Connections to the Feynman Period in $\phi^4$ Theory.}
\thesistype{Dissertation}
\author{Iain Crump}
\previousdegrees{%
	M.Sc., Simon Fraser University, 2012\\
	B.Sc., University of Winnipeg, 2009}
\degree{Doctor of Philosophy}
\discipline{Mathematics}
\department{Department of Mathematics}
\faculty{Faculty of Science}
\copyrightyear{2017}
\semester{Spring 2017}
\date{April 7, 2017}


\committee{%
	\chair{Dr.\ Cedric Chauve}{Professor}
	\member{Dr.\ Karen Yeats}{Senior Supervisor\\Associate Professor}
	\member{Dr.\ Matt DeVos}{Supervisor\\Associate Professor}
	\member{Dr.\ Bojan Mohar}{Internal/External Examiner\\Professor}
	\member{Dr.\ Robert \v{S}\'{a}mal}{External Examiner\\Associate Professor\\Computer Science Institute\\Charles University}
}

\usepackage{amsmath,amssymb,amsthm,amsfonts}
\usepackage[pdfborder={0 0 0}]{hyperref}
\usepackage{graphicx,color,colortbl}
\usepackage{caption,subcaption}
\usepackage{tikz}
\usepackage[compat=1.0.0]{tikz-feynman}

\newtheorem{theorem}{Theorem}
\newtheorem*{theorem*}{Theorem}
\newtheorem{corollary}[theorem]{Corollary}
\newtheorem{proposition}[theorem]{Proposition}
\newtheorem{lemma}[theorem]{Lemma}
\newtheorem{conjecture}{Conjecture}
\newtheorem*{conjecture*}{Conjecture}
\newtheorem{question}[conjecture]{Question}
\newtheorem*{question*}{Question}

\newtheorem*{Menger}{Menger's Theorem}

\newtheorem*{wilson}{Wilson's Theorem}

\theoremstyle{definition}
\newtheorem*{example}{Example}
\newtheorem{definition}[theorem]{Definition}
\newtheorem{remark}[theorem]{Remark}

\def\V{\mathfrak{v}}
\def\E{\mathfrak{e}}

\def\lcm{\text{lcm}}

\def\cut{\setminus}
\def\contract{/}

\newcommand{\2}[1]{\mathbf{2}^{(#1)}}

\graphicspath{{./Figures/}}
\allowdisplaybreaks[1]

\begin{document}

\frontmatter
\maketitle{}
\makecommittee{}

\begin{abstract}

Feynman diagrams in scalar $\phi^4$ theory have as their underlying structure $4$-regular graphs. In particular, any $4$-point $\phi^4$ graph can be uniquely derived from a $4$-regular graph by deleting a single vertex. The Feynman integral is encoded by the structure of the graph and is used to determine amplitudes of particle interactions. The Feynman period is a simplified version of the Feynman integral. The period is of special interest, as it maintains much of the important number theoretic information from the Feynman integral. It is also of structural interest, as it is known to be preserved by a number of graph theoretic operations. In particular, the vertex deleted in constructing the $4$-point graph does not affect the Feynman period, and it is invariant under planar duality and the Schnetz twist, an operation that redirects edges incident to a $4$-vertex cut. Further, a $4$-regular graph may be produced by identifying triangles in two $4$-regular graphs and then deleting these edges. The Feynman period of this graph with a vertex deleted is equal to the product of the Feynman periods of the two smaller graphs with one vertex deleted each. These operations currently explain all known instances of non-isomorphic $4$-point $\phi^4$ graphs with equal periods.

With this in mind, other graph invariants that are preserved by these operations for $4$-point $\phi^4$ graphs are of interest, as they may provide insight into the Feynman period and potentially the integral. In this thesis the extended graph permanent is introduced; an infinite sequence of residues from prime order finite fields. It is shown that this sequence is preserved by these three operations, and has a product property. Additionally, computational techniques will be established, and an alternate interpretation will be presented as the point count of a novel graph polynomial.

Further, the previously existing $c_2$ invariant and Hepp bound are examined, two graph invariants that are conjectured to be preserved by these graph operations. A combinatorial approach to the $c_2$ invariant is introduced.  \end{abstract}

\begin{acknowledgements} I would like to thank Francis Brown, Matt DeVos, Erik Panzer, Scott Sitar, and of course Karen Yeats for their helpful notes and ideas, and fantastic support.  \end{acknowledgements}

\addtoToC{Table of Contents}\tableofcontents\clearpage
\addtoToC{List of Figures}\listoffigures

%
%

\mainmatter%

\chapter{Introduction}
\label{Introduction}

\section{Background}\label{introbackground}

A goal of a quantum field theory is to understand the various interactions of the fundamental particles of which the universe is composed (see, for example, \cite{browngen}, \cite{qftnotes}, and  \cite{qftbook}). Each such theory restricts its allowable particle types and immediate interactions between these fundamental particles based on experimental observations. These interactions are encoded in \emph{Feynman diagrams}, fundamentally these are multigraphs with specific edge types, representing the different types of particles. Further, these diagrams allow \emph{external edges}, edges that connect to one vertex only. These external edges represent particles as they enter or exit the system. We may therefore view edges in the diagram as pairings of half-edges, while external edges are half-edges that remain unpaired.

The following example uses a framework that can be found in \cite{karensbook}.

\begin{example} Quantum electrodynamics is the quantum theory of electromagnetism. It has three half-edge types; a front-half fermion, a back-half fermion, and a photon. Edges may be constructed by combining front- and back-half fermions to create a fermion edge,  \begin{tikzpicture}
\begin{feynman}
\vertex(a);
\vertex[right=1cm of a](b);
\diagram*{
  (a) --[fermion](b),
 };
\end{feynman}
\end{tikzpicture}, or by combining two photon half-edges,  \begin{tikzpicture}
\begin{feynman}
\vertex(a);
\vertex[right=1cm of a](b);
\diagram*{
  (a) --[photon](b),
 };
\end{feynman}
\end{tikzpicture}. There is one interaction type, consisting of one of each half-edge types. A fermion edge oriented in the direction one is reading represents an electron, and against represents a positron. Reading Figure~\ref{qedex} left to right, then, an electron and positron combining to form a photon, which propagates for a time before splitting into an electron and positron again. \end{example}

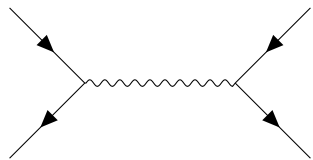
\begin{figure}[h] \centering
\begin{tikzpicture}
\begin{feynman}
\vertex(a);
\vertex[right=4cm of a](b);
\vertex[right=1cm of a](v1);
\vertex[right=3cm of a](v2);
\vertex[above=1cm of a](i1);
\vertex[below=1cm of a](i2);
\vertex[above=1cm of b](f1);
\vertex[below=1cm of b](f2);
\diagram*{
{[edges=fermion]
  (i1) -- (v1) --(i2),
  (f1) --(v2) --(f2),},
  (v1) --[photon](v2),
 };
\end{feynman}
\end{tikzpicture}
\caption{A Feynman diagram in quantum electrodynamics. All fermion type edges here are unpaired half-edges.}
\label{qedex}
\end{figure}

This thesis will focus on $\phi^4$ theory\footnote{Here, $\phi$ is the usual name of the field. The field appears to the forth power in the Lagrangian, giving the $4$-valent vertices and also the name $\phi^4$.}, which has allowable half-edge \begin{tikzpicture}
\begin{feynman}
\vertex(a);
\vertex[right=.5cm of a](b);
\diagram*{
  (b) --(a),
 };
\end{feynman}
\end{tikzpicture}, an undecorated and undirected edge, and interaction \begin{tikzpicture}
\begin{feynman}
\vertex(a);
\vertex[right=.3cm of a](v1);
\vertex[right=.3cm of v1](b);
\vertex[above=.3cm of v1](c);
\vertex[below=.3cm of v1](d);
\diagram*{
  (v1) --(a),
  (v1) --(b),
  (v1) --(c),
  (v1) --(d),
 };
\end{feynman}
\end{tikzpicture}. In particular, a $k$-point $\phi^4$ graph has precisely $k$ external edges, and when considering the motivating physics we restrict to $4$-point graphs in $\phi^4$ theory. These graphs can be uniquely derived from $4$-regular graphs by deleting a single vertex and all incident half-edges, leaving the remaining half-edges unpaired.

\def\diff{\mathrm{d}}

The Feynman diagram encodes the Feynman integral. Let $D$ be the dimension of space-time appropriate to the physical theory. Both quantum electrodynamics and $\phi^4$ theory, for example, are $4$-dimensional physical theories, corresponding the standard three space and one time dimensions. The \emph{loop number} (also known as the first Betti number) of a Feynman diagram $G$, $h_G$, is the minimum number of edges that must be removed to produce an acyclic graph. Suppose the Feynman diagram $G$ has $n$ internal edges, $m$ external edges, and each internal edge $e_j$ is assigned an orientation and has momentum $q_j$ and mass $m_j$. We impose momentum conservation at each vertex; the sum of momenta in must be equal to the sum of momenta out. This is akin to the graph theoretic notion of flows that will arise in numerous places throughout this thesis. As a result, we may express momenta flowing through internal lines as a linear combination of the independent loop momenta $k_1, ... , k_{h_G} \in \mathbb{R}^D$ and the external momenta $p_1 , ... , p_m \in \mathbb{R}^D$ as $$ q_i = \sum_{j=1}^{h_G} \rho_{ij}k_j + \sum_{j=1}^m \sigma_{ij}p_j, \hspace{1cm} \rho_{ij},\sigma_{ij} \in \{0, \pm1\} ,$$ where the $\sigma$ and $\rho$ values are determined by the direction of these internal lines and the orientation of these edges. From Equation 1.11 in \cite{bigintegralcite}, the scalar Feynman integral as a potentially divergent formal integral expression is $$ \int_{(\mathbb{R}^D)^{h_G}} \prod_{r=1}^{h_G} \diff^Dk_r \prod_{j=1}^n \frac{1}{(q_j^2-m_j^2)}.$$  The first step to make a convergent integral is to regularize. Two common choices are to raise each factor in the denominator to a non-integer parameter, or to take $D$ as $D - 2 \epsilon$ (see \cite{bogner}). Multiplication of vectors in $\mathbb{R}^D$ is taken to be the dot product, following standard physics notations. Non-scalar field theories result in more complicated Feynman integrals, as every edge and vertex will be represented by more complicated factors in the integral. However, $\phi^4$ theory is a scalar field theory, so this representation is sufficient for our needs. 

After parametrization\footnote{In literature, this is known as Schwinger parametrization or Feynman parametrization.} the Feynman integral is  (Equation 1.12 in \cite{bigintegralcite}) $$ \int_{x_i \geq 0} \left. \frac{\mathcal{U}_G^{h_G(D/2)-n}}{\Psi_G^{(h_G+1)(D/2)-n}} \right|_{x_n=1} \prod_{i=1}^{n-1} \diff x_i $$ (see \cite{erikthesis}, Chapter 2, for a detailed explanation of how the first integral is translated to this form). In the denominator, $\Psi$ is the \emph{Kirchhoff polynomial} (also known as the first Symanzik polynomial). Let $T_G$ be the set of spanning trees of diagram $G$, and for all $e \in E(G)$ assign variable $x_e$, the \emph{Schwinger parameter}. The Kirchhoff polynomial of $G$ is $$ \Psi_G = \sum_{T \in T_G} \left( \prod_{e \notin T} x_e \right). $$ This was introduced in \cite{ueber} as a tool for understanding electrical systems. Similarly, $\mathcal{U}$ is the second Symanzik polynomial. The sum in the second Symanzik polynomial is over $F_G$, the set of spanning forests with precisely two connected components. Using conventions from \cite{bogner}, and writing $P_{T_i}$ as the set of external momenta in tree $T_i$, $$\mathcal{U}_G = \sum_{(T_1, T_2) \in F_G} \left( \prod_{e \notin (T_1 \cup T_2)} x_e \right) \left( \sum_{p_j \in P_{T_1}} \sum_{p_k \in P_{T_2}} p_jp_k   \right)  + \Psi_G \sum_{i=1}^n x_i m_i^2. $$ 


It is of interest that the Feynman integral tends to diverge. \emph{Renormalization}, given mathematical foundation in \cite{renormpaper}, is the method by which this integral, as a part of a larger computation, is made to match experimentally observed values. A brief historical review can be found in \cite{renorm}.

\section{Feynman periods}\label{introperiods}

\subsection{Derivation from the Feynman integral}\label{subintroperiods}

Our interests lie in the Feynman period. For a diagram $G$ with $n$ internal edges, the Feynman period is $$\int_{x_i \geq 0} \frac{1}{\Psi^{D/2}|_{x_{n}=1}} \prod_{i=1}^{n-1} \diff x_i ,$$ a residue of the Feynman integral of $G$ viewed as a Feynman diagram in massless scalar field theory, with all external parameters and masses set to zero. As a result, we may ignore completely all external edges, and the Feynman diagrams truly are graphs or multigraphs. Throughout, standard graph theory terminology and notation will be assumed, following \cite{generalgraphtheory}. The period is an important object (see, for example, \cite{BlEsKr,Brbig,marcolli}). When the Feynman integral diverges, the period itself is the coefficient at infinity (see Section 6.2.3 in \cite{qftbook}). When the period converges, it does so independent of renormalization scheme. It is also understood precisely when the period is convergent, which we describe now.

Recall the loop number of a graph $G$, $h_G$, is the minimum number of edges that must be removed to produce an acyclic graph. A $\phi^4$ graph is \emph{primitive} (often \emph{primitive log divergent} in the literature) if $|E(G)| = 2h_G$ and for all proper subgraphs $H \subset G$, $|E(H)| > 2h_H$. Any subgraph that defies this relation is called a \emph{subdivergence}. For a $4$-regular graph, a subdivergence is a non-trivial $2$- or $4$-edge cut.

\begin{theorem*}[Proposition 5.2 in \cite{BlEsKr}] The period of a $\phi^4$ graph converges if and only if the graph is primitive. \end{theorem*}

Despite being a simplification of the Feynman integral, the period is still difficult to compute. Many can be expressed as multiple zeta values; sums and products of $$\zeta ( s_1 , ... , s_k) = \sum_{n_1 > \cdots > n_k > 0} \frac{1}{n_1^{s_1} \cdots n_k^{s_k}}$$ over $\mathbb{Q}$ for positive integers $n_i,s_i$, $1 \leq i \leq k$. These are objects of mathematical interest, due to the algebraic and number theoretic properties they possess, and their connections to other mathematical objects (see for example \cite{hoffman} and \cite{zudilin}). Computations for $4$-point $\phi^4$ graphs up to seven loops are presented by Broadhurst and Kreimer in \cite{knotsnumbers}, up to eight loops by Schnetz in \cite{Sphi4}, and for an important and infinite family of $\phi^4$ graphs by Brown and Schnetz in \cite{BSZigzag}. Both \cite{knotsnumbers} and \cite{Sphi4} use numeric techniques to produce numbers that match with consistently high degrees of accuracy, while \cite{BSZigzag} relies on recently developed tools to establish the equality with mathematical rigor. One way to deal with this computational difficulty of the integral Feynman period is to create large families of graphs with equal periods. To do this, we rely on graph operations that preserve the period. Four such operations are known for primitive $4$-point $\phi^4$ graphs.

\subsection{Operations that preserve the Feynman period}\label{introinvariance}

As observed in Section~\ref{introbackground}, any $4$-point $\phi^4$ graph can be uniquely created from a $4$-regular graph by deleting a vertex. This operation is called \emph{decompletion}, and the unique way a vertex can be added back to a $4$-point $\phi^4$ graph is \emph{completion}. 

\begin{theorem*}[Theorem 2.7 in \cite{Sphi4}] The period of a $4$-point $\phi^4$ graph is invariant under completion followed by decompletion. \end{theorem*}

\noindent As the graphs in the set of decompletions of a $4$-regular graph are not necessarily isomorphic, this can produce a number of non-isomorphic graphs with equal periods.

\begin{example} Consider the graphs drawn in Figure~\ref{completionex}. Graph $\Gamma$ is $4$-regular, while $G_1$ and $G_2$ are decompletions of $\Gamma$. As $G_1$ and $G_2$ have different numbers of edges in parallel, they are clearly non-isomorphic. 

\begin{figure}[h]
  \centering
      \includegraphics[scale=0.80]{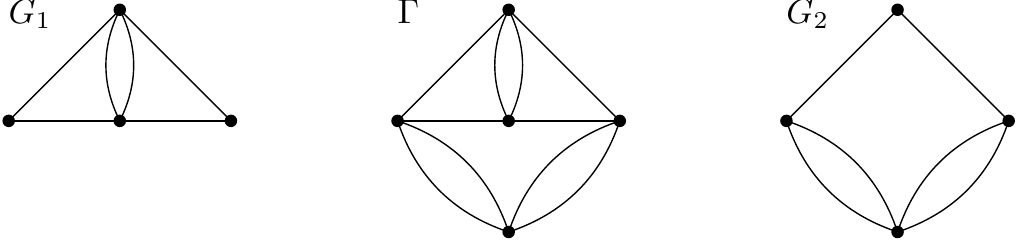}
  \caption{An example of completion followed by decompletion producing non-isomorphic $4$-point $\phi^4$ graphs.}
\label{completionex}
\end{figure}

  \end{example}

The Schnetz twist is another operation known to preserve the period. Shown in Figure~\ref{twist}, we partition the edges of a completed graph across a four-vertex cut. On one side of this cut, we then change the ends of edges of the form $\{v_i,w\}$, exchanging $v_1$ with $v_2$ and $v_3$ with $v_4$, using the labeling from this figure. We assume that both graphs are $4$-regular. 

\begin{theorem*}[Theorem 2.11 in \cite{Sphi4}] If two $4$-regular graphs differ by a Schnetz twist, then any pair of decompletions of these graphs have equal periods. \end{theorem*}

\begin{figure}[h]
  \centering
      \includegraphics[scale=0.80]{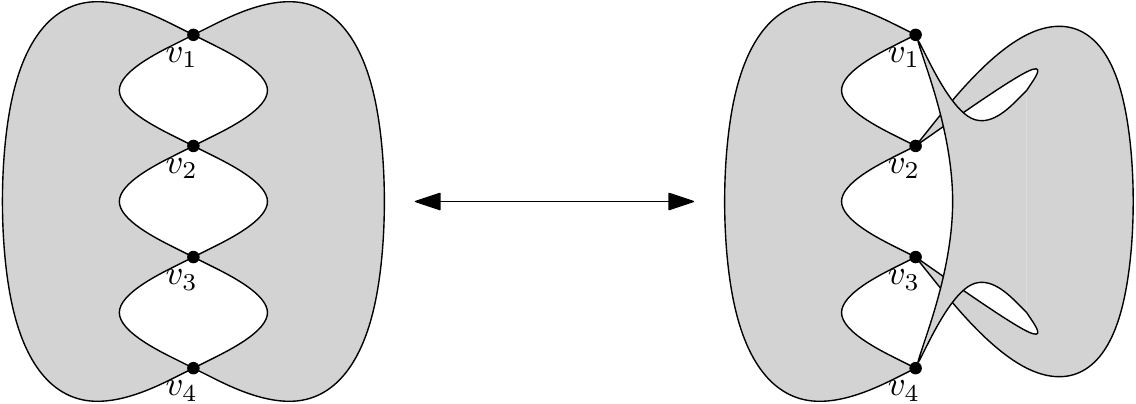}
  \caption{The Schnetz twist. If both graphs are $4$-regular, then all decompletions of these two graphs have equal Feynman periods.}
\label{twist}
\end{figure}

\begin{example} The completed $\phi^4$ graphs in Figure~\ref{schnetzex} use the naming convention from \cite{Sphi4}. There is a $4$-vertex cut, aligned vertically in the centre, and the two graphs differ by a Schnetz twist, performed on the edges to the right of the cut. These are non-isomorphic; $P_{7,4}$ contains five triangles, $P_{7,7}$ contains six. \end{example}

\begin{figure}[h]
  \centering
      \includegraphics[scale=0.80]{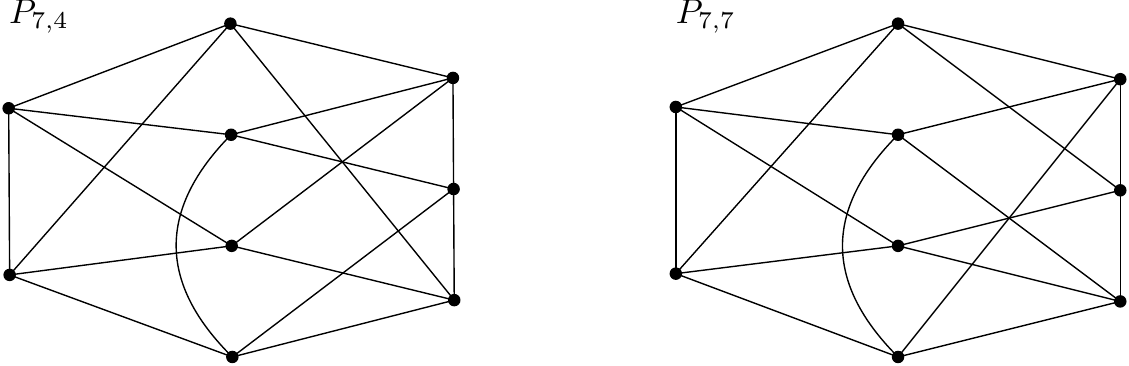}
  \caption{A Schnetz twist on a graph with seven loops.}
\label{schnetzex}
\end{figure}

Finally, the dual of a connected planar graph is a well-known graphic operation. Importantly, a $4$-point $\phi^4$ graph $G$ has $|E(G)| = 2|V(G)|-2$. Recall Euler's polyhedral formula; setting $F(G)$ as the set of faces for an embedding of a connected graph $G$, $|V(G)| -|E(G)| +|F(G)| = 2$. If $G$ is planar, then dual $G^*$ has $|V(G^*)| = |F(G)| = 2-|V(G)|+|E(G)|$. Then, $|E(G)| = |V(G^*)|+|V(G)|-2$, and hence $|V(G)|=|V(G^*)|$. As $|E(G)| = |E(G^*)|$ also, both $G$ and $G^*$ have the same vertex to edge relationship. 

\begin{theorem*}[Theorem 2.13 in \cite{Sphi4}, used as early as \cite{knotsnumbers}] Suppose graph $G$ and dual $G^*$ are $4$-point $\phi^4$ graphs. The periods of $G$ and $G^*$ are equal.  \end{theorem*}

\begin{example}  Again using the naming convention from \cite{Sphi4}, the graphs in Figure~\ref{planarex} are decompletions of $P_{7,5}$ and $P_{7,10}$, respectively. Both are $4$-point $\phi^4$ graphs. These are non-isomorphic; the vertices of degree three in $P_{7,5}$ are an independent set, while those in $P_{7,10}$ are not.  \end{example}

\begin{figure}[h]
  \centering
      \includegraphics[scale=0.80]{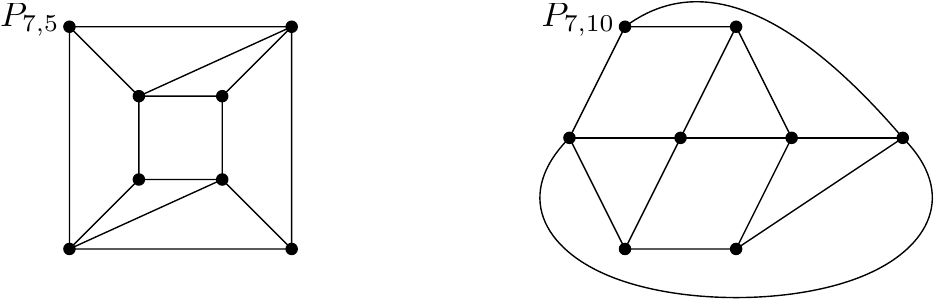}
  \caption{An example of two non-isomorphic $4$-point $\phi^4$ graphs that are dual to one another.}
\label{planarex}
\end{figure}

Lastly, while not period preserving itself, decompleted graphs with $2$-vertex cuts have an important property with regards to the period. Split the graph $G$ into minors $G_1$ and $G_2$ as in Figure~\ref{2cut} and assume that $G$, $G_1$, and $G_2$ are all $4$-point $\phi^4$ graphs. External half-edges must be distributed as shown. The operation is shown for completed graphs in Figure~\ref{2cutcompleted}.

\begin{figure}[h]
  \centering
      \includegraphics[scale=1.20]{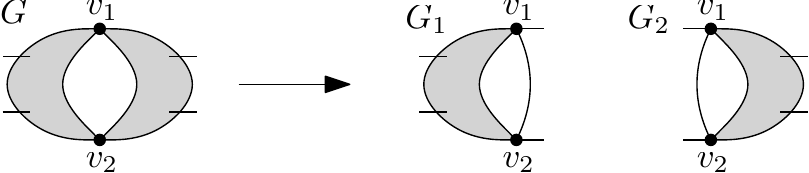}
  \caption{Operation on a two-vertex cut. If all graphs are $4$-point $\phi^4$, then the period of $G$ is equal to the product of the periods of $G_1$ and $G_2$.}
\label{2cut}
\end{figure}

\begin{theorem*}[Theorem 2.10 in \cite{Sphi4}] Suppose graphs $G$, $G_1$, and $G_2$ in Figure~\ref{2cut} are all $4$-point $\phi^4$ graphs. The period of $G$ is equal to the product of the periods of $G_1$ and $G_2$.   \end{theorem*}

\begin{figure}[h]
  \centering
      \includegraphics[scale=0.90]{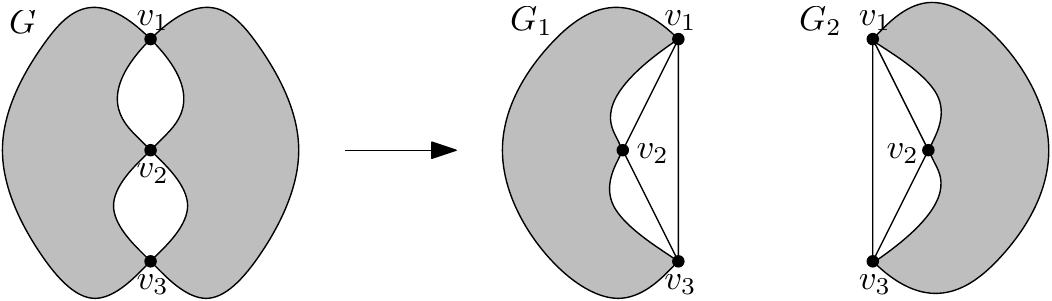}
  \caption{Operation on a three-vertex cut. If all graphs are in $4$-regular graphs, then the period of any decompletion of $G$ is equal to the product of the periods of decompletions of both $G_1$ and $G_2$.}
\label{2cutcompleted}
\end{figure}

\noindent It is therefore possible to produce non-isomorphic $4$-point $\phi^4$ graphs with equal periods by merging smaller graphs in this manner; either merging the same two graphs at different pairs of edges, or merging different graphs that differ by an aforementioned graph operation. From a collection of graphs with known periods, it is possible to produce an infinite family of graphs with easily computed periods, unlike the previous structural operations.

\begin{example} In Figure~\ref{mergingex}, the graph $G_1$ is a decompletion of $P_{6,2}$ and $G_2$ is a decompletion of $P_{3,1} = K_5$. Two ways of merging these graphs are shown below. These are again non-isomorphic, as the number of triangles is different between the two graphs. \end{example}

\begin{figure}[h]
  \centering
      \includegraphics[scale=0.7]{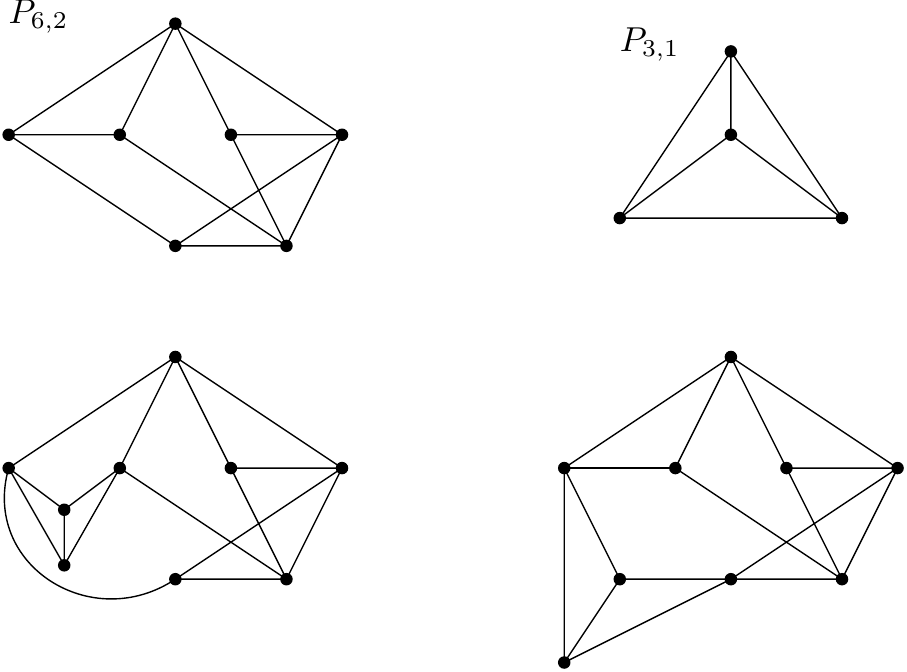}
  \caption{Using the $2$-vertex cut property to produce two non-isomorphic graphs with equal periods.}
\label{mergingex}
\end{figure}

The four operations listed above explain all currently known instances of $4$-point $\phi^4$ graphs with equal periods. Any non-trivial graph invariant that is preserved by these is therefore of interest, as it may provide insight into the Feynman period. The creation of such an invariant is one of the main topics of this thesis.

\section{Outline of the thesis}\label{introoutline}

Chapters~\ref{hepp} and~\ref{c2} introduce the Hepp bound and $c_2$ invariant, respectively. These are non-trivial invariants conjectured to be preserved by all operations known to preserve the period. A novel graph-theoretic interpretation of the $c_2$ invariant is introduced in Section~\ref{graphicc2}.

Chapter~\ref{chmatroid} provides a brief introduction to matroid theory, and some of the tools therein that will be of use. Those familiar with the subject may safely skip this chapter.

We previously created the extended graph permanent, and introduce it here in Chapter~\ref{Extended}. For an arbitrary graph, the extended graph permanent is a sequence of residues from (an infinite subset of) prime order finite fields. Constructed from the matrix permanent, Section~\ref{egppermanent} lays the foundation and establishes the required linear algebra tools, and Section~\ref{egpegp} introduces the invariant. Useful to our understanding of the extended graph permanent as a combinatorial object, we introduce a graphic formulation of this object in Section~\ref{graphicegp}. Section~\ref{signambiguity} briefly discusses a natural sign ambiguity that exists in this invariant. Section~\ref{nowherezero} explores a connection between the extended graph permanent and nowhere-zero flows.

In Chapter~\ref{Invariance} we prove that the extended graph permanent is in fact preserved by all previously mentioned operations known to preserve the period. 

\begin{theorem*} Let $\Gamma$ be a $4$-regular graph. 
\begin{itemize}
\item For $v,w\in V(\Gamma)$, the extended graph permanent of $\Gamma - v$ is equal to the extended graph permanent of $\Gamma-w$ (Theorem~\ref{egpcompletion}).
\item If $\Gamma$ and $\Gamma'$ are both $4$-regular graphs that differ by a Schnetz twist, then any decompletions of $\Gamma $ and $\Gamma'$ will have equal extended graph permanents (Proposition~\ref{schnetz}).
\item For $v \in V(\Gamma)$, let $G = \Gamma - v$. If $G$ is planar and $G^*$ is its planar dual, then $G$ and $G^*$ have equal extended graph permanents (Corollary~\ref{phi4dual}).
\item If $G=\Gamma -v$ has a $2$-vertex cut that can be split into two $4$-point $\phi^4$ minors as in Figure~\ref{2cut}, then the extended graph permanent of $G$ is the additive inverse of the (term-by-term) product of the extended graph permanents of $G_1$ and $G_2$ (Theorem~\ref{2vertexcut}).
\end{itemize}
\end{theorem*}

\noindent This of course leads naturally to the following conjecture.

\begin{conjecture*}[Conjecture~\ref{obviousconjecture}] If two $4$-point $\phi^4$ graphs have equal periods, then they have equal extended graph permanents. \end{conjecture*}

\noindent This connection to the Feynman period is further hinted at in Theorem~\ref{subdivthm}, which proves a similar term-by-term product property for $4$-regular graphs with $4$-edge cuts, corresponding to graphs with subdivergences.

Chapter~\ref{egpcomp} develops computational methods for finding the extended graph permanent. Specifically, permanent values are difficult to compute, and the extended graph permanent produces sequences that, even for small graphs, quickly require residues modulo prime $p$ of the permanents of obscenely large matrices. In this chapter, we develop methods for finding closed forms for all values in the sequence for any graph. We further find  closed forms for the extended graph permanent for the family of trees in Section~\ref{treesnshit}, wheels in Section~\ref{wheels}, and zig-zag graphs in section~\ref{zigzagcomp}. In producing these sequences we also observe a possible connection to the $c_2$ invariant.

In Chapter~\ref{hyperpower} we find a novel graph polynomial $\widetilde{F}_{G}$ such that the extended graph permanent can be represented as a point count over this polynomial. For function $f$ and prime $p$, let $[f]_p$ be the number of roots of $f$ over finite field $\mathbb{F}_p$.

\begin{theorem*}[Corollary~\ref{maincor}] Let $G$ be a $4$-point $\phi^4$ graph. The extended graph permanent of $G$ at $p$ is $$ \begin{cases} [\widetilde{F}_{G}]_p \pmod{p} \text{ if } |E(G)| \equiv 0 \pmod{4} \\ -[\widetilde{F}_{G}]_p \pmod{p} \text{ otherwise} \end{cases} .$$ \end{theorem*}

\noindent A number of extended graph permanent sequences also appear to relate to modular forms, as the $p^\text{th}$ Fourier coefficient modulo $p$. This is discussed in greater detail in Section~\ref{modformscoeffs}. It is interesting to note that in all observed instances, the loop number of the graph is equal to the weight of the modular form, the level of the modular form is a power of two, and in the Dirichlet character decomposition these all fall into subspaces of dimension $1$.

Finally, we conclude with Chapter~\ref{conclusion}. Here, we summarize the main results of this thesis, and indicate areas of potential future interest.

\chapter{The Hepp bound}
\label{hepp}

The Hepp bound was introduced by Panzer in 2016. The properties and results discussed here are due to Panzer unless otherwise stated (\cite{ErikEmail}). New material on the Hepp bound has also been published in \cite{Schhepp}.

For a graph $G$, the Hepp bound $\mathcal{H}(G)$ is an upper bound for the period, created by replacing the Kirchhoff polynomial in the period formula with the maximal monomial at all points of integration; $$\mathcal{H}(G) = \int_{x_i \geq 0} \frac{\prod_{i=1}^{|E(G)|-1}\diff x_i}{\left( \max_{T \in T_G}  \left\{ \prod_{e \notin T} x_e \right\} \right)^2|_{x_{|E(G)|} = 1}} \in \mathbb{Q}.$$ By integrating over smaller denominators, this naturally creates an upper bound for the Feynman period.

\begin{example} The \emph{banana graph}, $P_{1,1}$ using the naming convention from \cite{Sphi4}, is the unique graph with two vertices, and two edges in parallel between them. The Kirchhoff polynomial for this graph is $\Psi_{P_{1,1}} = x_1 + x_2$. The Hepp bound of this graph is therefore \begin{align*} \mathcal{H}(P_{1,1})= \int_{x_1 \geq 0} \frac{\diff x_1}{\max \{1,x_1 \}^2} &= \int_{x_1=0}^1 \frac{\diff x_1}{1} + \int_{x_1=1}^\infty \frac{\diff x_1}{x_1^2} \\ &= 1 + (-x_1)^{-1}|_1^\infty \\ &= 1 + (0-(-1)) = 2. \end{align*} From \cite{galois}, the period of this graph is $1$. \end{example}

Computational complexity rises quickly. The next smallest primitive $4$-point $\phi^4$ graph, $K_4$, has sixteen summands in its Kirchhoff polynomial. A more graphic and structural interpretation of the Hepp bound makes this computation easier. To do this, we define $\mathcal{F}(G)$ to be the set of maximal chains of bridgeless subgraphs of $G$; $$\mathcal{F}(G) = \{ \emptyset = \gamma_0 \neq \gamma_1 \subsetneq \gamma_2 \subsetneq \cdots \subsetneq \gamma_{h_1(G)} = G: \text{ each } \gamma_i \text{ is bridgeless}\}.$$ The length of these chains and  the fact that $h_1(\gamma_i) = i$ follows from the required maximality.

\begin{proposition}[\cite{ErikEmail}]\label{heppcomp} Let $G$ be a graph and $\mathcal{F}_G$ the set of maximal bridgeless chains. Define $\omega(\gamma) = |E(\gamma)| - 2h_1(\gamma)$. Then, $$ \mathcal{H}(G) = \sum_{[\gamma] \in \mathcal{F}(G)} \frac{1}{\omega(\gamma_1) \cdots \omega(\gamma_{h_1(G)-1})} \prod_{k=1}^{h_1(G)}(|E(\gamma_k)|-|E(\gamma_{k-1})|) .$$ \end{proposition}

From a graph theoretic standpoint, these chains resemble ear decompositions (see Section 5.3 in \cite{banddagain}). The key distinction is that the subgraphs in the chains may be disconnected. Any subgraph that has fewer connected components than its predecessor must add two edge-disjoint paths between these connected components.

\begin{example} Consider a decompletion of graph $P_{3,1} = K_5$.  There are, up to symmetries, two possible maximal chains of bridgeless subgraphs; $$ \raisebox{-0.4\height}{\includegraphics{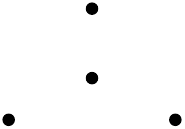}} \subsetneq  \raisebox{-0.4\height}{\includegraphics{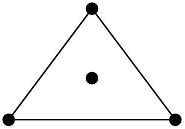}} \subsetneq \raisebox{-0.4\height}{\includegraphics{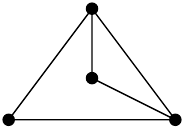}} \subsetneq \raisebox{-0.4\height}{\includegraphics{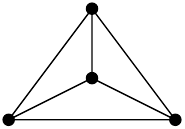}}$$ and 
$$ \raisebox{-0.4\height}{\includegraphics{hepp0}} \subsetneq  \raisebox{-0.4\height}{\includegraphics{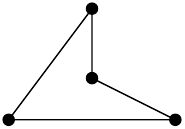}} \subsetneq \raisebox{-0.4\height}{\includegraphics{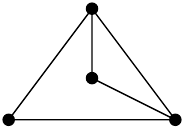}} \subsetneq \raisebox{-0.4\height}{\includegraphics{heppfull}} .$$ There are twelve chains of the first type, contributing $\frac{3\cdot 2 \cdot 1}{1 \cdot 1} = 6$ to the sum, and six chains of the second type, contributing $\frac{4 \cdot 1 \cdot 1}{2 \cdot 1} = 2$ to the sum. Using Proposition~\ref{heppcomp}, it follows that $\mathcal{H}(P_{3,1}) = 12 \cdot 6+6 \cdot 2 = 84$. Broadhurst proved in \cite{Broadhurst} that the period of this graph is $6\zeta(3) \approx 7.2$. \end{example}

As an upper bound, the Hepp bound does not appear to be incredibly precise. However, plotting the periods of primitive $4$-point $\phi^4$ graphs against these Hepp bounds for graphs up to eleven loops, there was an exceptionally strong correlation at every loop number. Panzer observed that the period of a graph $G$ was approximately $\frac{1.767}{4.493^{h_1(G)}}\mathcal{H}(G)^{4/3}$, with error approximately one percent. The data thus inspires the following conjecture.

\begin{conjecture}[\cite{ErikEmail}, Conjecture 3.2 in \cite{Schhepp}]\label{heppconjecture} For primitive $\phi^4$ graphs, the Hepp bounds are equal if and only if the periods are equal.   \end{conjecture}

\noindent This is unique among invariants studied here. For both the $c_2$ invariant (Conjecture~\ref{c2doingitsthing} in Chapter~\ref{c2}) and the extended graph permanent (Conjecture~\ref{obviousconjecture} in Chapter~\ref{Extended}), it is conjectured only that if two graphs have equal period, then they have equal $c_2$ invariant and extended graph permanent. In both cases, the converses of these statements are known to be false.

Of course, Conjecture~\ref{heppconjecture} implies that the Hepp bound must be preserved by completion followed by decompletion, planar duality, and the  Schnetz twist. Further, joining two graphs at a pair of edges as in Figure~\ref{2cut} should result in some sort of product property. Two of these properties have been established.

\begin{proposition}[Panzer] If the graph $G$ is planar, and $G^*$ its planar dual, are both $4$-point graphs in $\phi^4$ theory, then $G$ and $G^*$ have equal Hepp bounds.  \end{proposition}

\begin{proposition}[\cite{ErikEmail}] Suppose the graph $G$ is made from joining graphs $G_1$ and $G_2$, as in Figure~\ref{2cut}. If these graphs are all $4$-point $\phi^4$, $$ \frac{\mathcal{H}(G)}{2} = \frac{\mathcal{H}(G_1)}{2} \frac{\mathcal{H}(G_2)}{2} . $$ \end{proposition}

\noindent All other invariance properties remain conjectural.

Conjecture~\ref{heppconjecture}, and additionally the approximation given prior, suggest that there is a surprising usefulness to the Hepp bound. The fact that the upper bound lacks precision but scales fantastically to a very precise estimate indicates the value of future research. In particular, the Hepp bound is a considerably more easily computed value than the period itself; the Hepp bound can be computed in exponential time, while much of the difficulty in computing the period arises from the absence of an algorithm. In this regard, a number of graphs have known period and Hepp bounds equal to graphs with periods that have not been resolved. It follows immediately from Conjecture~\ref{heppconjecture} that the periods of $P_{8,31}$ and $P_{8,35}$ are conjectured to be equal, and that periods of $P_{8,30}$ and $P_{8,36}$ are conjectured to be equal.

\chapter{The $c_2$ Invariant}
\label{c2}

\section{The $c_2$ invariant and a conjectured connection to the Feynman period}\label{c2introduction}

The $c_2$ was first introduced in \cite{qftoverfq} and further developed in \cite{BrS}. Recall from the introduction that the Kirchhoff polynomial for a graph $G$ is  $$ \Psi_G = \sum_{T \in T_G} \left( \prod_{e \notin T} x_e \right) \in \mathbb{F}[x_1, ... , x_{|E(G)|}],$$ where $x_e$ is the Schwinger parameter for each edge $e \in E(G)$, $T_G$ is the set of spanning trees of $G$, and $\mathbb{F}$ can be any field. We define the \emph{point count} of a function $f$ over finite field $\mathbb{F}_p$, denoted $[f]_p$, to be the number of roots of $f$ in $\mathbb{F}_p$. Let $\mathcal{P}$ be the increasing sequence of all prime integers. Let $G$ be a graph such that $|V(G)| \geq 3$. It can be shown that $[\Psi_G]_p$ is divisible by $p^2$ (see \cite{qftoverfq}). We define the \emph{$c_2$ invariant} for $G$ at prime $p$ as $c_2(G)_p \equiv \frac{[\Psi_G]_p}{p^2} \pmod{p}$, and the $c_2$ invariant of $G$ as the sequence of residues $c_2(G) = \left( c_2(G)_p \right)_{p \in \mathcal{P}}$.

\begin{example} Consider the graph $K_3$. As any two edges form a spanning tree, $\Psi_{K_3} = x_1 + x_2 + x_3$. For prime $p$, any values $x_1, x_2 \in \mathbb{F}_p$ will force a unique value for $x_3$ as a solution to $\Psi_{K_3} =0$ in $\mathbb{F}_p$. Hence, $[\Psi_{K_3}]_p = p^2$, and the $c_2$ invariant for $K_3$ is equal to $1$ for all primes. \end{example}

It is known that the Kirchhoff polynomial for a graph $G$ can also be represented as the determinant of a matrix. Applying an arbitrary orientation to the edges of $G$, let $M_G^* = [m_{i,j}]$ be a signed incidence matrix of $G$, where rows are indexed by vertices and columns are indexed by edges; $$ m_{v,e} =  \begin{cases} 1, & \text{ if } h(e)=v  \\ -1, & \text{ if } t(e)=v \\ 0, & \text{ otherwise}  \end{cases} ,$$ where $h(e)$ is the head of edge $e$ and $t(e)$ is the tail.  Some authors use opposite signs for entries, but this is ultimately arbitrary in the construction. We create the matrix $M_G$ from $M^*_G$ by deleting a row associated to an arbitrary vertex; call this the \emph{reduced signed incidence matrix}. Let $A$ be the diagonal matrix with entries $x_e$ for $e \in E(G)$, edges ordered to align with $M_G$. Define the \emph{modified Laplacian matrix} to be $$K_G = \left[ \begin{array}{c|c} A & M_G^T \\ \hline -M_G & {\bf 0} \end{array} \right].$$ 

\begin{proposition}\label{poly=matrix} With notation as defined prior, $\det (K_G) = \Psi_G$, regardless of choice of orientation or row deleted in constructing $M_G$. \end{proposition}

The following lemma was first presented by Kirchhoff in \cite{ueber}. The proof of Proposition~\ref{poly=matrix} that follows is from \cite{Brbig}. 

\begin{lemma}[\cite{ueber}]\label{kirkslemma} Let $G$ be a graph and $I \subseteq E(G)$ such that $|I| = h_G$, the loop number of $G$. Let $M_G(I)$ denote the square matrix obtained from the previously defined reduced signed incidence matrix $M_G$ by deleting the columns of $M_G$ indexed by elements of $E(G)-I$. Then, $$ \det (M_G(I)) = \begin{cases} \pm 1 \text{ if } I \text{ is a spanning tree of } G \\ 0 \text{ otherwise}   \end{cases}  .$$  \end{lemma}

We further require, for the proof of Proposition~\ref{poly=matrix}, the Leibniz formula for the determinant of an $n \times n$ matrix $A = [a_{i,j}]$, $$\det (A) = \sum_{\sigma \in S_n} \text{sgn} (\sigma) \prod_{i = 1}^n a_{i,\sigma(i)},$$ where the sum is over all elements of the symmetric group $S_n$, and $\text{sgn}(\sigma)$ is the signature of $\sigma$.

\begin{proof}[Proof of Proposition~\ref{poly=matrix}] From the shape of the modified Laplacian matrix $K_G$, and using the Leibniz formula for the determinant, $$ \det (K_G) = \sum_{I \subseteq E(G)} \prod_{i \notin I} x_i \det \left[ \begin{array}{cc} \bf{0} & M_G^T(I) \\ -M_G(I) & \bf{0} \end{array} \right] = \sum_{\substack{I \subseteq E(G) \\ |I|= h_G}} \prod_{i \notin I} x_i \det(M_G(I))^2.$$ The restriction in the summation comes from the fact that if $|I| \neq h_G$, the columns of the matrix are not independent, and therefore the determinant vanishes. From Lemma~\ref{kirkslemma}, $\det(M_G(I))^2 =1$ if $I$ is a spanning tree and zero otherwise. Therefore, $\det (K_G) = \sum_{T \in T_G} \left( \prod_{e \notin T} x_e \right)  = \Psi_G.$  \end{proof}

\begin{example} One possible reduced signed incidence matrix for the graph $K_3$ is $$M_{K_3} = \left[ \begin{array}{ccc} 1 & 1 & 0 \\ -1 & 0 & 1 \end{array} \right] .$$ A modified Laplacian for this graph is therefore $$ K_{K_3} = \left[ \begin{array}{ccc|cc} x_1 & 0 & 0 & 1 & -1 \\ 0 & x_2 & 0 & 1 & 0 \\  0 & 0 & x_3 & 0 & 1 \\ \hline -1 & -1 & 0 & 0 & 0 \\ 1 & 0 & -1 & 0 & 0 \end{array}  \right]  .$$ The determinant of this matrix is $x_1 + x_2 + x_3$, which agrees with the Kirchhoff polynomial. \end{example}

As stated in the introduction, it is believed that the $c_2$ invariant is preserved by all operations known to preserve the period. Conjecture~\ref{c2doingitsthing} then follows.

\begin{conjecture}[Conjecture 5 in \cite{BrS}] \label{c2doingitsthing}If two $4$-point $\phi^4$ graphs have equal periods, then they have equal $c_2$ invariants. \end{conjecture}

It remains an open problem that the $c_2$ invariant is preserved by completion followed by decompletion and the Schnetz twist. Duality is established in \cite{Dor}.

\begin{theorem}[Theorem 39 in \cite{Dor}] If planar graph $G$ and its dual $G^*$ are both $4$-point $\phi^4$ graphs, then $c_2(G) = c_2(G^*)$. \end{theorem}

Lastly, $4$-point $\phi^4$ graphs with two-vertex cuts do not have a product property, but two graphs constructed using this method will have equal $c_2$ invariants.

\begin{proposition}[Proposition 16 in \cite{BSModForms}] \label{2cutsinc2}Let $G$ be a graph with a $2$-vertex cut. Suppose we may split the graph into two minors $G_1$ and $G_2$ across this cut, as in Figure~\ref{2cut}. If $G$, $G_1$, and $G_2$ are all $4$-point $\phi^4$ graphs, then $c_2(G)_p \equiv 0 \pmod{p}$ for all primes $p$.  \end{proposition}

\noindent Similarly, subdivergences also result in the $c_2$ invariant vanishing.

\begin{theorem}[Theorem 38 in \cite{propsc2}] \label{subdivc2} If a graph $G$ contains a subdivergence, then $c_2(G)_p \equiv 0 \pmod{p}$ for all primes $p$. \end{theorem}

\section{Graphic interpretation of the $c_2$ invariant}\label{graphicc2}

In my attempts to work with the $c_2$ invariant, especially with the hope of proving that two graphs that differ by completion followed by decompletion had equal $c_2$ invariants, I believed it would be useful to work with a more structural interpretation of the invariant. In this section we construct this interpretation. This is original work.

Throughout this section, assume that $p$ is prime. Let $G$ be a connected graph. By Proposition~\ref{poly=matrix}, we may consider a zero of the Kirchhoff polynomial of $G$ over $\mathbb{F}_p$ as an assignment of values to the Schwinger parameters such that the modified Laplacian matrix has determinant zero. Hence, this is an assignment of values to the Schwinger parameters such that there is a nontrivial linear combination of the rows of $K_G$ that sums to the zero vector over $\mathbb{F}_p$. Fix such a linear combination, and let $w_e$ be the coefficient given to the row indexed by edge $e \in E(G)$. Since $G$ is connected and a row of $M_G$ was deleted in the construction of $K_G$, the rows of $M_G$ are linearly independent, and hence at least one row in the first $|E(G)|$ must receive a nonzero coefficient in this linear combination. Whence, we may consider these coefficients on the first $|E(G)|$ rows as weights on the edges of $G$. Specifically, consider the signed incidence matrix $M_G^T$ in the upper right block. Denoting the head and tail of edge $e$ and $h(e)$ and $t(e)$, respectively,  $$\sum_{\substack{e \in E(G)\\h(e) = v}} w_e - \sum_{\substack{e \in E(G)\\t(e) = v}} w_e = 0$$ for all vertices $v \in V(G)$. 

\begin{definition} Let $G$ be a directed graph and $\mathcal{G}$ an abelian group. Let $\phi: E(G) \rightarrow \mathcal{G}$ be a weighting of the edges. We define the associated \emph{boundary function} $\partial \phi: V(G) \rightarrow \mathcal{G}$ by $$(\partial_v)\phi = \sum_{\substack{e \in E(G)\\h(e) = v}} \phi(e) - \sum_{\substack{e \in E(G)\\t(e) = v}} \phi(e). $$ We say that $\phi$ is a \emph{$\mathcal{G}$ flow} if, for all $v \in V(G)$,  $(\partial_v) \phi = 0$ in $\mathcal{G}$. \end{definition}

\noindent Hence, a linear combination of the matrix rows that sums to the zero vector can be thought of as assigning weights to the edges of the associated graph that creates a non-trivial $\mathbb{F}_p$ flow. As stated in the introduction, this is comparable to momentum conservation in Feynman graphs. See \cite{genflows} for a more general introduction to flows.

\begin{example} Consider the graph $K_4$, with orientation as drawn below. The modified Laplacian is included with this graph, vertex labels are assigned arbitrarily and edges are ordered lexicographically. $$\raisebox{-0.4\height}{\includegraphics[scale=0.8]{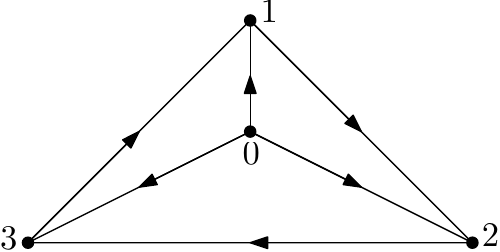}}, \hspace{1cm} K_{K_4} = \left[ \begin{array}{cccccc|ccc} x_1&&&&&&1&0&0 \\ &x_2&&&&&0&1&0 \\ &&x_3&&&&0&0&1 \\ &&&x_4&&&-1&1&0 \\ &&&&x_5&&1&0&-1 \\ &&&&&x_6&0&-1&1 \\ \hline -1&0&0&1&-1&0&&& \\ 0&-1&0&-1&0&1&&& \\ 0&0&-1&0&1&-1&&& \end{array} \right] $$ Considering only the upper right block, we may create a linear combination in $\mathbb{F}_5$ that sums to the zero vector as  $$\begin{array}{c} 4\\3\\3\\2\\3\\0   \end{array} \left[ \begin{array}{ccc} 1&0&0 \\ 0&1&0 \\ 0&0&1 \\ -1&1&0 \\ 1&0&-1 \\ 0&-1&1 \end{array} \right].$$ Treating these coefficients as weights on the edges, this linear combination does indeed translate to an $\mathbb{F}_5$ flow on the graph, $$\raisebox{-0.4\height}{\includegraphics{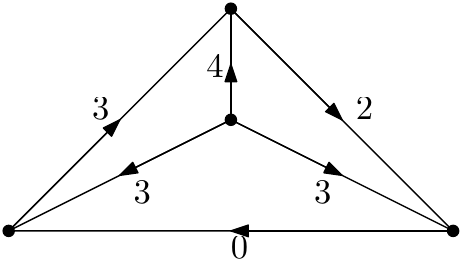}} .$$ \end{example}

Again, fix a linear combination of the rows of a modified Laplacian matrix that sum to the zero vector over $\mathbb{F}_p$. Consider now the coefficients on the last $|V(G)|-1$ rows in this linear combination. These may similarly be used to describe weights on the vertices; for $v\in V(G)$ denote this weight $w_v$, and again let $w_e$ be the weight on edge $e \in E(G)$. Then, for edge $e = (a,b) \in E(G)$, the Schwinger parameter for the edge must be a value that balances the equation $$w_ex_e + w_a - w_b = 0,$$ where $w_v =0$ for the vertex that was deleted in the creation of $M_G$. Consider traveling around a cycle in the underlying unoriented graph; adding $w_ex_e + w_a - w_b$ when moving with the underlying orientation of the edges and subtracting when not. This will necessarily sum to zero. As a result of the orientation determining whether each term is added or subtracted, all vertex weights will cancel. 

\begin{definition} Let $G$ be an oriented graph, and $\mathcal{G}$ an abelian group. Let $\tau :E(G) \rightarrow \mathcal{G}$ be a weighting of the edges, and $\mathcal{C}$ the collection of cycles in the underlying unoriented graph. Let $C \in \mathcal{C}$ be arbitrary. Traveling around $C$, if the orientation of an edge $e$ agrees with the direction of travel we put $e \in C_+$, and $e \in C_-$ otherwise. The \emph{boundary function} of cycle $C$ given $\tau$ is $$(\partial_C)\tau = \sum_{e \in C_+}\tau(e) - \sum_{e \in C_-}\tau(e). $$ We say that $\tau$ is a \emph{$\mathcal{G}$-tension} if for all $C \in \mathcal{C}$, $(\partial_C) \tau = 0$. \end{definition}

See, for example, \cite{ten1} and \cite{ten2}.

We see then that the Schwinger parameters are values that, multiplying edge-by-edge, turn a non-trivial $\mathbb{F}_p$ flow into a $\mathbb{F}_p$ tension. In the other direction, consider an $\mathbb{F}_p$ flow $\phi:E(G) \rightarrow \mathbb{F}_p$ and an assignment of Schwinger parameters $\mathbf{x}:E(G) \rightarrow \mathbb{F}_p$ such that $e\mapsto \mathbf{x}(e) \phi(e)$ is an $\mathbb{F}_p$ tension. We may create coefficients for the rows indexed by vertices uniquely by equation $w_ex_e + w_a - w_b$ and the fact that we may think of the row removed in creating the matrix $M_G$ as assigning weight zero to that vertex. This then defines the coefficients of a linear combination of the row vectors of this modified Laplacian that sums to the zero vector. Hence, this matrix has determinant zero.  Returning to this product on individual edges of flows and Schwinger parameters, a map of Schwinger parameters $\mathbf{x}:E(G) \rightarrow \mathbb{F}_p$ is a zero in $\Psi_G$ if and only if there exists a non-trivial $\mathbb{F}_p$ flow $\phi: E(G) \rightarrow \mathbb{F}_p$ such that $e \mapsto \mathbf{x}(e) \phi(e)$ is a $\mathbb{F}_p$ tension. Note that by this construction, we are able to ignore the vertex weights and may consider only the edges. We will call such maps \emph{Schwinger solutions} (or \emph{Schwinger solutions to flow $\phi$}).

\begin{example} Consider the graph and $\mathbb{F}_5$ flow established in the previous example; $$\raisebox{-0.4\height}{\includegraphics{c2flow}} .$$ We consider creating Schwinger parameters such that for each edge, the product of the Schwinger parameter and the weight of the flow creates a $\mathbb{F}_5$-tension in the graph. Such a collection of Schwinger parameters is given in parentheses below, including the original flow and the product for each edge in $\mathbb{F}_5$. The variable $x$ may be any value in $\mathbb{F}_5$, as the edge has weight zero in the flow. $$\raisebox{-0.4\height}{\includegraphics[scale=0.8]{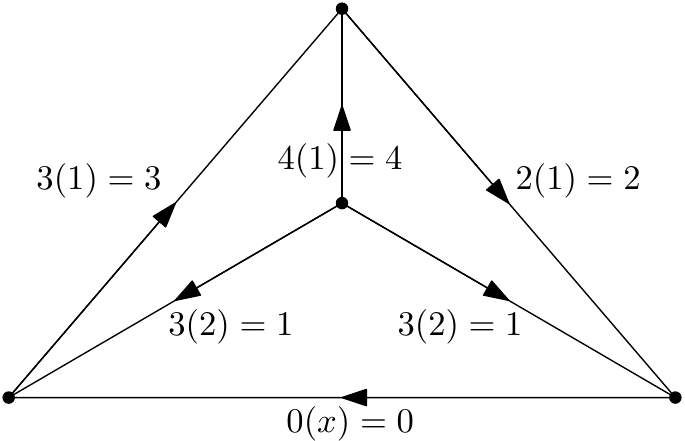}} $$  A quick check reveals that this is an $\mathbb{F}_5$ tension.

Further, we may check that this variable assignment to Schwinger parameters does create a zero of $\Psi_{K_4}$ in $\mathbb{F}_5$. That is, \begin{align*} \Psi_{K_4} &= x_4x_5x_6 + x_1x_3x_4 + x_2x_3x_6 + x_1x_2x_5 + x_1x_3x_5 + x_1x_3x_6 + x_2x_3x_5 + x_2x_3x_4 \\ &\hspace{1cm}  + x_1x_2x_4 + x_1x_2x_6 + x_2x_4x_6 + x_2x_4x_5 + x_1x_4x_6 + x_1x_5x_6 + x_3x_5x_6 + x_3x_4x_5 \\ &= 1 \cdot 1 \cdot x  + 1 \cdot 2 \cdot 1 + 2 \cdot 2 \cdot x  + 1 \cdot 2 \cdot 1 + 1 \cdot 2 \cdot 1  + 1 \cdot 2 \cdot x + 2 \cdot 2 \cdot 1 + 2 \cdot 2 \cdot 1  \\ &\hspace{1cm}  + 1 \cdot 2 \cdot 1  + 1 \cdot 2 \cdot x  + 2 \cdot 1 \cdot x  + 2 \cdot 1 \cdot 1 + 1 \cdot 1 \cdot x  + 1 \cdot 1 \cdot x  + 2 \cdot 1 \cdot x  + 2 \cdot 1 \cdot 1 \\ 
&= 0 .\end{align*} \end{example}

The following is a well-known result regarding the number of $\mathbb{F}_p$ flows in a graph.

\begin{proposition}\label{numberflows} A connected graph $G$ has $p^{|E(G)| - (|V(G)|-1)} = p^{h_1(G)}$ $\mathbb{F}_p$ flows. \end{proposition}

\begin{proof} Let $T$ be the edges of a spanning tree of $G$. Assign edges in $E(G)-T$ values from $\mathbb{F}_p$ arbitrarily. We will show that edges in $T$ can be assigned weights back uniquely.

Any vertex incident to precisely one edge that has not received a value may have that value assigned uniquely to maintain the flow property. Removing edges from $T$ as values are assigned, we may give values to the leaves of $T$ in this manner until only a single edge has no weight assigned, call it $e = (a,b)$. At this point, two vertices are incident to a single edge with no assigned Schwinger value, and we must check that the desired assignments agree. Summing over edges in $G-e$ with assigned weights, say $\phi: E(G-e) \rightarrow \mathbb{F}_p$, we know that $\sum_{v \in V(G)} (\partial_v)\phi = 0$, as each weight is added once as an head and subtracted once as a tail. As all other vertices $v$ have $(\partial_v)\phi=0$, $(\partial_a)\phi = - (\partial_b)\phi$, and a weight can be added uniquely to $e$, completing the flow in $G$. \end{proof}

A similar method may be used to count the number of Schwinger solutions for a fixed modular $\mathbb{F}_p$ flow. The following lemma and corollary will help us prove this property.

\begin{lemma}\label{latetothegame} Let $G$ be a graph. The sequence of arcs of a closed walk in $G$ can be decomposed into a collection of (possibly trivial) cycles that maintain the direction created by the walk and a set of edges which are traversed an equal number of times in opposite directions. \end{lemma}

\begin{proof} For a closed walk on graph $G$, create directed multigraph $G'$ on the same vertex set by adding a directed edge to $G'$ for each edge in the closed walk in $G$ following the direction of the walk. By construction, every vertex has equal in- and out-degree. Removing a maximal collection of directed cycles containing disjoint edges from $G$ must leave nothing but edges in pairs $(a,b)$ and $(b,a)$, which completes the proof. \end{proof}

\begin{corollary}\label{cyclesinoppositedirections} Let $G$ be a directed graph and $w:E(G) \rightarrow \mathbb{F}_p$ an edge weighting for some prime $p$. Define the weight of any walk in the underlying graph as the sum of the weights of edges in the walk that agree with the orientation, minus the sum of weights of edges that do not. If every cycle has weight zero in $\mathbb{F}_p$, then so does every closed walk. \end{corollary}

\begin{proof} This follows immediately from Lemma~\ref{latetothegame}, as the decomposition presented therein necessarily sums to zero. \end{proof}

\begin{proposition}\label{schwingercount} Fix an $\mathbb{F}_p$ flow for a graph $G$. Suppose a connected spanning subgraph that contains all edges of weight zero in the flow must have at least $k$ edges. There are $p^k$ Schwinger solutions to this flow. \end{proposition}

\begin{proof} Assign values to the Schwinger parameters of these $k$ edges arbitrarily. Treat tension weights on edges as the product of the Schwinger parameters and the values assigned by this fixed flow. By construction, all Schwinger parameters that have not received a value must correspond to nonzero weights in the flow. Every edge that has not received a Schwinger parameter value will necessarily create a cycle with the edges that have. As such, its tension weight is unique, and so the Schwinger parameter is unique.

We will prove that it is possible to assign values to the remaining Schwinger parameters to create a tension by induction. The base case -- the initial assignment of Schwinger values to the $k$ edges -- cannot create any cycles that break the tension property, as all cycles must be made entirely of weight zero edges. Suppose then that at some stage of the process the assignment of a values to the Schwinger parameters has created only cycles that have the tension property. Consider the next assignment of a Schwinger parameter value. If only a single new cycle will be completed, the assignment is immediate, so we suppose multiple cycles will be completed. Suppose the edge whose Schwinger parameter is being assigned a value is $\{a,b\}$. Then, these two cycles currently create two $(a,b)$-paths in the graph on edges that have received Schwinger parameter values. We may therefore create a closed walk in the graph by walking the edges of one $(a,b)$-path followed by the edges of the other traversed in reverse. By Corollary~\ref{cyclesinoppositedirections}, this walk may be decomposed into cycles and edges traversed an equal number of times in opposite directions. By induction then, all cycles sum to zero, and the remaining edges cancel. Hence, this closed walk must also sum to zero. Therefore, these two $(a,b)$-paths sum to equal values in $\mathbb{F}_p$, and the Schwinger value can be uniquely assigned to $\{a,b\}$. By induction, this completes the proof. \end{proof}

Similarly, we may count the number of flows that have a particular mapping of the Schwinger parameters as a Schwinger solution.

\begin{proposition} Let $G$ be a graph. Fix an assignment of values in $\mathbb{F}_p$ to the Schwinger parameters. The set of $\mathbb{F}_p$ flows for which this assignment forms a Schwinger solution forms a vector space. \end{proposition}

\begin{proof} Let $\phi_1$ and $\phi_2$ be modular $p$ flows and $x \in \mathbb{F}_p^{|E(G)|}$ a Schwinger solution to both flows. Let $c$ be a constant in $\mathbb{F}_p$. Summing around cycle $C$, $$ (\partial_C)(\phi_1 + \phi_2) = \sum_{e \in C_+}(\phi_1(e) + \phi_2(e))x_e - \sum_{e \in C_-}(\phi_1(e) + \phi_2(e))x_e = (\partial_C)\phi_1 + (\partial_C)\phi_2 =0   $$ and $$(\partial_C)(c\phi_1) =  \sum_{e \in C_+}c\phi_1(e)x_e - \sum_{e \in C_-}\phi_1(e)x_e = c(\partial_C)\phi_1 =0 .$$ Hence, the set of flows for which a fixed assignment of values to the Schwinger parameters is a Schwinger solution is a vector space. \end{proof}

\begin{proposition} Let $G$ be a graph with some arbitrary orientation on its edges. Create matrix $K$ by assigning numeric values in $\mathbb{F}_p$ to variables in the modified Laplacian matrix $K_G$. The nullspace of $K$ uniquely defines the set of $\mathbb{F}_p$ flows for which this assignment of edge weights is a Schwinger solution, and the number of such $\mathbb{F}_p$ flows is $p^{\text{Nullity}(K)}$. \end{proposition}

\begin{proof} Let $\mathbf{x}$ be an element of the nullspace, so $K\mathbf{x} = \mathbf{0}$. We may, as before, think of the first $|E(G)|$ entries of $\mathbf{x}$ as weights on the edges, this time the result of the block $-M_G$. This must define a $\mathbb{F}_p$ flow, this time in the graph made from $G$ by reversing the direction of every edge. This weighting must then also give a $\mathbb{F}_p$ flow in $G$. To show that these vectors give unique flows, note that if $e = (a,b) \in E(G)$, and the elements in $\mathbf{x}$ corresponding to columns indexed by $e$ and $a$ are assigned, then the value corresponding to $b$ is uniquely determined. Then, as we may thing of the row deletion in the creation of $K_G$ as assigning that vertex weight $0$, this uniqueness propagates out from this vertex. Hence, if two vectors $\mathbf{x}$ and $\mathbf{y}$ are in the nullspace and correspond to the same flow, it must be the case that $\mathbf{x} = \mathbf{y}$. The number of flows is therefore immediate from the dimension of the nullspace. \end{proof}

Note that the previous two propositions allow for the trivial flow, which of course would be a solution to all Schwinger values.

\begin{proposition} Let $\phi_1, ... , \phi_k$ be $\mathbb{F}_p$ flows for a graph $G$. The intersection of sets of Schwinger solutions of these $\mathbb{F}_p$ flows has cardinality a power of $p$, and is a subspace of $V(|E(G)|,\mathbb{F}_p)$. \end{proposition}

\begin{proof} For each $\mathbb{F}_p$ flow and each cycle in $G$, create an equation as follows. For cycle $C = e_1,e_2,...,e_n$, add equation $$(-1)^{c_{e_1}}w_{j,1}x_1 + (-1)^{c_{e_2}}w_{j,2}x_2 + \cdots + (-1)^{c_{e_n}}w_{j,n}x_n = 0,$$ where the $c_{e_i}=0$ if the edge is oriented in the direction the cycle is traversed and $c_{e_i}=1$ otherwise, and $\phi_j(e_i) = w_{j,i}$ are from flow $\psi_j$. Simultaneously solving for all cycles as a system of equations, if $k$ is the number of free variables, then there are $p^k$ solutions.  \end{proof}

The goal with the preceding work was to find a structural interpretation of the $c_2$ invariant, possibly for computational simplicity, though ideally to prove that the $c_2$ invariant is unchanged by completion followed by decompletion. Ultimately, inclusion-exclusion produced no useful results.

We lastly explore some potentially useful results regarding the open problems for the $c_2$ invariant. By Proposition~\ref{2cutsinc2} and assuming decompletion invariance per Conjecture~\ref{c2doingitsthing}, if a completed graph $G$ has connectivity $3$, it has $c_2$ invariant zero for all primes. Hence, we may restrict our investigation to graphs with $4$-connected completion.

\begin{Menger}[Edge-connectivity version \cite{mengersthm}] A graph $G$ is $k$-edge-connected if and only if every pair of distinct vertices is connected by $k$ edge-disjoint paths. \end{Menger}

By Theorem~\ref{numberflows}, the fact that the number of $\mathbb{F}_p$ flows is completion invariant is trivial. It is interesting that there is a natural bijection between $\mathbb{F}_p$ flows for these decompletions that follows from Menger's Theorem, as we may assume that the graph has no subdivergences, and hence no $4$-edge cuts.

Let $v$ and $w$ be distinct vertices in the completed graph. By Menger's Theorem, we may fix four edge-disjoint paths between these vertices. Consider one such path and a $\mathbb{F}_p$ flow in $G-v$. Adding vertex $v$ back to the graph and assigning all incident edges weight zero, we then preserve the flow structure of the graph. If the edge $e$ on this fixed path incident to $w$ has weight $w_e$, we may then subtract $w_e$ from this edges weight, and all weights on edges on this path oriented along the path in the same direction as $e$, and add it to all other edges. By construction, vertices along this path still have the flow property, and edges incident to $v$ may now have nonzero weight. If we do this for all paths, a quick check of the four possible cases of pairs of oriented edges incident to $w$ and $v$ on these paths reveals that vertex $v$ also has the flow property. This therefore creates a flow where all edges incident to $w$ have weight zero. As such, this can be turned into a $\mathbb{F}_p$ flow in $G-w$.  Clearly, this operation is bijective. Unfortunately, the change in the distribution of cycles and additionally of edges of weight zero in the flow makes this unlikely to be of use in proving decompletion invariance for the $c_2$ invariant.

A similar method will be of use in later work, in particular Theorem~\ref{egpcompletion}.

Consider now an oriented planar graph. With the orientation, we may distinguish between the sides of an edge, and hence in creating the planar dual may canonically orient the edges of the dual, for example by demanding all edges in the dual travel from the right to the left of the edge in the original graph. Note that with this method of creating the dual, $(G^*)^* \neq G$, as the orientation of all edges is reversed. For our purposes, though, this is not a hindrance.

\begin{lemma}\label{runningoutoflabelideas} Let $G$ be an oriented planar graph and suppose $\tau:E(G) \rightarrow \mathbb{F}_p$. If every facial cycle $C$ has $(\delta_C)\tau =0$, then $\tau$ is a $\mathbb{F}_p$ tension.  \end{lemma}

\begin{proof} As the bounded facial cycles of a planar graph can be used as a basis for the cycle space, the proof of this is similar to that of Proposition~\ref{schwingercount}; an arbitrary cycle $C$ can be represented as a binary sum of facial cycles. Traversing all of these facial cycles clockwise, then, this sum is zero, as all facial cycles have the tension property, and all edges in a facial cycle but not in $C$ are traversed twice, but in opposite directions, and hence cancel. \end{proof}

\begin{theorem} Let $G$ be a planar graph, and fix a planar orientation. Suppose $\phi:E(G) \rightarrow \mathbb{F}_p$ is an $\mathbb{F}_p$ flow. Preserving edge weights from $\phi$ and using the above canonical orientation of the dual, $\phi$ defines an $\mathbb{F}_p$ tension in $G^*$.  \end{theorem}

\begin{proof} By construction, a facial cycle $C$ in the dual comes from the edges incident to a single vertex $v$ in the original graph. Therefore, we may travel around this cycle in a direction such that $\{e \in E(G):t(e)=v\} = C_+$ and $\{e\in E(G):h(e)=v\} = C_-.$ Immediately, this cycle has the tension property. By Lemma~\ref{runningoutoflabelideas} then, this is a $\mathbb{F}_p$ tension. \end{proof}

Another graph theoretic method for examining the $c_2$ invariant will be discussed later. Additional tools required to make use of this will be introduced in subsequent chapters, and this alternate method will be introduced in Chapter~\ref{conclusion}.

\chapter{A brief introduction to matroid theory}
\label{chmatroid}

As some matroid tools will be of use, we include here a short introduction to matroid theory. Notational conventions are taken from \cite{Oxl}.

First presented in \cite{whitneymatroid}, a \emph{matroid} is a pair $(E, \mathcal{I})$ such that $E$ is a finite set and $\mathcal{I}$ is a collection of subsets of $E$ with the following properties:

\begin{itemize}
\item $\emptyset \in \mathcal{I}$,
\item if $I \in \mathcal{I}$ and $I' \subseteq I$, then $I' \in \mathcal{I}$, and
\item if $I_1, I_2 \in \mathcal{I}$ and $|I_1|< |I_2|$, then there is an element $e \in I_2 - I_1$ such that $e \cup I_1 \in \mathcal{I}$.
\end{itemize}

\noindent We call $E$ the ground set. An element of $\mathcal{I}$ is said to be independent, and any other subset of $E$ is dependent. A minimally dependent set is a circuit. We may provide an alternate axiom set based on the collection of circuits, $\mathcal{C}$. For ground set $E$ and collection of subset $\mathcal{C}$ of $E$, $\mathcal{C}$ is the set of circuits of a matroid on $E$ if and only if (Corollary 1.1.5 in \cite{Oxl});

\begin{itemize}
\item $\emptyset \notin \mathcal{C}$,
\item if $C_1, C_2 \in \mathcal{C}$ and $C_1 \subseteq C_2$, then $C_1 = C_2$, and
\item if $C_1, C_2 \in \mathcal{C}$, $C_1 \neq C_2$, $e \in C_1 \cap C_2$, then there exists a $C_3 \in \mathcal{C}$ such that $C_3 \subseteq (C_1 \cup C_2) -e$.
\end{itemize}

From these, we may present two fundamental classes of matroids, constructed from matrices and graphs.

\begin{proposition}[Proposition 1.1.1 in \cite{Oxl}] Let $E$ be the set of columns of an $m \times n$ matrix over a field $\mathbb{F}$, and $\mathcal{I}$ the set of subsets of $E$ that are linearly independent in the vector space $V(m,\mathbb{F})$.  Then, $(E, \mathcal{I})$ is a matroid (the \emph{vector matroid}).   \end{proposition}

\begin{proposition}[Proposition 1.1.7 in \cite{Oxl}] Let $G$ be a graph, $E$ the set of edges of $G$, and $\mathcal{C}$ the collection of edge sets of cycles of $G$. Then $\mathcal{C}$ is the set of circuits of a matroid on $E$ (the \emph{cycle matroid} of $G$).  \end{proposition}

It is of particular interest that the cycle matroid of a graph $G$ can be represented as a vector matroid over any field using the signed incidence matrix seen in Chapter~\ref{c2}. A matroid that can be represented as a vector matroid over every field is a \emph{regular matroid}. Seymour proved in \cite{moreseymour} that every regular matroid can be constructed from graphic matroids (matroids that may be represented as the cycle matroid of a graph), cographic matroids (the matroidal dual of graphic matroids, to be introduced shortly), and the matroid $R_{10}$ (which will be introduced in Section~\ref{compr10}), using a particular piecing operation.

Another alternate description of a matroid comes from \emph{bases}, maximal independent sets. Let $E$ be a set and $\mathcal{B}$ a set of subsets of $E$. Then $\mathcal{B}$ is a collection of bases of a matroid on $E$ if and only if (Corollary 1.2.5 in \cite{Oxl});
\begin{itemize}
\item $\mathcal{B}$ is non-empty, and 
\item if $B_1, B_2 \in \mathcal{B}$ and $x \in B_1 - B_2$, then there is an element $y \in B_2-B_1$ such that $(B_1-x) \cup y \in \mathcal{B}$.
\end{itemize}
The bases for a vector matroid are the bases for the column space of that matrix, and the bases for a cycle matroid are the sets of edges that induce spanning trees in the graph.

With this, we may develop a notion of duality in matroids.

\begin{theorem}[Theorem 2.1.1 in \cite{Oxl}] Let $M$ be a matroid with bases $\mathcal{B}$ and $\mathcal{B}^* = \{ E(M) - B : B \in \mathcal{B} \}$. Then $\mathcal{B}^*$ is the set of bases of a matroid on $E(M)$ (the \emph{dual} of $M$, $M^*$). \end{theorem}

Suppose $M_1$ and $M_2$ are two matroids with ground sets $E(M_1)$ and $E(M_2)$, respectively. We say that $M_1$ and $M_2$ are isomorphic, $M_1 \cong M_2$, if there is a bijective mapping from $E(M_1)$ to $E(M_2)$ that preserves independence.

\begin{theorem}[Theorem 2.3.4 in \cite{Oxl}]\label{matdual} Suppose $G$ is a planar graph, and $G^*$ is the normal graph theoretic dual of $G$. The matroidal dual of $M(G)$, the cycle matroid of $G$, is isomorphic to the cycle matroid of $G^*$; $$ M(G^*) \cong M^*(G) .$$ \end{theorem}

There is a structural characterization of when two graphs have isomorphic cycle matroids. Suppose graph $G$ is disconnected, and $v_1,v_2 \in V(G)$ are vertices of $G$ in different connected components. \emph{Vertex identification} identifies these two vertices. Vertex cleaving is the opposite of this, splitting a connected component at a cut vertex. Both operations are shown in Figure~\ref{ident}. 

\begin{figure}[h]
  \centering
      \includegraphics[scale=1.0]{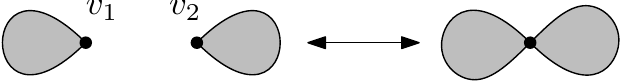}
  \caption{An example of vertex cleaving and vertex identification.}
\label{ident}
\end{figure}

The Whitney flip is an operation on a $2$-vertex cut. Suppose we have graphs $G_1$ and $G_2$, $v_1,v_2 \in V(G_1)$ and $w_1,w_2 \in V(G_2)$. Suppose we create graph $G$ by identifying $v_1$ with $w_1$ and $v_2$ with $w_2$, and we create graph $G'$ by identifying $v_1$ with $w_2$ and $v_2$ with $w_1$. Then, we say that $G$ and $G'$ differ by a Whitney flip. This is shown in Figure~\ref{flipnottwist}.

\begin{figure}[h]
  \centering
      \includegraphics[scale=1.0]{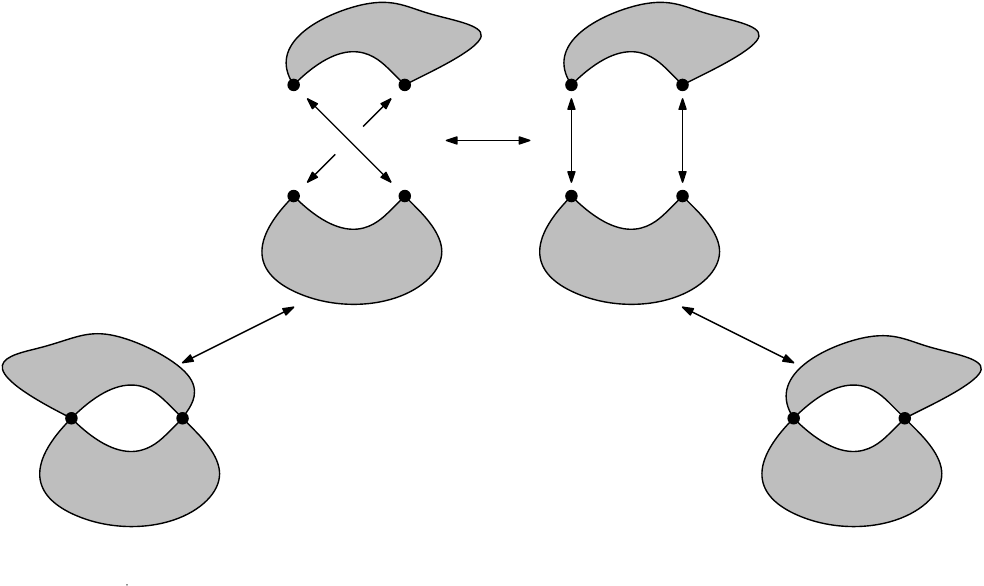}
  \caption{The Whitney flip.}
\label{flipnottwist}
\end{figure}

\newtheorem*{whitney2}{Whitney's $2$-isomorphism theorem}

\begin{whitney2}[\cite{whitneyisomorphism}] Suppose $G_1$ and $G_2$ are two graphs with no isolated vertices. The cycle matroids of $G_1$ and $G_2$ are isomorphic if and only if the graphs $G_1$ and $G_2$ differ by some sequence of vertex cleaving, vertex identification, and Whitney flip operations. \end{whitney2}

Of particular relevance, the Feynman integral is known to be unchanged by these graph operations  (see \cite{bogner}).

\chapter{The Extended Graph Permanent}
\label{Extended}

We move now to the extended graph permanent, which will be the focus of the remainder of this thesis. Content in the following chapters has appeared prior in collaboration with DeVos and Yeats in \cite{pppoaiftposim}, and I was able to expand upon this in \cite{potegp}. The creation of extended graph permanent was motivated by the work of Alon, Linial, and Meshulam in~\cite{AlLiMe}. This chapter introduces this invariant, though first we must introduce some preliminary tools.

\section{Properties of the matrix permanent}\label{egppermanent}

 For notational convenience we will use the \emph{Kronecker product} to construct block matrices. For matrices $A=[a_{i,j}]$ and $B$, $$A \otimes B = \left[ \begin{array}{ccc} a_{1,1}B & a_{1,2} B & \cdots \\ a_{2,1} B & a_{2,2} B & \cdots \\ \vdots && \end{array} \right].$$ We will denote the $n \times m$ matrix with all entries $t$ by $\mathbf{t}_{n \times m}$, or simply $\mathbf{t}_n$ if it is an $n \times n$ square. We denote the $n \times n$ identity matrix as $I_n$, or $I$ if dimension is clear from context.

\begin{definition}\label{permanentdef} Let $A=[a_{i,j}]$ be an $n$-by-$n$ matrix. The \emph{permanent} of $A$ is $$\text{Perm}(A) = \sum_{\sigma \in S_n} \prod_{i=1}^n a_{i,\sigma(i)},$$ where the sum is over all elements of the symmetric group $S_n$. \end{definition}

From this definition of the permanent, we recognize that it is the determinant with no signed factors, comparing to the Leibniz formula for the determinant. In fact, the permanent also can be computed using cofactor expansion, similar to the determinant. For $n \times n$ matrix $M = [m_{x,y}]$, let $N$ be matrix $M$ with row $i$ removed, and $N_j$ the matrix $N$ with column $j$ removed. Cofactor expansion along the $i^\text{th}$ row gives $$\text{Perm}(M) = \sum_{j=1}^n m_{i,j}\text{Perm}(N_j).$$ As $1 \equiv -1 \pmod{2}$, any square matrix $M$ has $\text{Perm}(M) \equiv \det (M) \pmod{2}$, which suggests that the permanent may have some interesting properties modulo integers.

\begin{remark} \label{rowops} From the definition of the permanent, it is clear that we may interchange two rows or columns without affecting the permanent. Further, multiplying a row or column by a constant results in the permanent being multiplied by that constant. \end{remark}

What happens when a multiple of one row is added to another is less clear, and in general not well-behaved. However, there is greater control when using matrices with multiple copies of each row, modulo one more than the number of copies. This will be examined in Proposition~\ref{redlemma} and Corollary~\ref{reduction}.

\begin{lemma}\label{comrow} For an $n \times n$ matrix $M = [m_{x,y}]$, if there is a set $\{a_1,a_2,...,a_k\}$ such that rows $r_{a_1}, r_{a_2}, ..., r_{a_k}$ are equal, there is a factor of $k!$ in the permanent of $M$. \end{lemma}

\begin{proof}We may write \begin{align*} \text{Perm}(M) &= \sum_{\sigma \in S_n} \prod_{i=1}^n m_{i,\sigma(i)} \\ &=k! \sum_{\sigma \in S_n^*} \prod_{i=1}^n m_{i,\sigma(i)}, \end{align*} where $S_n^*$ is the set of elements of the symmetric group such that $\sigma(a_1)<\sigma(a_2)< \cdots < \sigma(a_m)$, and the $m!$ term allows for further permutations of these elements.   \end{proof}

\begin{proposition}\label{redlemma} Let $M$ be a $n \times n$ matrix and $r_i$ and $r_j$ rows of $M$, $r_i \neq r_j$ as vectors. Suppose there are $k$ copies of $r_j$ in $M$.  Let $M'$ be a matrix derived from $M$ by adding a constant integer multiple of $r_j$ to $r_i$. Then $\text{Perm}(M) \equiv \text{Perm}(M') \pmod {k+1}$. \end{proposition}

\begin{proof}   Suppose that $$M = [m_{x,y}] = \left[ \begin{array}{c} r_1 \\ r_2 \\ \vdots  \end{array} \right], \text{ } M' = [{m'}_{x,y}] =  \left[ \begin{array}{c} r_1 \\ \vdots \\ r_i +cr_j  \\ \vdots  \end{array} \right].$$ Define $N$ as the matrix $M$ with row $i$ removed. We will use $N_t$ to denote the matrix $N$ with column $t$ removed. By cofactor expansion along the $i^\text{th}$ row, \begin{align*} \text{Perm}\left(M\right) &= \sum_{t=1}^{n} m_{i,t}  \text{Perm}\left(N_t\right),\\ \text{Perm}\left(M'\right) &= \sum_{t=1}^{n} {m'}_{i,t} \text{Perm}(N_t) \\ &=  \sum_{t=1}^{n}(m_{i,t} + cm_{j,t}) \text{Perm}(N_t) \\ &=  \text{Perm}\left(M\right) + c  \text{Perm} \left[ \begin{array}{c} r_1 \\ \vdots \\ r_{i-1} \\ r_j \\ r_{i+1} \\ \vdots  \end{array} \right]  .\end{align*} As this last matrix has $k+1$ copies of row $r_j$, it has permanent congruent to zero modulo $k+1$ by Lemma~\ref{comrow}.\end{proof}

\begin{corollary}\label{reduction} Suppose $M = \mathbf{1}_{k \times 1} \otimes K$ is a square block matrix for some matrix $K$, and $r_i$ and $r_j$ are rows of $M$ in a common block, $i \neq j$. Let $M'$ be a matrix derived from $M$ by adding a constant integer multiple of $r_j$ to $r_i$ in each block. Then $\text{Perm}(M) \equiv \text{Perm}(M') \pmod {k+1}$. \end{corollary}

This construction -- matrices of the form $\mathbf{1}_{k \times 1} \otimes K$ and the permanent modulo $k+1$ -- will be the foundation of our invariant. Corollary~\ref{reduction} shows that row reduction techniques, applied to each block simultaneously, act on the permanent of this matrix as one would desire.

\begin{corollary}\label{oddcase} For non-prime $k + 1$, the permanent of any matrix $\mathbf{1}_{k \times 1} \otimes K$ where $K$ has at least two rows is zero modulo $k + 1$.   \end{corollary}

\begin{proof} As $K$ has at least two rows, $k!^2$ is a factor in the permanent by Lemma~\ref{comrow}. Factoring $k+1 = ab$ where $a,b>1$, both appear in the product $k!$, and the result follows. \end{proof}

\section{The extended graph permanent}\label{egpegp}

\begin{definition} Let $G$ be a connected graph. Arbitrarily apply directions to the edges in $G$, and  let $M^*_G$ be the signed incidence matrix associated with this digraph; columns indexed by edges and rows by vertices. Select a vertex $v$ in $V(G)$ and delete the row indexed by $v$ in $M_G^*$. Call this new matrix $M_G$, and call $v$ the \emph{special vertex}. Define the \emph{fundamental block matrix of $M_G$} (or \emph{a fundamental block matrix of $G$}, dependent on the orientation and choice of $v$), $\overline{M}_G$, to be the smallest square matrix that can be created using blocks of $M_G$. That is, the smallest values of $m$ and $n$ such that $\mathbf{1}_{m \times n} \otimes M$ is square. Graphs with $|E(G)| = k(|V(G)|-1)$ for some $k \in \mathbb{N}$ will be of particular interest, so we define the \emph{$k$-matrix} of any matrix $M$ to be the block matrix $\mathbf{1}_{k \times 1} \otimes M$. \end{definition}

\begin{example} Consider the graph $K_3$, shown below. We select the marked vertex as the special vertex and orient as indicated. This results in the fundamental matrix $\overline{M}_G$.
$$G = \raisebox{-.48\height}{\includegraphics{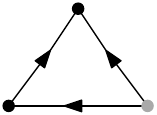}} \hspace{5mm} \overline{M}_G= \left[ \begin{array}{ccc|ccc} 1 & 0 & 1 & 1 & 0 & 1  \\ -1 & 1 & 0 & -1 & 1 & 0 \\ \hline 1 & 0 & 1 & 1 & 0 & 1  \\ -1 & 1 & 0 & -1 & 1 & 0 \\ \hline 1 & 0 & 1 & 1 & 0 & 1  \\ -1 & 1 & 0 & -1 & 1 & 0  \end{array} \right]$$

Similarly, for $K_4$, we produce the following fundamental matrix.

$$G = \raisebox{-.48\height}{\includegraphics[scale=.9]{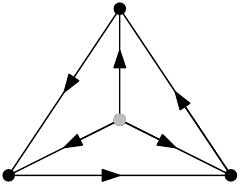}} \hspace{5mm} \overline{M}_G= \left[ \begin{array}{cccccc} 
1 & 0 & 0 & 1 & -1 & 0  \\ 
0 & 1 & 0 & -1 & 0 & 1 \\ 
0 & 0 & 1 & 0 & 1 & -1 \\ \hline
1 & 0 & 0 & 1 & -1 & 0  \\ 
0 & 1 & 0 & -1 & 0 & 1 \\ 
0 & 0 & 1 & 0 & 1 & -1  \end{array} \right]$$

\end{example}

Any decompleted $4$-regular graph will have a fundamental matrix that is a $2$-matrix, as for all graphs $G$ in this family, $|E(G)| = 2(|V(G)|-1)$. In fact, any fundamental block matrix is indeed a $k$-matrix, as $$\mathbf{1}_{m \times n} \otimes M = \mathbf{1}_{m \times 1} \otimes \left( \mathbf{1}_{1 \times n} \otimes M  \right).$$

\begin{proposition} \label{specialchoice} The choice of special vertex may only affect the overall sign of the permanent modulo $k+1$  in a $k$-matrix. If $k$ is odd, changing the special vertex results in a sign change. If $k$ is even, changing the special vertex has no effect on the permanent. \end{proposition}

\begin{proof} For signed incidence matrix $M^*$, let $r_1, ... , r_n$ be the rows associated to vertices $1,...,n$ in the original graph $G$, and suppose vertex $i$ is the special vertex, $i \in \{1,...,n\}$. Then, $$r_i = -(r_1 + r_2 + \cdots + r_{i-1} + r_{i+1} + \cdots + r_n),$$ a property of the signed incidence matrix. For all blocks in $M$, we may therefore turn row $r_j$, $i \neq j$, into row $r_i$ using the above equation. By Remark~\ref{rowops} and Corollary~\ref{reduction}, only the multiplication of a row in each block by $-1$ affects the permanent modulo $k+1$, flipping the overall sign once for each block. This produces the $k$-matrix where $j$ was the special vertex. The permanent is unaffected if $k$ is even, and multiplied by $-1$ otherwise.\end{proof}

We use Proposition~\ref{specialchoice} and Wilson's Theorem to show that the permanent of a fundamental matrix is invariant under choice of special vertex.

\begin{wilson} A number $p$ is prime if and only if $(p-1)! \equiv -1\pmod{p}$. For a composite number $n>4$, $(n-1)! \equiv 0 \pmod{n}$. \end{wilson}

\begin{theorem}\label{specialinvariant} Suppose $G$ is a graph, $M$ is a reduced signed incidence matrix associated to $G$, and consider the square matrix $N=\mathbf{1}_r \otimes \overline{M}$. Suppose every row in $N$ appears $k$ times, where $k$ is some multiple of $r$ by construction. With a fixed orientation to the edges, the permanent of this $k$-matrix $N$ is invariant under choice of special vertex modulo $k+1$. \end{theorem}

\begin{proof} By Corollary~\ref{oddcase} and Proposition~\ref{specialchoice}, the permanent is invariant modulo $k+1$ under choice of special vertex if $k$ is even or $k$ is odd and $|V(G)|>2$. As a graph with only a single vertex creates an empty matrix, it remains to be shown that the permanent is invariant if $|V(G)| = 2$ and $k+1$ is even.

 Suppose then that $G$ is a graph with two vertices and $k$ parallel edges. Applying an arbitrary orientation to the edges, the $k$-matrix of this graph has entirely nonzero entries, each column either all $1$ or all $-1$. Thus, it trivially has permanent $\pm k!$. As $k+1$ is even, it follows from Wilson's Theorem that $k! \equiv 0 \pmod{k+1}$ if $k+1 > 4$. Hence, we need only consider $k \in \{1,3\}$. As $\pm 1! \equiv 1 \pmod{2}$ and $\pm 3! \equiv 2 \pmod{4}$, the permanent is invariant under any choice made in constructing the $k$-matrix, as desired. \end{proof}

Noted prior, the fundamental matrix of any graph is a $k$-matrix for some $k$, as $\mathbf{1}_{k \times m} \otimes M = \mathbf{1}_{k \times 1} \otimes \left( \mathbf{1}_{1 \times m} \otimes M  \right)$. From this construction, the matrix $\mathbf{1}_{1 \times m}\otimes M$ may be viewed as the reduced signed incidence matrix of a graph with a number of edges added in parallel to the original graph. The following proposition is a restatement of Corollary~\ref{oddcase}, but in the language of fundamental matrices of graphs. This method of duplicating edges will be of use throughout this thesis, and will be discussed in greater detail in Remark~\ref{altrem}.

\begin{proposition} \label{prime} For non-prime $k+1$, the permanent of any square $k$-matrix associated to a graph $G$ with $|V(G)|>2$ is zero modulo $k+1$. \end{proposition}

From Proposition~\ref{prime}, we see that only prime residues are of interest when computing permanents of fundamental matrices for non-trivial graphs. The following classical theorem, coupled with Proposition~\ref{prime}, is key to our construction of sequences based on the permanent.

\newtheorem*{dirichlet}{Dirichlet's Theorem}
\begin{dirichlet} For relatively prime $a$ and $b$, the sequence $(an+b)_{n \in \mathbb{N}}$ contains infinitely many primes. \end{dirichlet}

It follows that there are infinitely many primes of the form $an+1$ for arbitrary positive integer $a$.

\begin{definition} Let $G$ be a graph and $\overline{M}_G = \mathbf{1}_{k \times m} \otimes M_G$ a fundamental matrix of $G$. Let $(p_i)_{i \in \mathbb{N}} $ be the increasing sequence of all primes that can be written $p_i = n_ik+1$ for some non-negative integer $n_i$. Then, matrix $\mathbf{1}_{n_i} \otimes \overline{M}_G$ is square and each row appears $n_ik$ times. As such, the permanent of $\mathbf{1}_{n_i} \otimes \overline{M}_G$ is well-defined modulo $n_ik+1 =p_i$. Call this residue the \emph{$p_i^\text{th}$ graph permanent}, $\text{GPerm}^{[p_i]}(G)$. Define the \emph{extended graph permanent} for $G$ as the sequence $$\left( \text{GPerm}^{[p_i]}(G) \right)_{i \in \mathbb{N}}.$$ \end{definition}

It is important to note that this sequence is not in general trivial, or in some obvious way based entirely on the smallest term -- see Appendix~\ref{chartofgraphs} for justification of this. Trees, which uniquely produce sequences that have values at all primes, are a special class that will be discussed in Section~\ref{treesnshit}. The $4$-point graphs in $\phi^4$ theory, our motivating class, produce sequences with values at all odd primes. 

The extended graph permanent relies on the arbitrary orientation of edges in a graph in the construction of the matrix. As changing the orientation is equivalent to multiplying a column of the signed incidence matrix by $-1$, there is potentially a sign ambiguity associated to this permanent. However, as the definition fixes an orientation for all copies of the edge-defined columns, this sign ambiguity occurs only over primes that require an odd number of duplications of columns, and the ambiguity affects all values of this type together. The sign ambiguity will be discussed in greater detail in Section~\ref{signambiguity}.

\begin{remark} For graph $G$, note that the reduced signed incidence matrix has dimensions $\left( |V(G)|-1 \right) \times |E(G)|$. Let $L = \lcm (|V(G)|-1,|E(G)|)$. A fundamental matrix $\overline{M}_G$ therefore has dimensions $L \times L$. Let $\E = \frac{L}{|E(G)|}$ and $\V = \frac{L}{|V(G)|-1}$. Each row in $\overline{M}_G$ appears $\V$ times, and each column $\E$ times. The extended graph permanent for $G$ is defined over primes of the form $p = \V n+1$ for some integer $n$. \end{remark}

\begin{remark}\label{altrem} We may alternately define the extended graph permanent in a more structural setting. Create the graph $G^{[n]}$ by replacing all edges in $G$ with $n$ edges in parallel. Let $M_n$ be a signed incidence matrix of $G^{[n]}$ with some choice of special vertex deleted, such that all edges in parallel are oriented in the same direction. Then, when there are values $k,n \in \mathbb{N}$ such that $\mathbf{1}_{k \times 1} \otimes M_n$ is square and $k+1 = p$ is prime, $\text{GPerm}^{[p]}(G) = \text{GPerm}^{[p]}(G^{[n]}) \equiv  \text{Perm}(\mathbf{1}_{k \times 1}\otimes M_n) \pmod{p}$. \end{remark}

\begin{remark} \label{connected} While the definition of the extended graph permanent makes no mention of connectedness of the graph, a connected component that does not contain the special vertex will cause the permanent to vanish for all primes. This is consistent with the quantum field theory motivation.  

If we instead require one special vertex per connected component, the matrix again becomes full rank. It is impossible in this matrix to differentiate between this disconnected graph and a similar connected graph where the special vertex in each connected component is identified, resulting in a cut vertex. By Theorem~\ref{specialinvariant}, we may therefore cleave a graph at a cut vertex, switch which vertex is special in each component, and then identify the special vertices again. \end{remark}

\section{A graphic interpretation of the extended graph permanent}\label{graphicegp}

\begin{remark}\label{excludeloops} A graph with a loop edge has a reduced signed incidence matrix with a column that is equal to the zero vector, per standard graph theory conventions. As such the permanent of this matrix is equal to zero. Hence, we will generally ignore these graphs. In this section in particular, they must be expressly forbidden.  \end{remark}

Recall that for an $n \times n$ matrix $A$, the permanent of $A$ is $\text{Perm}(A) = \sum_{\sigma \in S_n} \prod_{i=1}^n a_{i,\sigma(i)}$.  If a particular $\sigma \in S_n$ is such that $ \prod_{i=1}^n a_{i,\sigma(i)} \neq 0$, we will say that it \emph{contributes} to the permanent.

Consider a graph $G$ and an associated fundamental matrix $M_G$ with special vertex $v \in V(G)$. Per Remark~\ref{altrem}, we allow for the possibility that $G = (G')^{[n]}$ for some graph $G'$. Suppose each row in $M_G$ appears $k$ times. For each contribution to the permanent, precisely one nonzero value is selected from each row and similarly from each column. Fix such a contribution. Given the block structure that is used to create matrix $M_G$, we may associate each block of rows with a unique colour. Then, each edge is selected once, and each non-special vertex $k$ times. Assign colour $c$ to an edge if the contribution uses a value in the associated column that is in the $c^\text{th}$ block. For each coloured edge, assign a tag on the edge close to the vertex that uses that edge in $M_G$.

\begin{example}  Consider the graph $K_4$, drawn below, and an associated signed incidence matrix. We include the special vertex in the matrix for completeness in this example. The selection of entries in this matrix that form a contribution is shown, entries in the first block coloured black, and those in the second block coloured grey.  The colours and tags are indicated on the graph.

\centering \includegraphics[scale=.95]{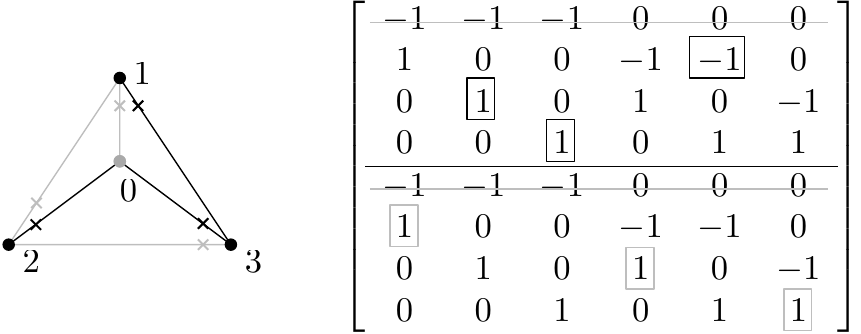}
\end{example}

Note from this construction that the special vertex cannot receive a tag. All other vertices must receive precisely $k$ tags, one on an edge of each colour. In fact, an arrangement of edge tags and colours on $G$ that assigns each non-special vertex $k$ tags - one on an edge of each colour - and no tags to the special vertex can immediately be turned into a selection of nonzero entries in the $k$-matrix.

\begin{example} The graph below is created by duplicating edges in $K_3$, and is drawn with an arbitrary edge orientation and special vertex again in grey. We use parallel edges, per Remark~\ref{altrem}, to graphically represent duplicated columns in the associated fundamental matrix. A contribution to the permanent, and the associated edge tagging, is shown below.

 \centering     \includegraphics[scale=1.0]{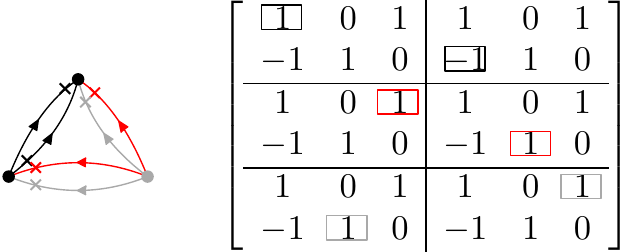}

 \end{example}

\begin{remark} \label{tagcolourbij} There is a bijection between these assignments of tags and colours and the contributions to the permanent. \end{remark}

We will use this bijection in a number of proofs as a way of considering these contributions on the graph itself.

\begin{remark} \label{ops} There are two operations on the tags and colours of the graph that produce other contributions to the permanent. The first is, for any non-special vertex, we may permute the colours of the $k$ edges that have a tag at that vertex. There are $k!$ ways to perform this permutation. 

The second operation, at its most intuitive, is that we may switch which vertex receives the tag on every edge in a cycle where all edges have the same colour. Recall that each vertex gets precisely one tag of each color. As the previous operation allows the colours of edges that have tags at a common vertex to be interchanged, this may be restated as reversing the directions of the tags of a cycle $C$ in $G$ such that each vertex in the cycle receives one tag from this edge set. \end{remark}

Suppose then that the graph $G$ has $n$ vertices. From the colour permuting operation, each valid configuration of tags produces $(k!)^{n-1}$ valid colourings. As the tags determine the position in the original matrix that is selected, choice of edge colours does not affect the value of the contribution. As such, $(k!)^{n-1}$ is a factor in the permanent, and the colours do not matter.

It is interesting to note, then, that the value of the permanent for a $k$DSI matrix is determined completely by the tag assignments. This interpretation is similar to that presented by Matou\v{s}ek (which he credits to Chaiken) to prove the matrix-tree theorem in \cite{matousek}, miniature 21.

\begin{proposition} For a $k$-matrix associated to a graph, we may produce all contributions to the permanent from a single contribution and the two operations stated in Remark~\ref{ops}. \end{proposition}

\begin{proof} By Remark~\ref{ops}, it is sufficient to show that all valid tag placements can be obtained. Starting from a fixed orientation, then, consider getting to another by switching which vertex receives the tag on a set of edges. Since each edge must receive $k$ tags, it is immediate that this selection of edges must induce a subgraph such that every vertex has even degree, thus a collection of cycles. This completes the proof. \end{proof}

\begin{remark}\label{compval} The value of a particular contribution, if the underlying orientation is known, can be determined by the position of the tags. An entry in a matrix with value $1$ is selected if the tag on an edge appears at the head of the edge, and the value $-1$ is selected otherwise. As such, the parity of the number of times the tags disagree with the underlying orientation determines the value of each contribution.  \end{remark}

\begin{remark} Returning to graphs with loops briefly, recall that the extended graph permanent of such a graph is necessarily zero at all primes. It is interesting that the tagging information does in fact capture this. That is, partition the taggings of such a graph based on which end of the loop receives a tag. By construction, we may move between these two sets bijectively just by switching which end of the loop receives the tag, as both mark the same vertex. If we were to apply Remark~\ref{compval}, this changes the sign of each tagging, and hence the sum must be equal to zero. While graphs with loops therefore do not capture the aspect of tagging as a selection of nonzero elements of the matrix, this tagging method does compute the permanent in graphs with loops.  \end{remark}

\section{Sign ambiguity} \label{signambiguity}

It is an unfortunate aspect of the arbitrary nature of the underlying edge orientation that a sign ambiguity appears in the extended graph permanent. Given, however, that we demand all duplicated edges or copies of the fundamental matrix preserve this initial orientation, there is only a small ambiguity to the invariant. If each edge is duplicated an even number of times the value is not influenced by the edge orientation. For $4$-point graphs in $\phi^4$ theory, this corresponds to primes in the sequence of the form $4k+1$ for some integer $k$. All other values change sign together with a change in orientation, corresponding to all matrices having an odd number of columns multiplied by $-1$. 

As a result, this does little to reduce the surprise of finding familiar sequences. Over the first twelve odd primes, there are approximately $1.52\times 10^{14}$ possible sequences of residues. The sign ambiguity then allows approximately $7.6\times  10^{13}$ sequences over the first twelve odd primes. The occurrence of sequences identical to those of the $c_2$ invariant (Sections~\ref{wheels} and~\ref{compr10}), or modular forms such that the weight of the modular form is equal to the loop number of the graph (Section~\ref{modformscoeffs}), is unlikely to be merely a coincidence.

Other families of graphs may avoid this sign ambiguity entirely. Suppose the reduced signed incidence matrix $M$ results in fundamental matrix $\overline{M} = \mathbf{1}_{j,k} \otimes M$ where $k$ is even. Then, all entries in the extended graph permanent are computed using matrices of the form $\mathbf{1}_{t} \otimes \overline{M}$ for integer $t$, and hence have an even number of copies of each column. By construction, then, there is no sign ambiguity in this extended graph permanent. The graph $K_3$, for example, produces a $2 \times 3$ reduced signed incidence matrix and hence a $6 \times 6$ fundamental matrix, each column appearing twice. Defined over primes of the form $p = 3n+1$ for integers $n$, there is no sign ambiguity in the extended graph permanent of this graph.  Similarly, if fundamental matrix $\overline{M}= \mathbf{1}_{j,k}\otimes M$ where both $j$ and $k$ are odd, then for any odd prime the matrix used in the permanent calculation must be $\mathbf{1}_m \otimes \overline{M}$ for some even $m$, and hence each edge is duplicated an even number of times. As both elements of $\mathbb{F}_2$ are invariant under sign change, residues modulo $2$ of course have this property as well: this applies only to trees.

\section{A relation to nowhere-zero flows}\label{nowherezero}

The connection between the Feynman period and flows inspired the work on the extended graph permanent, specifically as the matrix permanent can be used to certify the existence of a particular class of flow in a graph. What follows is an explicit connection between certain nonzero entries in the extended graph permanent and the existence of $\mathbb{Z}_p$ orientations in a graph. This is specifically DeVos' work, presented in \cite{pppoaiftposim}.

Recall from Section~\ref{graphicc2} that for a graph $G$ with edge weights  $\phi: E(G) \rightarrow \mathcal{G}$, the boundary function is $$(\partial_v)\phi = \sum_{\substack{e \in E(G)\\h(e) = v}} \phi(e) - \sum_{\substack{e \in E(G)\\t(e) = v}} \phi(e), $$ where $\mathcal{G}$ is any abelian group (though we will again restrict to prime order finite fields later). A $\mathcal{G}$ flow is a weighting of the edges such that, for all $v \in V(G)$,  $(\partial_v) \phi = 0$ in $\mathcal{G}$. We say that $\phi$ is \emph{nowhere-zero} if $\phi(e) \neq 0$ for all $e \in E(G)$. See \cite{genflows} for a broader introduction to nowhere-zero flows. The following important properties are due to Tutte.

\begin{theorem}[\cite{imbedding}, \cite{chromtut}] $  $ \begin{enumerate} 
\item For planar graph $G$ and dual $G^*$, $G$ has a $k$-colouring if and only if $G^*$ has a nowhere-zero $\mathbb{Z}/k\mathbb{Z}$ flow.

\item Graph $G$ has a nowhere-zero $\mathbb{Z}/k\mathbb{Z}$ flow if and only if it has a $\mathbb{Z}$ flow with range a subset of $\{ \pm 1 , \pm 2 , ... , \pm (k-1)\}$. 

\item If $G$ has a nowhere-zero $\mathcal{G}$ flow for finite abelian group $\mathcal{G}$, then it has a nowhere-zero $\mathcal{G}'$ flow for every abelian group $\mathcal{G}'$ such that $|\mathcal{G}'| \geq |\mathcal{G}|$.
  \end{enumerate}  \end{theorem}

In studying flows, Jaeger introduced the concept of a \emph{modulo $k$ orientation} (\cite{jaeger}), which is an orientation such that the difference between the in- and out-degree at each vertex is congruent to zero modulo $k$. This is similar to the $\mathbb{F}_p$ flows seen prior, with the restriction that all edges receive weight in $\{ \pm 1 \}$.  As such, a modulo $k$ orientation can be turned into a nowhere-zero $\mathbb{Z}/k\mathbb{Z}$ flow by assigning each edge a value of $1$. In the other direction, if a graph $G$ has a $\mathbb{Z}/k\mathbb{Z}$ flow $\phi:E(G) \rightarrow \mathbb{Z}/k\mathbb{Z}$ such that $\phi(e) \in \{\pm 1\}$ for all $e \in E(G)$, we may obtain a modulo $k$ orientation by reversing direction of edges with flow value $-1$. Hence, the existence of a modulo $k$ orientation is equivalent the existence of a $\mathbb{Z}/k\mathbb{Z}$ flow with all edge weights in $\{ \pm 1\}$.

For graph $G$ with an arbitrary orientation and signed incidence matrix $M^*_G$, treating entries in $M^*_G$ as elements of $\mathcal{G}$, the flows in $\mathcal{G}$ are vectors in the nullspace of $M^*_G$ over $\mathcal{G}$. We may further restrict to the reduced signed incidence matrix $M$, as the rank is unaffected so the nullspace is preserved. Hence, a nowhere-zero $\mathcal{G}$ flow in $G$ corresponds to a vector in the nullspace of $M$ with no zero entries. Similarly, a modulo $k$ orientation in $G$ corresponds to a $\pm 1$-valued vector in the nullspace of $M$ over $\mathbb{Z}/k\mathbb{Z}$.

A key tool that we will use is the following polynomial method, due to Alon and Tarsi.

\begin{theorem}[\cite{alontarsi}]\label{alontarsipoly} Let $\mathbb{F}$ be a field and $f(x_1,...,x_n)$ a polynomial in $\mathbb{F}[x_1,...x_n]$. Suppose that the coefficient of $x_1^{d_1}x_2^{d_2} \cdots x_n^{d_n}$ is nonzero and that the degree of $f$ is $d_1 + d_2 + \cdots + d_n$. Then for every $S_1, ... , S_n \subseteq \mathbb{F}$ with $|S_i| > d_i$ for all $i \in \{1,...,n\}$, there exist $s_i \in S_i$ such that $f(s_1,...,s_n) \neq 0$. \end{theorem}

We may now explore the certificate that the extended graph permanent provides to the existence of a modulo $p$ orientation. The following theorem is implicitly introduced in \cite{vectorspacesprimefields}. This proof is from \cite{pppoaiftposim}.

\begin{theorem}[\cite{vectorspacesprimefields}] \label{makingorientations} Let $p$ be prime. Let $G$ be a graph and $H$ a spanning subgraph of $G$ such that $|E(H)| = (p-1)(|V(G)|-1)$. If a fundamental $(p-1)$-matrix   of $H$ has nonzero permanent modulo $p$, then $G$ has a modulo $p$ orientation. \end{theorem}

\begin{proof} Orient the edges of $G$ arbitrarily. Define $H' = G - E(G)$, and map $\phi':E(H') \rightarrow \{\pm1\}$ arbitrarily. We will use Theorem~\ref{alontarsipoly} to show that $\phi'$ can be extended to an $\mathbb{F}_p$ flow of $G$ with range $\pm1$. 

Fix a vertex $w \in V(G)$ as the special vertex arbitrarily.  Define $V' =V-\{w\}$. For all $v \in V'$, define $A_v = \mathbb{F}_p - \{(\partial_v)\phi'\}$. From this, build the polynomial $f$ using Schwinger parameters for all edges in $E(H)$, $$ f = \prod_{v \in V'} \prod_{a \in A_v} \left( a + \sum_{\substack{e \in E(H)\\h(e) = v}} x_e - \sum_{\substack{e \in E(H)\\t(e) = v}} x_e  \right). $$

Note that $f$ has degree at most $(p-1)(|V(G)|-1) = E(H)$. As such, the coefficient of $\mathbf{x} = \prod_{e \in E(H)} x_e$ in $f$ is the same as the coefficient of $\mathbf{x}$ in $$g= \prod_{v \in V'} \left( \sum_{\substack{e \in E(H) \\ h(e) = v}} x_e - \sum_{\substack{e \in E(H) \\ t(e)=v}} x_e \right)^{p-1} .$$ This coefficient is also the permanent of the $(p-1)$-matrix built from the incidence matrix of $H$ with $w$ as the special vertex; it corresponds to each edge being selected once at either its head or tail, and each non-special vertex being selected $p-1$ times by these tags on incident edges, hence giving a contribution. By assumption, this permanent is nonzero in $\mathbb{F}_p$. Hence, the coefficient of $\mathbf{x}$ is nonzero in $f$.

As the degree of $f$ is at most $(p-1)(|V(G)|-1)$, and since $\mathbf{x}$ has nonzero coefficient, the degree of $f$ is precisely $(p-1)(|V(G)|-1)$. By Theorem~\ref{alontarsipoly}, we may choose an assignment of variables $\phi:E(H) \rightarrow \{\pm1\}$ in such a way that evaluating $f$ on these variables is nonzero.  As $f$ is written, we see that for the polynomial to be nonzero for some assignment of Schwinger parameters, these inner-most factors must be nonzero for all vertices. Hence, the boundary function at vertex $v$ is not equal to $a$ for every $a \in A_v$. So, $(\partial_v)\phi = -(\partial_v)\phi'$. As this holds at every vertex in $V'$, the function $\phi \cup \phi' :E(G) \rightarrow \{\pm 1\}$ is a flow, mapping edges $e \in E(H')$ to $\phi'(e)$ and edges $e \in E(H)$ to $\phi(e)$. As this is a $\mathbb{F}_p$ flow with edge assignments only $1$ or $-1$, we may turn this into a modulo $p$ orientation by reversing the direction of edges with flow value $-1$. \end{proof}

An alternate approach to the previous proof hints at a connection between the extended graph permanent and the Feynman integral, and will be presented in Chapter~\ref{conclusion}.

\chapter{Invariance Under Period Preserving Operations}
\label{Invariance}

We return now to our motivation, the Feynman period, and show that the extended graph permanent is invariant under all graphic operations known to preserve the period of $4$-point $\phi^4$ graphs.

\section{Decompletion invariance}\label{secdecompletion}

We begin by showing that the extended graph permanent, like the Feynman period, is invariant under choice of decompletion vertex for $2k$-regular graphs. A useful tool, combined with the graphic interpretation of the extended graph permanent seen in Section~\ref{graphicegp}, is the following extension to completed graphs.

\begin{remark}\label{extending} Let $G$ be a graph and $v \in V(G)$. Suppose we are computing the permanent of a fundamental $k$-matrix associated to the decompleted graph $G-v$ using contributions, as in Section~\ref{graphicegp}. We may then extend this notion of taggings from $G-v$ to $G$ by adding vertex $v$ back, and insisting that it receives the tags from all incident edges. Then, the special vertex receives no tags, the decompletion vertex receives all possible tags, and all other vertices receive $k$ tags. Clearly, this extension does not affect the collection of valid taggings of $G$.

This extension does not affect our computation of the permanent using contributions, as edges incident to $v$ do not appear in the matrix. Viewing the value of each contribution as a measure of the number of edges that agree with the underlying orientation of the graph, then, we may preserve this notion by orienting all edges in $G$ that are incident to $v$ towards $v$. \end{remark}

\begin{theorem}\label{egpcompletion} Let $G$ be a $2k$-regular graph. For any choice of $v\in V(G)$, $G - v$ has the same extended graph permanent.  \end{theorem}

\begin{proof} Let $v,w \in V(G)$. We prove this by showing that for any odd prime $p = nk+1$, there is an orientation of the edges of $G_v = G-v$ and $G_w = G - w$ such that $\text{GPerm}^{[p]}(G_v) = \text{GPerm}^{[p]}(G_w)$. Let $w$ be the special vertex for $G_v$, and similarly let $v$ be the special vertex for $G_w$.

For a contribution to the permanent of $G_v$ for prime $p$, use the graph $(G_v)^{[n]}$ per Remark~\ref{altrem}. Extend such a tagging to the graph $G^{[n]}$ as in Remark~\ref{extending}. Apply an orientation to the edges so that all edges incident to $v$ are oriented towards $v$, and all edges incident to $w$ are oriented away from $w$. The remaining edges may be oriented arbitrarily.

We bijectively move between such a tagging of $(G_v)^{[n]}$ and $(G_w)^{[n]}$ by reversing the orientation of all tags, thus reversing the roles of $v$ and $w$ as the special and decompletion vertices. Further, reverse the underlying orientation of all edges. In doing so, the value of each contribution is fixed by this bijection, and thus the extended graph permanents are equal.  \end{proof}

\section{The Schnetz twist}\label{secschnetz}

Recall the Schnetz twist, introduced in Section~\ref{introinvariance} and seen in Figure~\ref{twist}. We extend the notion of the Schnetz twist to $2k$-regular graphs by demanding that both graphs in this figure are $2k$-regular.

\begin{proposition}\label{schnetz} Consider two $2k$-regular graphs that differ by a Schnetz twist, say $G_1$ and $G_2$. Decompletions of these graphs have equal extended graph permanents. \end{proposition}

\begin{proof} Label the vertices in the four-vertex cut as in Figure~\ref{twist}.  By Theorem~\ref{specialinvariant} and Theorem~\ref{egpcompletion}  we may chose vertex $v_3$ as the special vertex and $v_4$ as the decompletion vertex for both graphs. For prime $p = nk+1$, we again extend the contributions to $G_1^{[n]}$ and $G_2^{[n]}$ per Remark~\ref{extending}. These are both $2kn$-regular graphs.

Fix a contribution to the permanent in $G_1^{[n]}$. Then, the decompletion vertex receives $2kn$ tags, the special vertex receives none, and all others get $kn$ tags. Since we assume $G_1$ and $G_2$ are both $2k$-regular, suppose vertex $v_1$ is incident to $d_1$ edges on the left side of the $4$-vertex cut, and $v_3$ is incident to $d_3$ on the left. Then, vertices $v_2$ and $v_4$ must be incident to $d_1$ and $d_3$ edges on the left, respectively. If there are $v$ vertices properly contained on the left, then there are $$\frac{1}{2}(2knv+2d_1+2d_3) = knv+d_1+d_3$$ edges, and hence total tags, on the left. Each of the $v$ vertices properly on this side receive $kn$ tags, while $v_3$ receives none and $v_4$ receives $d_3$. Thus, if $v_1$ receives $t$ tags on the left, then $v_2$ must receive $$ (knv+d_1+d_3)-(knv+d_3+t) =d_1-t.$$ 

By construction, $v_1$ must receive $kn-t$ tags on the right side of the cut, while $v_2$ receives $kn-d_1+t$. Consider reversing the direction of all tags on the right side in this contribution. Then, $v_1$ receives $(2kn-d_1)-(kn-t)=kn-d_1+t$ tags on the right, while $v_2$ receives $(2kn-d_1)-(kn-d_1+t)=kn-t$. Further, $v_3$ receives $2kn-d_3$ tags on the right, and $v_4$ receives none. Changing the edges of the form $\{v_i,w\}$ for $i\in \{1,2,3,4\}$ for edges on the right as in the Schnetz twist, this then becomes a contribution to the permanent in $G_2^{[n]}$. Clearly, this is a bijection. Fixing an orientation in $G_1$ arbitrarily, and an orientation in $G_2$ by reversing the direction of all edges on the right side of the $4$-vertex cut, we see that the values of all contributions are preserved. Hence, the permanents of these graphs are equal at this prime, and so the extended graph permanents must be equal. \end{proof}

\section{Planar duals}\label{secdual}

Last in our list of period preserving operations, planar duals are known to have equal periods. We now show that they also have equal extended graph permanents. We will first prove that a graph and its dual have extended graph permanents defined on the same set of primes.

\begin{lemma}\label{gcdstuff} For positive integers $s$ and $t$ such that $s>t$, $\gcd (s,s-t) = \gcd (s,t)$. \end{lemma}

The proof of this lemma follows immediately from the definition of the greatest common divisor.

\begin{proposition}\label{lcmstuff} For positive integers $s$ and $t$ such that $s>t$, $\frac{\lcm (s,t)}{t} = \frac{\lcm (s,s-t)}{s-t}.$ \end{proposition}

\begin{proof} Recall that for positive integers $a$ and $b$, $\lcm(a,b) \cdot \gcd(a,b) = a \cdot b$. Then, \begin{align*} \frac{\lcm(s,t)}{t} &= \frac{s}{\gcd(s,t)} \\ &= \frac{s}{\gcd(s,s-t)} \hspace{1cm} \text{by Lemma~\ref{gcdstuff}} \\ &= \frac{\lcm(s,s-t)}{s-t}. \end{align*} \end{proof}

\begin{corollary} \label{primesaregood} Suppose $G=(V,E)$ is a connected planar graph such that $|V|>1$ and $|E| > |V|-1$. Let  $G^*=(V^*,E^*)$ be the planar dual of $G$. Then $\frac{\lcm(|V|-1,|E|)}{|V|-1} = \frac{\lcm(|V^*|-1,|E^*|)}{|V^*|-1}$ and the extended graph permanents for $G$ and $G^*$ are defined on the same set of primes. \end{corollary}

\begin{proof} By construction, $G^*$ has $|E|$ edges, and by Euler's polyhedral formula $2+|E|-|V|$ vertices. As the fundamental matrix $M_G$ has dimensions $\left( |V|-1 \right) \times |E|$, the extended graph permanent of $G$ is defined over primes of the form $\frac{\lcm (|V|-1, |E|)}{|V|-1}n+1$. The fundamental matrix $M_{G^*}$ is of dimension $(1+|E|-|V|) \times |E|$, and hence the extended graph permanent of $G^*$ is defined over primes of the form $\frac{\lcm (1+ |E| -|V| , |E|)}{1 + |E| -|V|}n+1$. By Proposition~\ref{lcmstuff}, $ \frac{\lcm (|V|-1, |E|)}{|V|-1} = \frac{\lcm (1+ |E| -|V| , |E|)}{1 + |E| -|V|}$,  which completes the proof. \end{proof}

The following is Theorem 2.2.8 in \cite{Oxl}. 

\begin{theorem}\label{oxleydualtrueform} Let $M$ be the vector matroid of the matrix $[ I_r | D ]$ where the columns of this matrix are labeled, in order, $e_1,e_2,...,e_n$ and $1 \leq r < n$. Then $M^*$ is the vector matroid of $[ -D^T | I_{n-r} ]$ where its columns are also labeled $e_1,e_2,...,e_n$ in that order. \end{theorem}

We translate this to graph theoretic language for convenience, using in particular Theorem~\ref{matdual}.

\begin{theorem}\label{oxleydual} Let $G$ be a connected planar graph with $n$ edges, such that $G$ is neither a tree nor the empty graph. Order the edges of $G$ so that the first $r= |V(G)|-1$ form a spanning tree. Then, the reduced signed incidence matrix row reduces to $[I_r |A]$. Maintaining this ordering on the edges, the dual $G^*$ has reduced signed incidence matrix that row reduces to $[-A^T|I_{n-r}]$. \end{theorem}

This is key to proving that the extended graph permanent is indeed invariant under planar duals for $4$-point $\phi^4$ graphs. We require, then, that this row reduction does not change the permanents modulo the appropriate prime. We need to know then that we never must scale by a number other than $\pm 1$. To do this, we will use \emph{totally unimodular matrices}. These are matrices such that every square submatrix has determinant in $\{0,\pm1\}$. An important aspect of regular matroids is that they are precisely the matroids that are representable over $\mathbb{R}$ as a totally unimodular matrix. Though this immediately establishes that a signed incidence matrix is totally unimodular, we include a brief proof of this fact here.

\begin{lemma}\label{tummy} A signed incidence matrix is totally unimodular.  \end{lemma}

\begin{proof} We prove this by induction. By construction, any $1 \times 1$ submatrix will have determinant in $\{0, \pm1\}$. Suppose now that $S$ is a $t \times  t$ submatrix. If $S$ has a zero column, the determinant of $S$ is zero. If $S$ has a column with precisely one nonzero element, then cofactor expansion along this column can be used to compute the determinant using a smaller matrix,  and by induction the determinant of $S$ is in $\{0,\pm1\}$. Otherwise, all columns of $S$ contain one $1$, one $-1$, and the remaining entries $0$. As these row vectors sum to the zero vector, the determinant of $S$ is zero. \end{proof}

Let $A = [a_{x,y}]$ be a matrix. Following Oxley, define \emph{pivoting} on entry $a_{s,t}$ as the series of operations used in standard Gaussian elimination to turn the $t^\text{th}$ column into the $s^\text{th}$ unit vector. 

\begin{proposition}[\cite{Oxl}] Let $A$ be a totally unimodular matrix. If $B$ is obtained from $A$ by pivoting on the nonzero entry $a_{s,t}$ of $A$, then $B$ is totally unimodular.  \end{proposition}

The following corollary is therefore immediate.

\begin{corollary}\label{constmultrowred} Row reducing as in Theorem~\ref{oxleydual}, we may choose a sequence of operations such that multiplication of a row by a constant only ever uses constant $- 1$. \end{corollary}

To work with these matrices, we will use cofactor expansion techniques. An important one is included in the following remark.

\begin{remark}\label{howitgoes} Let $M=[I_r|A]$ be a matrix and $\overline{M}= \mathbf{1}_{\V \times \E} \otimes M$ the fundamental matrix of $M$. Suppose we want to find the permanent of $$\mathbf{1}_n \otimes \overline{M} = \mathbf{1}_n \otimes (\mathbf{1}_{\V \times \E} \otimes M) = \mathbf{1}_{n\V \times n\E} \otimes M.$$

Note that there are $n\E$ copies of each column of $M$ in $\overline{M}$, and similarly $n\V$ copies of each row. Performing cofactor expansion along all copies of a column in the identity matrix blocks then produces a factor of $$n\V (n\V-1) \cdots (n\V -n\E +1) = \frac{(n\V)!}{(n\V - n\E)!}.$$ Each expansion removes one specific row of $A$, so in total $n\E$ copies of this row are removed. Over all columns in the identity matrix blocks, then, we get $$\text{Perm}(\mathbf{1}_n \otimes \overline{M}) = \left( \frac{(n\V)!}{(n\V - n\E)!} \right)^r \text{Perm}(\mathbf{1}_{(n\V -n\E) \times n\E} \otimes A).$$
\end{remark}

\begin{proposition}\label{dual} Suppose the graph $G$ is a connected, planar, and not a tree. Let $\V_G$ be the number of copies of each row in the fundamental matrix of $G$, and $\E_G$ the number of copies of each column. Suppose the reduced signed incidence matrix for $G$ is $M_G$, and row reduces to $[I_{|V(G)|-1}|A]$. For prime $n\V_G +1$, \begin{align*} \text{Perm}(\mathbf{1}_n \otimes \overline{M}_G) &\equiv  \left( \frac{(n\V_G)!}{(n\V_G - n\E_G)!} \right)^{|V(G)|-1} \text{Perm} (\mathbf{1}_{(n\V_G -n\E_G) \times n\E_G} \otimes A) \pmod{n\V_G +1}.  \end{align*}  \end{proposition}

\begin{proof}
By Remark~\ref{rowops}, Corollary~\ref{reduction}, and Corollary~\ref{constmultrowred}, row reduction operations preserve the extended graph permanent modulo $n\V_G +1$, as restricting to non-trees forces all primes to be odd, and hence an even number of repeated rows for all matrices. We may therefore row reduce $M_G$, the signed incidence matrix of $G$, to $ \left[ \begin{array}{c|c} I_{|V(G)|-1} & A \end{array} \right]$ by Theorem~\ref{oxleydual}. 

For prime $n\V_G+1$, by Remark~\ref{howitgoes},  $$\text{Perm}(\mathbf{1}_n \otimes \overline{M}_G) \equiv \left( \frac{(n\V_G)!}{(n\V_G - n\E_G)!} \right)^{|V(G)|-1} \text{Perm}(\mathbf{1}_{(n\V_G -n\E_G) \times n\E_G} \otimes A) \pmod{n\V_G +1},$$ as desired. \end{proof}

Proposition~\ref{dual} translates to dual graphs quickly. For graph $G$ that meets the requirements and dual $G^*$, $|V(G^*)| = 2 - |V(G)| + |E(G)|$ by Euler's polyhedral formula, and $|E(G^*)| = |E(G)|$. Therefore, \begin{align*} \text{Perm}(\mathbf{1}_n \otimes \overline{M}_{G^*}) & \\ &\hspace{-24mm} \equiv \left( \frac{(n\V_{G^*})!}{(n \V_{G^*} - n\E_{G^*})!} \right)^{1-|V(G)|+|E(G)|} \text{Perm}(\mathbf{1}_{(n\V_{G^*} -n\E_{G^*}) \times n\E_{G^*}} \otimes -A^T) \pmod{n\V_G +1},\end{align*} using the reduction from Theorem~\ref{oxleydual}. By Corollary~\ref{primesaregood}, $\V_G = \V_{G^*}$, but $\E_G$ is not necessarily equal to $\E_{G^*}$.

\begin{corollary}[to Proposition~\ref{lcmstuff}] \label{brokenlcms} For positive integers $s$ and $t$ such that $s>t$, $\lcm (s,t) = \lcm (s-t,s)$ if and only if $2t = s$. \end{corollary}

\begin{proof} If $2t = s$, then $s-t = t$ and $\lcm (s-t,s) = \lcm (t,s)$. In the other direction, suppose $\lcm (s-t,s) = \lcm (t,s)$. By Proposition~\ref{lcmstuff}, $\frac{\lcm (s-t,s)}{s-t} = \frac{\lcm (t,s)}{t}$, so $s-t = t$ and hence $2t=s$.   \end{proof}

As $\E_G = \frac{\lcm (|V(G)|-1, |E(G)|)}{|E(G)|}$, Corollary~\ref{brokenlcms} shows that $\E_G = \E_{G^*}$ if and only if $2(|V(G)|-1) = |E(G)|$. It follows that duality for $4$-point $\phi^4$ graphs is a special instance of general duality. The following Lemma and Corollary will be of use in establishing the general duality statement for arbitrary graphs.

\begin{lemma}\label{anotherntthing} For positive integers $s$ and $t$ such that $s>t$, $\frac{\lcm (s,t)}{s} + \frac{\lcm (s,s-t)}{s} = \frac{\lcm (s,t)}{t}$. \end{lemma}

\begin{proof} As $\gcd (s,t) \cdot \lcm (s,t) = s \cdot t$, \begin{align*} \frac{\lcm(s,t)}{s} + \frac{\lcm(s,s-t)}{s} &= \frac{t}{\gcd(s,t)} + \frac{s-t}{\gcd(s,s-t)} \\ &= \frac{t}{\gcd(s,t)} + \frac{s-t}{\gcd(s,t)} & \text{ by Lemma~\ref{gcdstuff}} \\ &= \frac{s}{\gcd(s,t)} \\ &= \frac{\lcm(s,t)}{t}.  \end{align*}  \end{proof}

Let $G=(V,E)$ be a graph and $G^*=(V^*,E^*)$ its planar dual. Let $\overline{M}_G = \mathbf{1}_{\V_G \times \E_G} \otimes M_G$  and $\overline{M}_{G^*} = \mathbf{1}_{\V_{G^*} \times \E_{G^*}} \otimes M_{G^*}$ be respective fundamental matrices of these graphs. It follows from Lemma~\ref{anotherntthing} and Proposition~\ref{primesaregood} that $\E_G + \E_{G^*} = \V_G = \V_{G^*}$.

\begin{lemma} \label{anotherwilsoncor} For $j=a+b+1$ where $a$ and $b$ are positive integers, $$a!\cdot b! \equiv (-1)^b(j-1)! \pmod{j}.$$  \end{lemma}

\begin{proof} Briefly, \begin{align*} a! \cdot b! &\equiv a! (b (b-1) \cdots 1) \\ & \equiv a! ( (b-j)(b-1-j) \cdots (1-j)) \\ &\equiv (1 \cdots a)((-1)^b(j-b)(j-b+1) \cdots  (j-1)) \\ & \equiv (-1)^b (j-1)! \pmod{j} \end{align*} as $j-b = a+1$.  \end{proof}

Using the previous notation, it follows from Lemma~\ref{anotherwilsoncor} that $$(n\V_G)!\cdot (n\V_{G^*})! \equiv (-1)^{n\E_G}(n\V_G)! \pmod{n\V_G+1}.$$ While we will generally be assuming that $n\V_G+1$ is prime and hence further simplification follows from Wilson's Theorem, we will be using this to simplify future calculations, and hence leave this computation here.

It is also worth briefly noting that for a graph $G$ and fundamental matrix $\overline{M}_G = \mathbf{1}_{\V_G \times \E_G} \otimes M_G$, the product $n\E_G|E(G)|$ is always even, where $n$ is an integer such that $p=n\V_G+1$ is an odd prime. If we suppose that $\E_G$ and $|E(G)|$ are both odd, then as $\E_G = \frac{\lcm(|E(G)|,|V(G)|-1)}{|E(G)|}$, it follows that $\lcm(|E(G)|,|V(G)|-1)$ is also odd. Thus, $\V_G = \frac{\lcm(|E(G)|,|V(G)|-1)}{|V(G)|-1}$ must be odd also. As $p$ is assumed to be an odd prime, $n$ must therefore be even.

\begin{proposition}\label{painful} Let $G=(V,E)$ be a graph and $G^*=(V^*,E^*)$ its planar dual, and suppose they have fundamental matrices $\overline{M}_G = \mathbf{1}_{\V_G \times \E_G} \otimes M_G$ and $\overline{M}_{G^*} = \mathbf{1}_{\V_{G^*} \times \E_{G^*}} \otimes M_{G^*}$. For common prime $p = n\V_G +1$, $$ \text{GPerm}^{[p]}(G) = (-1)^{|E|-|V|+1}(n\E_G)!^{|E|}\text{GPerm}^{[p]}(G^*)  .$$ \end{proposition}

\begin{proof} Computing the extended graph permanent of $G$ at prime $p$ and modulo $p$, \begin{align*} \text{GPerm}^{[p]}(G) &\equiv \text{Perm}(\mathbf{1}_n \otimes \overline{M}_G) & \\ 
&\equiv \left( \frac{(n\V_G)!}{(n\V_G - n\E_G)!} \right)^{|V|-1} \text{Perm}(\mathbf{1}_{(n\V_G - n\E_G) \times n\E_G} \otimes A) & \text{ by Proposition~\ref{dual}} \\ 
&\equiv \left( \frac{(n\V_G)! (n\E_G)!}{(n\E_{G^*})!(n\E_G)!} \right)^{|V|-1}\text{Perm}(\mathbf{1}_{(n\V_G - n\E_G) \times n\E_G} \otimes A) & \text{ by Lemma~\ref{anotherntthing}} \\
&\equiv (-1)^{n\E_G(|V|-1)}(n\E_G)!^{|V|-1}\text{Perm}(\mathbf{1}_{(n\V_G - n\E_G) \times n\E_G} \otimes A) & \text{ by Lemma~\ref{anotherwilsoncor}}.  \end{align*}  Similarly for $G^*$, and as $\V_G = \V_{G^*}$ by Corollary~\ref{primesaregood}, \begin{align*} \text{GPerm}(G^*) & \equiv \text{Perm}(\mathbf{1}_n \otimes \overline{M}_{G^*})  \\
&\equiv \left( \frac{ (n\V_G)!}{(n\V_{G^*} - n\E_{G^*})!} \right)^{|E|-|V|+1} \text{Perm} (\mathbf{1}_{(n\V_{G^*} - n \E_{G^*}) \times n \E_{G^*}} \otimes -A^T)  \\
& \equiv \left( \frac{-1}{(n\E_G)!} \right)^{|E|-|V|+1} (-1)^{(|E|-|V|+1)n\E_G}\text{Perm}(\mathbf{1}_{n\E_G \times (n\V_G - n\E_{G})} \otimes A^T)   \\ 
& \equiv \frac{(-1)^{(|E|-|V|+1)(n\E_G+1)}}{(n\E_G)!^{|E|-|V|+1}} \text{Perm} (\mathbf{1}_{(n\V_G - n\E_G) \times n \E_G} \otimes A)  \pmod{n\V_G+1}. \end{align*}  

Note the common factors in these two equivalences. Therefore, \begin{align*} \text{GPerm}^{[p]}(G) 
&\equiv  (-1)^{n\E_G(|V|-1)}(n\E_G)!^{|V|-1} (n\E_G)!^{|E|-|V|+1}(-1)^{(|E|-|V|+1)(n\E_G+1)}\text{GPerm}^{[p]}(G^*) \\
&\equiv (-1)^{|E|-|V|+1}(n\E_G)!^{|E|}\text{GPerm}^{[p]}(G^*) \pmod{p}.\end{align*}
\end{proof}

To show specifically that this results in invariance for $4$-point $\phi^4$ graphs, we will use the following corollary to both Wilson's Theorem and Lemma~\ref{anotherwilsoncor}.

\begin{corollary}\label{wilsoncor} Let $p= 2n+1$ be an odd prime. Then, $$n!^2 \equiv \begin{cases} -1\pmod{p} \text{ if $n$ is even} \\ 1 \pmod{p} \text{ if $n$ is odd} \end{cases}.$$ \end{corollary}

\begin{corollary}\label{phi4dual} For planar graph $G=(V,E)$ where $2(|V(G)|-1) = |E(G)|$ and its planar dual $G^*$, $G$ and $G^*$ have equal extended graph permanents. \end{corollary}

\begin{proof} Here we consider primes of the form $p=2n+1$ for integers $n$. By Proposition~\ref{painful}, it suffices to consider only $(-1)^{|E| -|V|+1}(n\E_G)!^{|E|} \pmod{p}$. As $|E|=2(|V|-1)$ and $\E_G =1$, this is equivalent to $(-1)^{|V|-1}n!^{2(|V|-1)}$. 

If $n$ is even, then by Corollary~\ref{wilsoncor} $n!^2 \equiv -1 \pmod{p}$, and $$(-1)^{|V|-1}n!^{2(|V|-1)} \equiv (-1)^{2(|V|-1)} \equiv 1 \pmod{p} .$$ Otherwise, $n!^2 \equiv 1 \pmod{p}$, and $$ (-1)^{|V|-1}n!^{2(|V|-1)} \equiv (-1)^{|V|-1} \pmod{p} .$$ As these are primes for which the extended graph permanent vales may vary based on the underlying orientation of the directed graph, this produces either equivalence or a constant sign difference that can be corrected by reversing the direction of one edge. \end{proof}

\section{Two-vertex cuts}\label{sec2cut}

The $2$-vertex cut property of the period was introduced in Section~\ref{introinvariance} and shown in Figure~\ref{2cut}. To establish the comparable result for the extended graph permanent, we require an observation about the permanents of matrices.

\begin{lemma}\label{moregeneralpigeon} Let $M= \left[ \begin{array}{cc} A & B \\ \mathbf{0} & C   \end{array}  \right]$ where submatrix $A$ has dimensions $m \times n$, $m<n$, and $\mathbf{0}$ is the matrix with all entries $0$. The permanent of $M$ is zero.  \end{lemma}

\begin{proof} This follows from the pigeonhole principle and the definition of the permanent. \end{proof}

\begin{corollary}\label{pigeon} Suppose $M = \left[ \begin{array}{cc} A & \bf{0} \\ \bf{0} & B \end{array} \right]$ a square block matrix, where $A$ and $B$ are arbitrary and $\bf{0}$ has only $0$ entries. If $A$ is not square then the permanent of $M$ is zero.  \end{corollary}

We generalize the notion of the $2$-vertex cut property to graphs that meet the required vertex-edge relationship seen in $4$-point $\phi^4$ graphs. Drawn as in Figure~\ref{2cutgen}, we make no requirement as to the general behaviour of any of these graphs, only that all have this vertex-edge relationship. This is equivalent to the required distribution of external edges for $4$-point $\phi^4$ graphs seen in Figure~\ref{2cut}.

\begin{figure}[h]
  \centering
      \includegraphics[scale=1.20]{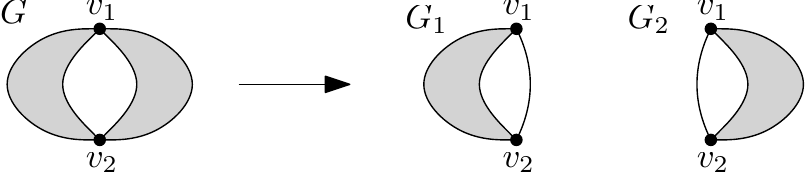}
  \caption{The general $2$-vertex cut property as it relates to the extended graph permanent. Here, we require only that all graphs $G'$ have $|E(G')| = 2|V(G')|-2$.}
\label{2cutgen}
\end{figure}

\begin{theorem} \label{2vertexcut} Consider the graph $G$ and two minors $G_1$ and $G_2$ seen in Figure~\ref{2cutgen}. If for all $G' \in \{G,G_1,G_2\}$, $2|V(G')| -2= |E(G')|$, then for all odd primes $p$,  $$\text{GPerm}^{[p]}(G) = - \text{GPerm}^{[p]}(G_1) \text{GPerm}^{[p]}(G_2) \pmod{p}.$$ \end{theorem}

\begin{proof} Set $v_2$ as the special vertex for all graphs, and write $(C|D)$ as the row corresponding to vertex $v_1 \in V(G)$. Then, we have signed incidence matrices 
$$M_G = \left[ \begin{array}{ccc|ccc}
& G_1 & & & \bf{0}& \\ 
&C& && D& \\ 
 & \bf{0} & & & G_2 & \end{array} \right], $$  

 $$M_{G_1} = \left[ \begin{array}{ccc|c}
& G_1& & \bf{0} \\  
&C& & 1\end{array} \right], $$ and 

 $$M_{G_2} = \left[ \begin{array}{c|ccc}
\bf{0} && G_2 &  \\ 
1 &&D&  \end{array} \right]. $$ By extension then, the fundamental matrices are $2$-matrices, and we want to compute permanents for $\mathbf{1}_{2k \times k} \otimes M$ for $M \in \{ M_G, M_{G_1}, M_{G_2} \}$.

Computing the permanent for $G$ by cofactor expansion along $2k$ rows $(C|D)$, the matrix after this sequence of expansions will have the form $\left[ \begin{array}{cc} A & \mathbf{0} \\ \mathbf{0} & B  \end{array}  \right]$, and the remaining blocks will only be square if $k$ columns are taken from the edges in $G_1$ and $k$ from edges in $G_2$. By Corollary~\ref{pigeon}, all other distributions will vanish, and hence will be ignored. We use notation $N_S$ to denote the matrix $N$ with a set of columns $S$ removed. Further, we assume the edges are oriented so that all entries in rows $(C|D)$ are in $\{ 0,1\}$, and for notational convenience take $\mathcal{C}$ and $\mathcal{D}$ as sets of the indices of columns that are nonzero in $C$ and $D$, respectively. Hence, \begin{align*} 
\text{Perm} \left( \mathbf{1}_{2k \times k} \otimes M_G \right)
&= (2k)! \sum_{\substack{i_1,...,i_k \in \mathcal{C} \\ j_1 , ..., j_k \in \mathcal{D}}} \text{Perm} \left( \left( \mathbf{1}_{2k \times k} \otimes \left[ \begin{array}{cc} G_1&\bf{0} \\ \bf{0}&G_2 \end{array} \right] \right)_{\substack{\{ i_1,...,i_k, \\ \hspace{4mm} j_1,...,j_k\} } } \right) \\ 
&= (2k)! \sum_{i_1,...,i_k \in \mathcal{C}} \text{Perm} \left( \mathbf{1}_{2k \times k} \otimes \left[ \begin{array}{c} G_1 \end{array} \right]_{\left\{i_1,...,i_k \right\}} \right) \\
&\hspace{1cm} \times \sum_{j_1 , ..., j_k \in \mathcal{D}} \text{Perm} \left( \mathbf{1}_{2k \times k} \otimes \left[ \begin{array}{c} G_2 \end{array} \right]_{\left\{ j_1,...,j_k \right\} } \right) . \end{align*}

Similarly, expanding along the $k$ rows corresponding to the new edges in $G_1$ and then the $k$ remaining rows corresponding to $(C)$, we get;
\begin{align*} 
\text{Perm} \left( \mathbf{1}_{2k \times k} \otimes M_{G_1} \right) 
&= \frac{(2k)!}{k!} \text{Perm} \left( \mathbf{1}_k \otimes \left[ \begin{array}{c} G_1 \\ G_1 \\ C  \\ \end{array} \right] \right) \\ 
&= \frac{(2k)!}{k!} k! \sum_{ i_1,...,  i_k \in \mathcal{C}} \left( \mathbf{1}_k \otimes  \text{Perm} \left[ \begin{array}{c} G_1 \\ G_1 \end{array} \right]_{ \{i_1,...,i_k\}} \right).  \end{align*} 
Similarly, $$\text{Perm}(\mathbf{1}_{2k \times k} \otimes M_{G_2})= (2k)! \sum_{j_1,...,j_k \in \mathcal{D}} \text{Perm} \left( \mathbf{1}_k \otimes \left[ \begin{array}{c} G_2 \\ G_2 \end{array} \right]_{\{ j_1,...,j_k\} } \right).$$ As $(2k)! \equiv -1 \pmod{2k+1}$  by Wilson's Theorem, the extended graph permanents differ by a constant sign.
  \end{proof}

  While this constant sign difference might not seem ideal, it does establish that such a graph operation is controlled with respect to the extended graph permanent. If two non-isomorphic graphs $G_1$ and $G_2$ are produced by joining the same two smaller graphs in this fashion, that $G_1$ and $G_2$ will indeed have equal extended graph permanents.
  
The graphs in Figure~\ref{didntwork} provide a counterexample to the notion of generalizing Theorem~\ref{2vertexcut} to $2$-vertex cuts in arbitrary graphs. Note first that the changed vertex-edge relationship means the extended graph permanents of $G$ and $G_1$ are no longer defined over the same set of primes. Using methods that will be discussed in Chapter~\ref{egpcomp}, the graph $G$ has extended graph permanent equal to $- \binom{5n}{n}^2 \binom{5n}{2n}^2 \pmod{8n+1}$ at prime $p=8n+1$, and $G_1$ and $G_2$ have extended graph permanent $- \binom{3n}{n}^2 \pmod{5n+1}$ at prime $p=5n+1$. In addition to the sequence of primes not matching in general, at common prime $p=241$, $G$ has extended graph permanent $201$, while $G_1$ has extended graph permanent $10$. Hence, the product of the extended graph permanents of graphs $G_1$ and $G_2$ do not match the extended graph permanent of $G$ at this prime.

The Whitney flip however, as in Chapter~\ref{chmatroid}, does not change the relationship between the number of edges and vertices in a graph.

\begin{figure}[h]
  \centering
      \includegraphics[scale=1.20]{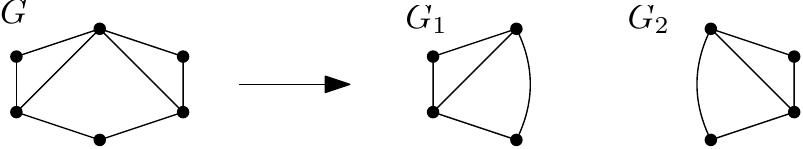}
  \caption{A set of graphs that provide a counterexample to a more general form of Theorem~\ref{2vertexcut} for arbitrary graphs.}
\label{didntwork}
\end{figure}

\begin{conjecture}\label{qwhitney} If two graphs differ by a Whitney flip, then they have equal extended graph permanents. \end{conjecture}

\noindent The technique used in Theorem~\ref{2vertexcut} does not immediate appear to be of use here, as cofactor expansion along rows associated to the vertices in the cut will result in a different set of deleted columns in the two graphs that differ by such a flip.

An interesting restriction follows from Lemma~\ref{moregeneralpigeon}.

\begin{proposition}\label{parallellimits} Let $G$ be a graph. Write $L = \lcm (|V(G)|-1, |E(G)|)$, $\E = \frac{L}{|E(G)|}$, and $\V = \frac{L}{|V(G)|-1}$, so the fundamental matrix for $G$ has $\E$ copies of each column and $\V$ copies of each row. Suppose there is a pair of vertices in $V(G)$ with $m$ edges in parallel between them. If $m\E > \V$, the extended graph permanent of $G$ is zero at all primes.   \end{proposition}

\begin{proof} Suppose vertices $a$ and $b$ have $m$ edges between them. Set $a$ as the special vertex. Then, gathering up identical rows and columns for the matrix used in computing the extended graph permanent for prime $n\V +1$, there are $n m \E$ columns that have a nonzero value only at the rows associated to vertex $b$. As there are $n\V$ such rows, $nm\E > n\V$ and it follows from Lemma~\ref{moregeneralpigeon} that this matrix has permanent zero.  \end{proof}

\section{Relation to the period}

Given the collection of theorems in this chapter, it is natural to make the following conjecture.

\begin{conjecture}\label{obviousconjecture} If two $4$-point $\phi^4$ graphs have equal period, then they have equal extended graph permanent. \end{conjecture}

\noindent This is supported by the fact that the graph operations shown in this chapter explain all known instances of $4$-point $\phi^4$ graphs with equal periods, the majority of which have been computed, up to graphs with eleven loops (\cite{galois}).

Consider then a subdivergence in a graph, as introduced in Section~\ref{introperiods}. It would make sense that if the extended graph permanent is related to the Feynman period, there would be some determinable property of the extended graph permanent for graphs with internal $4$-edge cuts.

\begin{theorem}\label{subdivthm} As in Figure~\ref{4edgecut}, suppose a completed $\phi^4$ graph $\Gamma$ has a $4$-edge cut, and label these minors $\Gamma_1$ and $\Gamma_2$. Call the decompleted graphs $G$, $G_1$, and $G_2$, respectively. Then, for odd prime $p$, $$\text{GPerm}^{[p]}(G) = \text{GPerm}^{[p]}(G_1)  \cdot \text{GPerm}^{[p]}(G_2) \pmod{p}.$$ \end{theorem}

\begin{figure}[h]
  \centering
      \includegraphics[scale=1.20]{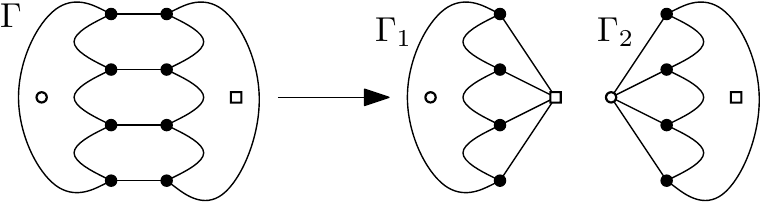}
  \caption{The product property for a completed $\phi^4$ graph with a $4$-edge cut. Marked vertices play a role in the proof of Theorem~\ref{subdivthm}.}
\label{4edgecut}
\end{figure}

\begin{proof} By Theorem~\ref{specialinvariant} and Theorem~\ref{egpcompletion}, the choice of special vertex and decompletion vertex do not affect the extended graph permanent. We choose then to decomplete these graphs at the vertex labeled with a hollow circle and make the vertex labeled with a square the special vertex. Hence, a fundamental matrix for $G$ is $$\overline{M}_G = \left[ \begin{array}{c|cc} I_4 & A & \mathbf{0} \\ \mathbf{0} & B & \mathbf{0} \\ -I_4 & \mathbf{0} & C \\ \mathbf{0} & \mathbf{0} & D \\ \hline  I_4 & A & \mathbf{0} \\ \mathbf{0} & B & \mathbf{0} \\ -I_4 & \mathbf{0} & C \\ \mathbf{0} & \mathbf{0} & D \end{array} \right],$$ where $B$ corresponds to edges in the left subgraph, $D$ the edges on the right, and $A$ and $C$ the vertices incident to the edges in this cut on either side. Fixing a prime $2k+1$ for integer $k$, the graph $G^{[k]}$ has fundamental matrix $\overline{M}_{G^{[k]}} = \mathbf{1}_k \otimes \overline{M}_G$. If we suppose $G$ has $l$ edges contained on the left of the $4$-edge cut, then there are $\frac{l-4}{2}$ vertices contained properly on this side (excluding the vertices incident to the edges in the cut), and block $B$ has $\frac{l-4}{2}$ rows. Similarly, if there are $m$ edges on the right of the $4$-edge cut in $G$, then there are $\frac{m-6}{2}$ vertices properly contained on the right block. As this side also contains the special vertex, block $D$ has $\frac{m-8}{2}$ rows. Therefore, performing cofactor expansion along the $4$ edges in this cut gives square blocks, and hence nonzero permanent by Corollary~\ref{pigeon}, if and only if the rows deleted by this operation meet block $A$. Therefore, we may write \begin{align*} \text{Perm} (\overline{M}_{G^{[k]}}) &= \left( \frac{2k!}{k!} \right)^4 \text{Perm} \left( \mathbf{1}_k \otimes \left[ \begin{array}{cc} B & \mathbf{0} \\ \mathbf{0} & C \\ \mathbf{0} & D \\ \hline A & \mathbf{0} \\ B & \mathbf{0} \\ \mathbf{0} & C \\ \mathbf{0} & D  \end{array} \right] \right) \\
&= \left( \frac{2k!}{k!} \right)^4 \text{Perm} \left( 1_k \otimes \left[  \begin{array}{c} B \\ A \\ B  \end{array}  \right]  \right) \text{Perm}\left( 1_k \otimes \left[ \begin{array}{c} C \\ D \\ C\\D  \end{array}  \right]  \right)
.\end{align*}

Similarly the fundamental matrices for $G_1$ and $G_2$ are $$\overline{M}_{G_1} = \left[ \begin{array}{cc} I_4 & A \\ \mathbf{0} & B \\ \hline I_4 & A \\ \mathbf{0} & B  \end{array}  \right], \hspace{2mm}  \overline{M}_{G_2} =  \left[ \begin{array}{c} C \\ D \\ C\\D  \end{array} \right] ,$$ so for $G_1$ and $G_2$ we have permanents \begin{align*} \text{Perm}(\mathbf{1}_k \otimes \overline{M}_{G_1}) &= \text{Perm} \left( \mathbf{1}_k \otimes \left[ \begin{array}{cc} I_4 & A \\ \mathbf{0} & B \\ \hline I_4 & A \\ \mathbf{0} & B  \end{array} \right] \right) = \left( \frac{2k!}{k!} \right)^4 \text{Perm} \left( \mathbf{1}_k \otimes \left[ \begin{array}{c} B \\ A \\ B  \end{array} \right] \right), \\ \text{Perm} (\mathbf{1}_k \otimes \overline{M}_{G_2}) &= \text{Perm} \left( \mathbf{1}_k \otimes \left[ \begin{array}{c} C \\ D \\ C\\D  \end{array} \right] \right). \end{align*} Therefore, $\text{Perm} (\overline{M}_{G^{[k]}}) = \text{Perm} (\overline{M}_{G_1^{[k]}}) \text{Perm}(\overline{M}_{G_2^{[k]}})$, and hence $$\text{GPerm}^{[p]}(G) = \text{GPerm}^{[p]}(G_1)\text{GPerm}^{[p]}(G_2),$$ as desired.
 \end{proof}
 
 It follows that the extended graph permanent has a term-by-term product property for $4$-point $\phi^4$ graphs containing subdivergences; the extended graph permanent of the graph is the product of the extended graph permanents of the graph with the subdivergence contracted and the subdivergence isolated. This provides further hints that the extended graph permanent is in some way connected to the Feynman period.

Theorem~\ref{subdivthm}, along with main results of this chapter, provide a strong foundation for believing that the extended graph permanent may be intrinsically connected in some way to the Feynman period for $\phi^4$ graphs. Just how this connection can be used, however, remains unclear. The start of these sequences for all primitive $\phi^4$ graphs up to eight loops are presented in Appendix~\ref{chartofgraphs} as families of decompletions of $4$-regular graphs. We ignore any graph that differs by one already computed by a Schnetz twist or duality of any pair of decompletions, per the equalities demonstrated in this chapter. We see in particular that the converse of Conjecture~\ref{obviousconjecture} does not appear to hold true. For example, using notation from \cite{Sphi4}, decompletions of graphs $P_{8,1}$ and $P_{8,10}$ appear to have the same extended graph permanent, though the periods are certainly not equal. From \cite{galois}, the period of $P_{8,1}$ is $1716 \zeta(13) \approx 1716.21$ while the period of $P_{8,10}$ has a longer multiple zeta representation\footnote{$\frac{19011}{2} \zeta (3) \zeta (9) + \frac{9}{5} \pi^6 \zeta (3)^2 - 441 \pi^2 \zeta (5)^2- \frac{63}{5}\pi^4 \zeta(5,3)- 3891 \zeta (5) \zeta (7) -378 \pi^2 \zeta(7,3)+3193 \zeta (9,3)+ 756 \zeta(6,4,1,1) -1260 \pi^2 \zeta(3) \zeta(7)+\frac{21}{5} \pi^4 \zeta (3) \zeta(5)+ \frac{113854613}{10216206000} \pi^{12}-63 \zeta(3)^4$}, and is approximately $735.764$.

\chapter{Computation of the extended graph permanent}
\label{egpcomp}

The permanents of large matrices are notoriously difficult to compute; the lack of row-reduction techniques mean that usually computations are done using the definition or cofactor expansion. However, as we desire only the residue, we can use row reduction, provided we have not prior used cofactor expansion to reduce the number of identical blocks. Further, our matrices are constructed with a great deal of repetition, which results in easier cofactor expansion. In this chapter, we simplify the computation of the extended graph permanents, and produce closed forms as reasonably small, single equations that work for all primes for several graphs, matroids, and graph families. We do this using standard combinatorial counting techniques and cofactor expansion. 

It is important and interesting to note that the computations in the chapter are used to produce the permanents over $\mathbb{Q}$ first, unless otherwise noted. Consider the fact that the permanent for graph $W_{44}$ at prime $3049$, in Section~\ref{wheels}, is the permanent of a $134122 \times 134122$ matrix and over twenty thousand digits in length. This is an example of the complexity that is being reduced to reasonable closed forms.

To emphasize the structural nature of our cofactor expansion, we will represent the permanents of $k$-matrices as weighted graphs, weights on edges counting the number of columns appearing in the matrix that represent that edge, and weights on vertices the number of rows in the matrix that represent that vertex. Since we are representing the permanents graphically, we will differentiate them from graphs by drawing these weighted graph representations of the permanent in square brackets.

Representations of this type are not unique. If a graph has multiple vertices of weight zero, those vertices are indistinguishable, as they correspond to rows that do not occur in the matrix. However, up to reordering the rows and columns, the graphical representation does uniquely produce a matrix. Trivially, the matrix must be square if we are to take a permanent, and hence we require that the sum of the vertex weights must be equal to the sum of the edge weights. For $4$-point $\phi^4$ graphs in particular and prime $p=2n+1$, non-special vertices will receive weight $2n$ and edges will receive weight $n$.

\begin{example} \begin{align*} \left[ \raisebox{-.48\height}{\includegraphics{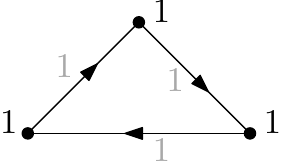}} \right] 
&= \text{Perm} \left[ \begin{array}{ccc} 1 & 0 & -1 \\ -1 & 1 & 0 \\ 0 & -1 & 1 \end{array} \right]\\
\left[ \raisebox{-.48\height}{\includegraphics{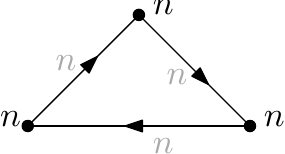}} \right] 
&= \text{Perm} \left( \mathbf{1}_n \otimes \left[ \begin{array}{ccc} 1 & 0 & -1 \\ -1 & 1 & 0 \\ 0 & -1 & 1 \end{array} \right] \right) \\
\left[ \raisebox{-.48\height}{\includegraphics{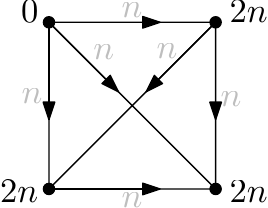}} \right] 
&= \text{Perm} \left( \mathbf{1}_{2n\times n} \otimes \left[ \begin{array}{cccccc} 1&0&0&-1&-1&0 \\ 0&1&0&1&0&1 \\ 0&0&1&0&1&-1 \end{array} \right] \right) 
 \end{align*} \end{example}

We may therefore introduce a general method for writing the cofactor expansion along rows, and hence at vertices, using this representation. Suppose that vertex $v$ has weight $w_v \neq 0$, and further that the $n$ incident edges $e_1 = (v,v_1),...,e_n = (v,v_n)$ have weights $w_1, ..., w_n$. Let $m_{e_i}$ denote the value in the matrix of edge $e_i$ at vertex $v$. Performing cofactor expansion along all rows corresponding to vertex $v$, 
\begin{align*}
\left[ \raisebox{-.48\height}{\includegraphics[scale=0.6]{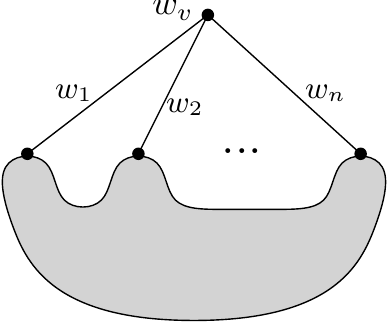}} \right] = \sum_{\substack{k_1 + \cdots + k_n = w_v\\ k_i \geq 0}} w_v! \prod_{j=1}^n \binom{w_j}{k_j}  m_{e_j}^{k_j} \left[ \raisebox{-.48\height}{\includegraphics[scale=0.6]{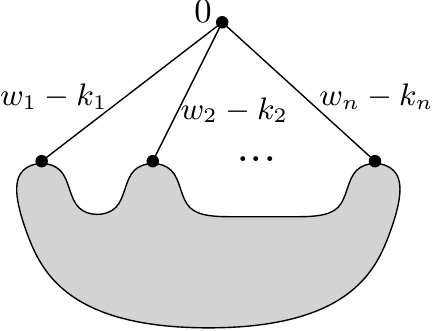}} \right],
\end{align*}
where the $k_i$ values count the number of times we meet each column associated to an incident edge in this cofactor expansion. The $w_v!$ factor comes from the fact that order matters in the selection of edges.

An alternate, and more graphic, interpretation of the cofactor expansion relies on counting the number of times a particular set of columns is deleted in this operation using a multinomial. That is, we first determine the number of times each column associated to an edge incident to this vertex will be deleted in this cofactor expansion, and count the number of such instances with a multinomial. Then, cofactor expansion along a row designated to meet a column indexed by edge $e_i$ meets that edge $w_i$ times in the first expansion, $w_i-1$ times in the second expansion, and so on. This gives \begin{align*}\left[ \raisebox{-.48\height}{\includegraphics[scale=0.6]{genex1}} \right] = \sum_{\substack{k_1 + \cdots + k_n = w_v\\ k_i \geq 0}} \binom{w_v}{k_1,...,k_n} \prod_{j=1}^n \frac{w_j!m_{e_j}^{k_j}}{(w_j-k_j)!} \left[ \raisebox{-.48\height}{\includegraphics[scale=0.6]{genex2}} \right].\end{align*} It is a straightforward check that these are equal.

In the special case that the vertex is incident to a single edge, the computation simplifies;  \begin{align*} \left[ \raisebox{-.48\height}{\includegraphics[scale=1]{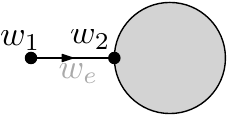}} \right] &= w_1! \binom{w_e}{w_1}(-1)^{w_1} \left[ \raisebox{-.48\height}{\includegraphics[scale=1]{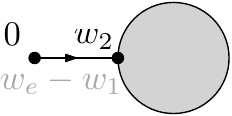}} \right]\\ &= \frac{w_e!(-1)^{w_1}}{(w_e-w_1)!} \left[ \raisebox{-.48\height}{\includegraphics[scale=1]{singvert2}} \right] .\end{align*}

One may also do cofactor expansion along a column, which corresponds to an edge. Herein, for algorithmic simplicity we will generally only use edges when the weight on one vertex is zero. Let $m_{v_i}$ be the value in the matrix at edge $e_i$ and vertex $v_i$. Then, with weights $w_i$ and $x_i$,\begin{align*}
\left[ \raisebox{-.48\height}{\includegraphics[scale=0.6]{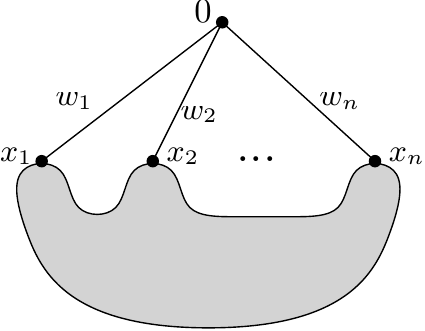}} \right]
&=  \prod_{i=1}^n \frac{x_i!}{(x_i-w_i)!} m_{v_i}^{w_i} \left[ \raisebox{-.48\height}{\includegraphics[scale=0.6]{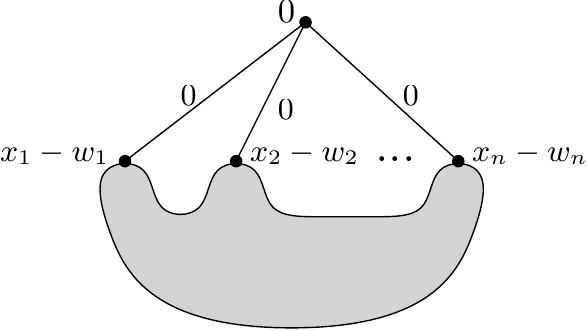}} \right]\\
&= \prod_{i=1}^n \frac{x_i!}{(x_i-w_i)!} m_{v_i}^{w_i} \left[ \raisebox{-.48\height}{\includegraphics[scale=0.6]{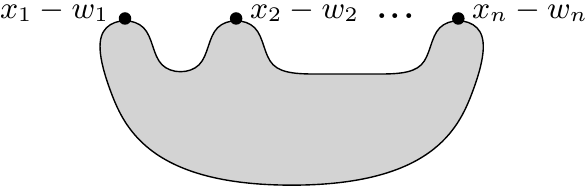}} \right].
\end{align*} This last line follows from the fact that an edge with weight zero contributes nothing to the matrix and is hence removable. Similarly, a vertex with weight zero and all incident edges having weight zero can be removed.

As orientations are ultimately arbitrary, we will from now on indicate directions on edges only when we are about to act upon that edge or vertex, purely for the sake of simplicity in the figures. For the sake of generality in these early examples, orientations were omitted from the figures, but can be extracted from the value $m_{e_i}$ in the equations, corresponding to the value of the entry in the matrix of edge $e_i$ at the chosen vertex.

\section{A small example} \label{smallstart}

We begin with the $4$-point $\phi^4$ graph $K_{3,4}$, a decompletion of $P_{6,4}$ in \cite{Sphi4}. While this computation is straightforward given our tools, it is a single, reasonably sized graph to orient us with this method. The remaining computations in this chapter will be on matroids or families of graphs, which will require additional tricks. For prime $2n+1$, 

\begin{align*}
\left[  \raisebox{-.48\height}{\includegraphics[scale=.7]{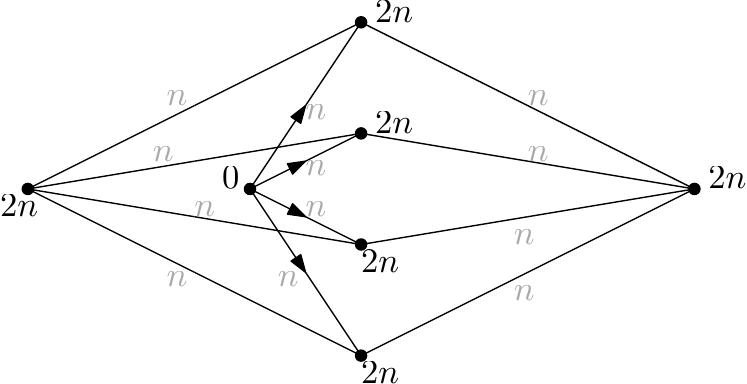}} \right] 
&= \left( \frac{(2n)!}{n!} \right)^4     \left[ \raisebox{-.48\height}{\includegraphics[scale=.7]{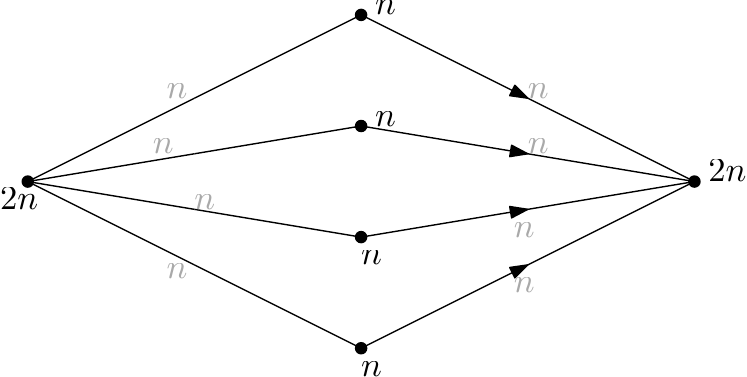}} \right] \\
&\hspace{-4cm}= \left( \frac{(2n)!}{n!} \right)^4 \sum_{\substack{k_1+k_2+k_3+k_4 = 2n\\0 \leq k_i \leq n}} \left( \prod_{i=1}^4 \binom{n}{k_i} \right) (2n)!   \left[ \raisebox{-.48\height}{\includegraphics[scale=.7]{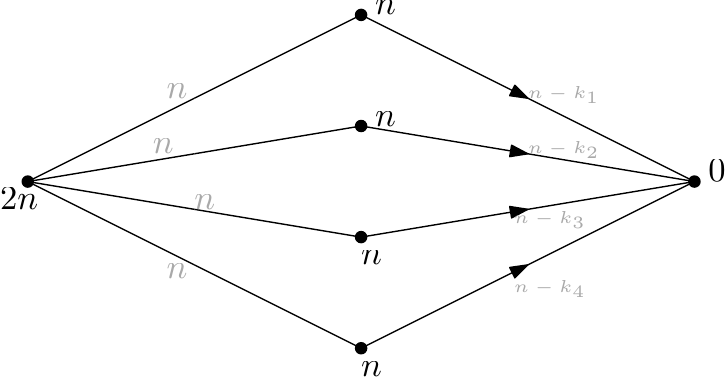}}  \right]\\
&\hspace{-3cm}=  \frac{(2n)!^5}{n!^4}  \sum_{\substack{k_1+k_2+k_3+k_4 = 2n\\0 \leq k_i \leq n}} \left( \prod_{i=1}^4 \binom{n}{k_i} \right) \frac{ n!^4 (-1)^{2n}}{\displaystyle{\prod_{i=1}^4} (k_i)!} \left[ \raisebox{-.48\height}{\includegraphics[scale=.7]{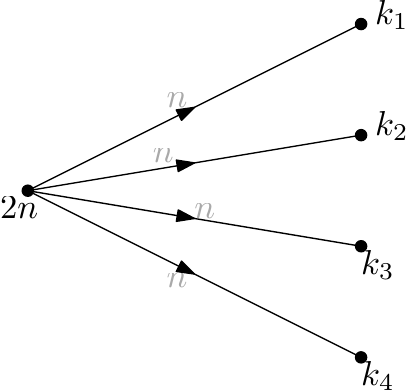}}  \right]\\
&\hspace{-3cm}= (2n)!^5 \sum_{\substack{k_1+k_2+k_3+k_4 = 2n\\0 \leq k_i \leq n}} \left( \prod_{i=1}^4 \binom{n}{k_i} \right) \frac{  n!^4 (-1)^{2n}}{\displaystyle{\prod_{i=1}^4} (k_i!(n-k_i)! ) }\left[ \raisebox{-.48\height}{\includegraphics[scale=.7]{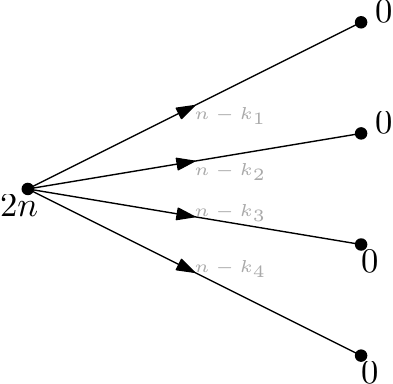}}  \right]\\
&\hspace{-3cm}= (2n)!^6 \sum_{\substack{k_1+k_2+k_3+k_4 = 2n\\0 \leq k_i \leq n}} \binom{n}{k_1}^2 \binom{n}{k_2}^2 \binom{n}{k_3}^2  \binom{n}{k_4}^2\\
&\hspace{-3cm} \equiv  \sum_{\substack{k_1+k_2+k_3+k_4 = 2n\\0 \leq k_i \leq n}} \binom{n}{k_1}^2 \binom{n}{k_2}^2 \binom{n}{k_3}^2  \binom{n}{k_4}^2 \pmod{2n+1}.
\end{align*}

Panzer and Schnetz prove in \cite{galois} that the period of this graph is $ \frac{-6912}{5}\zeta(5,3) +\frac{928}{2625}\pi^8 - 2592\zeta(3)\zeta(5) \approx 71.5  .$ The $c_2$ invariant for this graph is zero for all primes (\cite{BSModForms}), and the Hepp bound is $13968$ (\cite{ErikEmail}).

\section{Trees}\label{treesnshit}

The reduced signed incidence matrix of a tree $T$, $M_T$, will be a square matrix, and hence $M_T = \overline{M}_T$. As such, we are interested in $\mathbf{1}_n \otimes M_T$ for all primes $p=n+1$. Applying Wilson's Theorem to a minimal non-trivial tree, 
$$ \left[ \raisebox{-.48\height}{\includegraphics{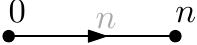}} \right] = \text{Perm} \left[ \mathbf{1}_{n} \right] = n! \equiv -1 \pmod{n+1}.$$ Note from Corollary~\ref{oddcase} that this is the unique graph with a nonzero value at a non-prime residue. At $4$, we would produce the matrix $\mathbf{1}_3$, which has permanent $2 \pmod{4}$. 

The decompletion of the graph $P_{1,1}$, the unique graph with two vertices and two edges in parallel between them, falls into this case. As $n$ will be even after prime two, the duplicated-edges view of the permanent is agnostic to one edge duplicated $n = 2k$ times or two edges in parallel duplicated $k$ times. 

For general trees, we progress inductively. Suppose we are computing the permanent for prime $p=n+1$. As any tree $T$ with at least two vertices starts with the special vertex having weight $0$ and all edges and non-special vertices with weight $n$, we will assume that the special vertex is a leaf. Hence, \begin{align*} \left[ \raisebox{-.48\height}{\includegraphics{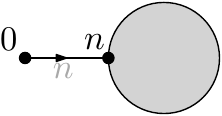}} \right] &= n! \left[ \raisebox{-.48\height}{\includegraphics{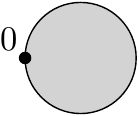}} \right]. \end{align*} This second figure represents the $(n+1)^\text{th}$ graph permanent of of a smaller tree. Thus, we may move the special vertex again to a leaf. With base case established prior, we get that $\text{GPerm}^{[p]}(T) = (-1)^{|V(T)|-1} \pmod{p}$. As mentioned in Section~\ref{signambiguity}, there is no sign ambiguity to any values in this sequence.

For most graphs, pendant vertices trivialize the computation.

\begin{proposition}\label{zeropendant} Suppose $G$ is a graph that is not a tree. If $v \in V(G)$ is a pendant vertex, then the extended graph permanent is zero for all primes. \end{proposition}

\begin{proof} By Remark~\ref{connected} we may assume that $G$ is connected. Then, the reduced signed incidence matrix $M$ has strictly more columns than rows, and the fundamental matrix $\overline{M}$ has more copies of each row than each column. Suppose $v$ is adjacent to vertex $v'$ in $G$. Set $v'$ as the special vertex. We may therefore write $\overline{M} = \left[ \begin{array}{cc} A &\bf{0} \\ \bf{0} & B \end{array} \right]$, where rows in $A$ are indexed by $v$. As $A$ is not a square submatrix, by Corollary~\ref{pigeon} the permanent of $\overline{M}$ is zero. \end{proof}

In fact, we may generalize this to edge cuts and vertex cuts in a way similar to the max-flow min-cut theorem (see Theorem 11.3 in \cite{generalgraphtheory}). A \emph{separation} of a graph $G$ is a partition of $E(G)$ into sets $(A,B)$. We write $V(G_A)$ and $V(G_B)$ as the vertex sets of this partition, and define the \emph{order} of a separation as $|V(G_A) \cap V(G_B)|$.

\begin{theorem}\label{mymaxflow} Let $G$ be a graph. Suppose that in the computation of the extended graph permanent, edges receive weight $\E n$ and non-special vertices receive weight $\V n$. If $G$ contains a partition $(A,B)$ such that $V(G_A)$ contains $V_A$ non-special vertices and $|A|\E  > V_A\V$, then the extended graph permanent is zero at all primes. Similarly, if a $k$-edge cut in $G$ has $V_1$ non-special vertices on one side and $E_1$ edges in the subgraph induced by these vertices such that $ V_1 \V  -E_1 \E > k \E$, the extended graph permanent is zero for all primes. \end{theorem}

\begin{proof} Cofactor expansion removes an equal number of rows and columns, and hence reduces the total weight on vertices and edges by an equal amount. Consider the separation, then. We may perform cofactor expansion on columns and rows indexed by edges in $A$ and vertices in $V(G_A) - (V(G_A) \cap V(G_B))$ so that only vertices in $V(G_A) \cap V(G_B)$ and incident edges remain with nonzero weight. The associated matrices at this stage, then, would have a number of columns indexed by these edges exceeding the number of rows indexed by vertices incident to them. By Lemma~\ref{moregeneralpigeon}, such a permanent is zero. 

The proof for edge cuts is analogous.\end{proof}

Both Propositions~\ref{parallellimits} and~\ref{zeropendant} are immediate corollaries to  this theorem. 

\begin{example} The graph below has seven vertices and ten edges, so this technique gives all non-special vertices weight $5n$ and edges weight $3n$. $$ \raisebox{-.48\height}{\includegraphics[scale=0.8]{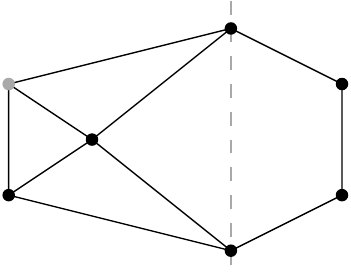}} $$Setting the vertex marked in grey as the special vertex, the marked $2$-vertex cut has total weight $10n$ and $10n$ on the vertices on the left. The edges on the left have total weight $21n$. While Theorem~\ref{mymaxflow} tells us immediately that this extended graph permanent is zero for all primes, we examine this by performing the calculation. Acting on the left and omitting previous factors for brevity, we compute
\begin{align*}
\left[ \raisebox{-.48\height}{\includegraphics[scale=0.7]{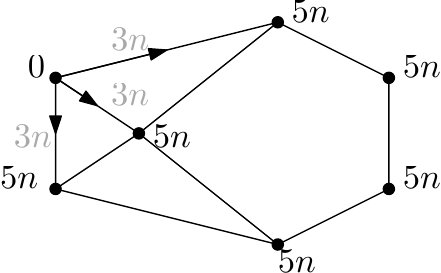}} \right] 
&= \left( \frac{(5n)!}{(2n)!} \right)^3 \left[  \raisebox{-.48\height}{\includegraphics[scale=0.7]{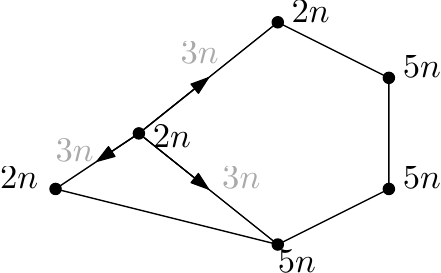}} \right] \\
&\hspace{-2cm}= \left( \frac{(5n)!}{(2n)!} \right)^3 \sum_{k_1=0}^{2n} \sum_{k_2 = 0}^{2n} \binom{3n}{k_1} \binom{3n}{k_2} \binom{3n}{2n-k_1-k_2} \left[  \raisebox{-.48\height}{\includegraphics[scale=0.7]{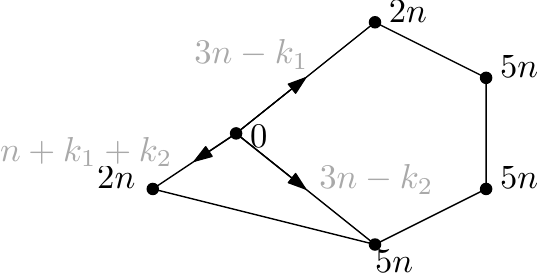}} \right] \\
&\hspace{-2cm}= \cdots \frac{(2n)!}{(n-k_1-k_2)!} \left[  \raisebox{-.48\height}{\includegraphics[scale=0.7]{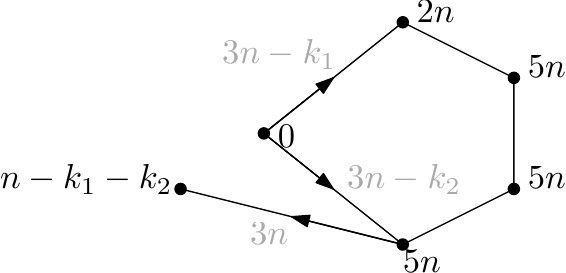}} \right] \\
&\hspace{-2cm}= \cdots \frac{(3n)!}{(2n+k_1+k_2)!} \left[  \raisebox{-.48\height}{\includegraphics[scale=0.7]{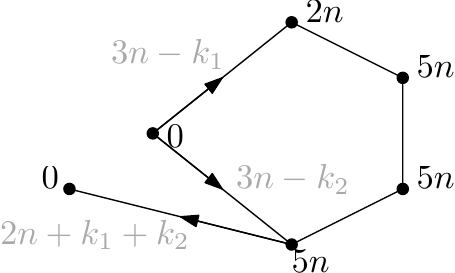}} \right] .
\end{align*}
At this stage of the computation, one of the  vertices in the cut is incident to a set of edges on the left with total weight exceeding its own. The permanent of the associated matrix must therefore be zero.
  \end{example}

\section{Wheels}\label{wheels} 

An important family of graphs is the wheels, built from cycles by adding an apex vertex. Consider a wheel with $w$ vertices in the outer cycle, call it $W_w$. While only $W_3$ and $W_4$ are $4$-point $\phi^4$ graphs, all have $|E(W_w)| = 2(|V(W_w)|-1)$, and hence have extended graph permanent sequences built over all odd primes. For wheel $W_w$, Broadhurst proved in \cite{Broadhurst} that the  Feynman period is $\binom{2w-2}{w-1}\zeta (2w-3)$. The $c_2$ invariant of  all wheels is $-1$ for all primes (\cite{BSModForms}, see Section~\ref{othercombc2} for a computational example), and the Hepp bounds for wheels $W_3$ and $W_4$ are $84$ and $572$, respectively (\cite{ErikEmail}).

For prime $2n+1$, 
\begin{align*} \left[ \raisebox{-.48\height}{\includegraphics{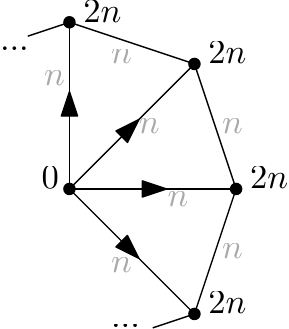}} \right] &= \left( \frac{(2n)!}{n!} \right)^w \left[  \raisebox{-.48\height}{\includegraphics{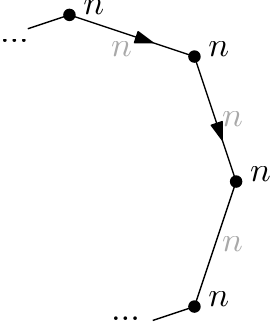}} \right] \\
&\hspace{-1.5cm}= \left( \frac{(2n)!}{n!}\right)^w \sum_{k=0}^n \binom{n}{k} \binom{n}{n-k} n!(-1)^k  \left[  \raisebox{-.48\height}{\includegraphics{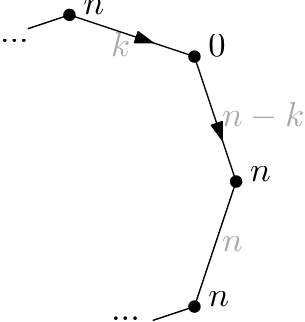}} \right] \\
&\hspace{-1.5cm}=  \left( \frac{(2n)!}{n!} \right)^w \sum_{k=0}^n (-1)^n \binom{n}{k}^2 n! \frac{n!}{k!} \frac{n!}{(n-k)!} \left[ \hspace{-1mm} \begin{array}{c} w-1 \\ \text{vertices} \end{array} \hspace{-1mm} \left\{  \raisebox{-.48\height}{\includegraphics{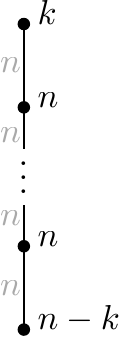}}  \right.  \right] . \\
&\hspace{-1.5cm}=  \left( \frac{(2n)!}{n!} \right)^w \sum_{k=0}^n (-1)^n \binom{n}{k}^3 n!^2 \left[ w-1 \left\{  \raisebox{-.48\height}{\includegraphics{wheel3}}  \right.  \right] .
\end{align*}

We pause in this calculation to consider the permanent of the path created. We will orient all edges away from the middle of the path. We then get 
\begin{align*}
\left[ w-1 \left\{  \raisebox{-.48\height}{\includegraphics{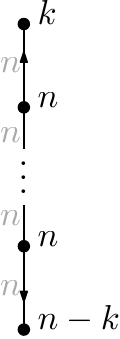}} \right. \right] 
&=  \frac{n!}{(n-k)!} \frac{n!}{k!} \left[ w-1 \left\{  \raisebox{-.48\height}{\includegraphics{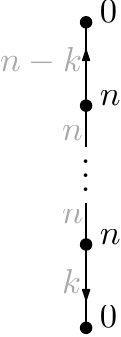}} \right. \right] \\
&\hspace{-.6cm}=  \binom{n}{k}^2 n!^2 (-1)^{n-k}(-1)^k \left[ w-3 \left\{  \raisebox{-.48\height}{\includegraphics{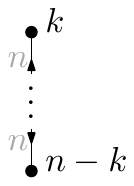}} \right. \right] \\
&\hspace{.75cm} \vdots \\
&\hspace{-.6cm}=  \binom{n}{k}^{w-3}  n!^{w-3} (-1)^{n \lfloor \frac{w-3}{2}  \rfloor}
\begin{cases} \left[  \raisebox{-.48\height}{\includegraphics{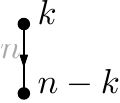}} \right] \vspace{1mm} & \text{ if $w-1$ is even} \\  
\left[  \raisebox{-.48\height}{\includegraphics{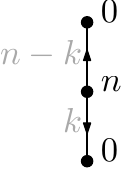}} \right] & \text{ if $w-1$ is odd}
\end{cases} \\
&\hspace{-.6cm}=   \binom{n}{k}^{w-3}  n!^{w-3} (-1)^{n \lfloor \frac{w-3}{2}  \rfloor}
\begin{cases}(-1)^k n! & \text{ if $w-1$ is even} \\  (-1)^n n! & \text{ if $w-1$ is odd}
\end{cases}.
\end{align*}

Continuing with the original calculation,
\begin{align*}
 \left[ \raisebox{-.48\height}{\includegraphics{wheel1}} \right] 
&=  \left( \frac{(2n)!}{n!} \right)^w \sum_{k=0}^n (-1)^n \binom{n}{k}^3 n!^2 \left[ w-1 \left\{  \raisebox{-.48\height}{\includegraphics{wheel3}}  \right.  \right]\\
&\hspace{-1.2cm}=  \left( \frac{(2n)!}{n!} \right)^w \sum_{k=0}^n (-1)^n \binom{n}{k}^w n!^w \cdot \begin{cases} (-1)^{k+n\lfloor \frac{w-3}{2}\rfloor} \text{ if $w$ is odd} \\ (-1)^{n+n\lfloor \frac{w-3}{2}\rfloor} \text{ if $w$ is even}\end{cases} \\
&\hspace{-1.2cm}= \begin{cases} (2n)!^w (-1)^{n\lfloor \frac{w-1}{2} \rfloor} \sum_{k=0}^n (-1)^k \binom{n}{k}^w \text{ if $w$ is odd} \\ (2n)!^w (-1)^{n \lfloor \frac{w+1}{2} \rfloor}  \sum_{k=0}^n \binom{n}{k}^w \text{ if $w$ is even} \end{cases} \\
&\hspace{-1.2cm}= (2n)!^w (-1)^{n\lfloor \frac{w+(-1)^w}{2} \rfloor} \sum_{k=0}^n (-1)^{kw}\binom{n}{k}^w . 
\end{align*}

As a factor $(-1)^n$ corresponds to reversing the direction of the $n$ columns corresponding to a common edge, and as $2n! \equiv -1 \pmod{2n+1}$ by Wilson's Theorem, we may write this as $(-1)^{w} \sum_{k=0}^n (-1)^{kw}\binom{n}{k}^w  \pmod{2n+1}$.

It is interesting to note that, if $w$ is odd and $2n+1$ is congruent to three modulo four, then \begin{align*}  \sum_{k=0}^n (-1)^k \binom{n}{k}^w &= \binom{n}{0}^w - \binom{n}{1}^w + \cdots - \binom{n}{n}^w \\ &= \sum_{k=0}^{(n-1)/2} \left( \binom{n}{k}^w - \binom{n}{n-k}^w \right) =0. \end{align*} This vanishing permanent property generalizes to graphs with a particular symmetry. To prove this generalized equality, we use the graphical interpretation of the permanent calculation from Section~\ref{graphicegp}.

\begin{definition} Let $G$ be a graph. If graph automorphism $\tau$ has $\tau(\tau(v)) =v$ for all $v \in V(G)$, then $\tau$ is an \emph{involution}. For a particular involution $\tau$, we will say that an edge $e =uv$ is \emph{crossing} if $\tau(u)=v$. \end{definition}

\begin{theorem} \label{symmetry} Suppose $G$ is a graph. If there is an involution $\tau$ with an odd number of crossing edges and at least one vertex $v$ fixed by $\tau$, then the permanent of the associated fundamental matrix for $G$ with special vertex $v$ is identically zero. \end{theorem}

\begin{proof} For non-crossing edges $e=uv$, orient such that the involution preserves the orientation; if $e = (u,v)$ then $\tau(u)\tau(v)= (\tau(u),\tau(v))$.  Finally, orient the crossing edges arbitrarily.

Valid edge colourings and taggings are preserved by the automorphism. As sign changes occur only when an odd number of tags change direction, and as there are an odd number of crossing edges, this orientation will produce a sign change when we map between colourings and taggins using this automorphism. We may therefore partition all taggings into two sets by fixing a crossing edge and dividing the taggings based on which vertex incident to this edge received the tag. As the involution provides an obvious bijection between sets and in particular elements with opposite signs, the sum is zero. \end{proof}

As the previous theorem holds if $G = G'^{[k]}$ for some $k$, we get the following corollary.

\begin{corollary} \label{symmetrycor} If a decompleted $4$-regular graph $G$ has an involution as described in Theorem~\ref{symmetry}, then the permanent of the fundamental matrix associated to $G^{[k]}$ for odd $k$ has permanent zero. \end{corollary}

\begin{proof} If $G$ has an odd number of crossing edges, then so does $G^{[k]}$ for odd $k$, and the result follows. \end{proof}

All wheels $W_k$ for odd $k$ have such involutions, shown in Figure~\ref{wheelinvolution}. Computing up to prime $p=4999$, this explains all zeros in the extended graph permanent of $W_3$. Wheel $W_5$ has $\text{GPerm}^{[5]}(W_5) \equiv 0 \pmod{5}$, though the actual permanent of the associated matrix is nonzero, so this does not explain all zeros up to residues.

\begin{figure}[h]
  \centering
      \includegraphics[scale=0.8]{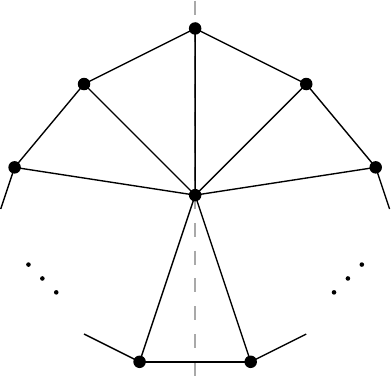}
  \caption{An involution with an odd number of crossing edges for an arbitrary wheel with an odd number of spokes.}
\label{wheelinvolution}
\end{figure}

The graph shown in Figure~\ref{p711}, a decompletion of $P_{7,11}$ (following the naming conventions from \cite{Sphi4}), is a counterexample to the converse of Theorem~\ref{symmetry}. The figure is drawn so that the marked symmetry captures the only element of the symmetric group, which has an even number of crossing edges and no fixed vertices. Choosing the grey vertex as special, the permanent of the signed incidence $2$-matrix is equal to zero in $\mathbb{R}$ (that is, the computed permanent at prime $p=3$). Up to prime $p=199$, this is the only identically zero permanent in the extended graph permanent sequence for this graph. A closed form of the permanents in this sequence is \small $$ (2n)!^7 \sum_{x_i = 0}^n
\binom{n}{ x_0 }^2
\binom{n}{ x_1 }^2
\binom{n}{ x_2 }
\binom{n}{ x_3 }
\binom{n}{ 2n - x_2 - x_3 }^2
\binom{n}{ -n + x_0 + x_1 + x_2 }
\binom{n}{ -x_0 + x_3 }(-1)^{ x_1  }  .$$ \normalsize

\begin{figure}[h]
  \centering
      \includegraphics[scale=1.0]{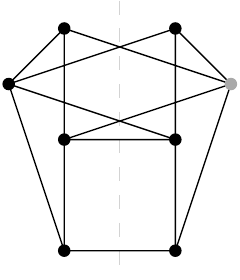}
  \caption{A graph that provides a counterexample to the converse of Theorem~\ref{symmetry}.}
\label{p711}
\end{figure}

Consider the wheel $W_3 = K_4$, which is also the decompletion of $P_{3,1}$ using the naming convention from \cite{Sphi4}. Let $P_{3,1}^2$ be the unique graph made by joining two copies of this graph, as in Figure~\ref{2cut}. By Theorem~\ref{2vertexcut} the extended graph permanent of this graph is equal to the additive inverse of the term-by-term square of $K_4$'s sequence. Hence, the sequence begins $$( 0,4, 0, 0, 4, 1 ,0,0,5,0,21,4,0,0,9,...   ).$$ It is an interesting observation that this appears to be the same as the $c_2$ invariant for graphs $P_{9,161}$, $P_{9,170}$, $P_{9,183}$, and $P_{9,185}$ computed by Brown and Schnetz in \cite{BSModForms}, shown in their completed forms in Figure~\ref{p9161}. No immediate connection between the two graphs is apparent, nor has a relationship between the $c_2$ invariant and the extended graph permanent been revealed. The only other common sequences found between $c_2$ invariants and extended graph permanents for $4$-point $\phi^4$ graphs are the trivial sequences, always $0$ or always $-1$. The graph $P_{1,1}$ has extended graph permanent equal to the $c_2$ invariant of any decompletion of $P_{3,1}$; all entries $-1$. Further, the $c_2$ invariant is preserved by the double triangle operation (Corollary 34 in \cite{BrS}), shown in Figure~\ref{dubtri}. Hence, the extended graph permanent of $P_{1,1}$ is equal to the $c_2$ invariant of all double triangle descendants of $P_{3,1}$. Additionally, there are numerous ways in which a $4$-point $\phi^4$ graph can have $c_2$ invariant equal to zero at all primes, Proposition~\ref{2cutsinc2} and Theorem~\ref{subdivc2} provide examples of this. Theorem~\ref{mymaxflow} provides numerous graphs for which the extended graph permanent is equal to zero at all primes.

\begin{figure}[h]
  \centering
      \includegraphics[scale=0.8]{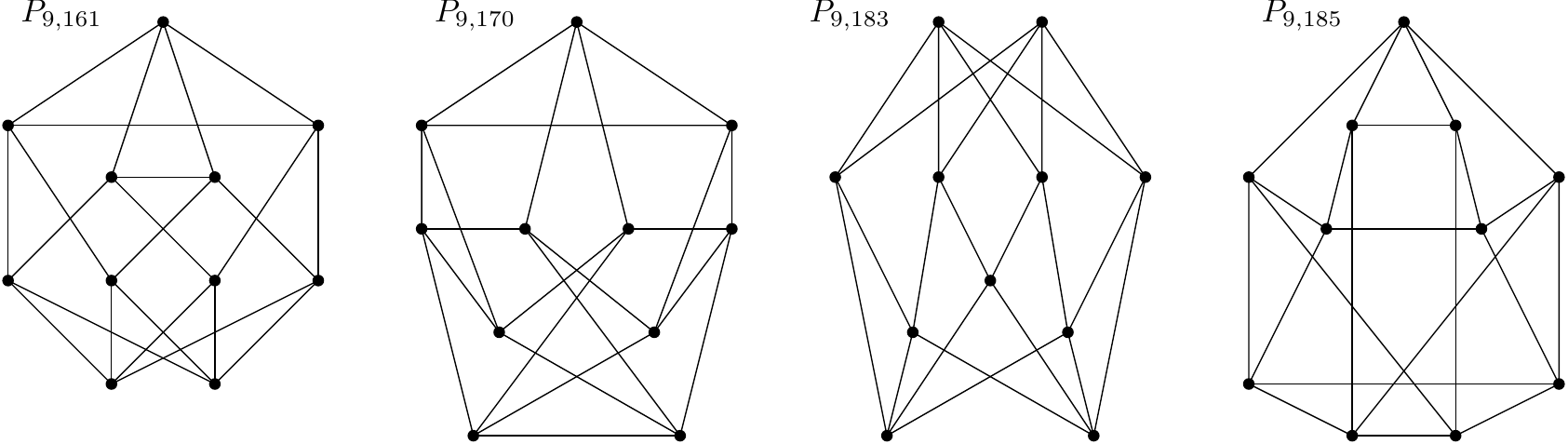}
  \caption{Four graphs that, when decompleted, have $c_2$ invariant that appears to be equal to the extended graph permanent of $P_{3,1}^2$.}
\label{p9161}
\end{figure}

\begin{figure}[h]
  \centering
      \includegraphics[scale=1.0]{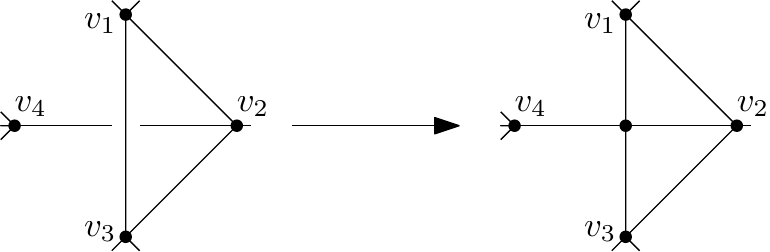}
  \caption{The double triangle operation, which is known to preserve the $c_2$ invariant.}
\label{dubtri}
\end{figure}

\section{Zig-zag graphs}\label{zigzagcomp}

The zig-zag graphs are an important family in $\phi^4$ theory.  Graphically, it is a family whose completions are certain \emph{circulant graphs}. The circulant graph $C^m_{a,b}$ is defined by vertex and edge sets $$V(C^m_{a,b}) = \{0,...,m-1\}, \hspace{2mm} E(C^m_{a,b}) = \left\{ \{i,j\} : |i-j| \in \{a,b\pmod{m}\}  \right\}.$$ The zig-zag graphs are the family $C^m_{1,2}$ for integer $m\geq 5$. For the zig-zag graph on $m$ vertices, Brown and Schnetz prove in \cite{BSZigzag} that the Feynman period is $\frac{4(2m-2)!}{m!(m-1)!}\left( 1- \frac{1-(-1)^m}{2^{2m-3}} \right) \zeta (2m-3)$. The $c_2$ invariant for all zig-zag graphs is $-1$ for all primes (\cite{BrS,BSModForms}, following from the above double triangle operation), and the Hepp bound for the first nine zig-zag graphs are $84$, $572$, $3703$, $26220$, $190952$, $\frac{4290568}{3}$, $10927146$, $84859647$, and $667807932$ (\cite{ErikEmail}).

Consider the zig-zag graph on $m$ vertices, $m \geq 4$, as seen in Figure~\ref{zarb02}. We will take the right-most vertex as the special vertex, and for the sake of future row reduction use the edges highlighted as the first $m-1$ columns in the signed incidence matrix. As such, our signed incidence matrix is 
\begin{align*} \left[
\begin{array}{cccccc|ccccccc}
1&0 &0 & &0 &0 &1 &0 &0 &0 & &0 &1 \\
-1&1 &0 & &0 &0 &0 &1 &0 &0 && 0 &0 \\
0&-1 &1 & &0 &0 &-1 &0 &1 &0 & &0 &0 \\
0&0 &-1 & &0 &0 &0 &-1 &0 &1 & &0 &0 \\
& & &\ddots & & & & & & &\ddots & & \\
0&0 &0 & &1 &0 &0 &0 &0 &0 & &1 &0 \\
0&0 &0 & &-1 &1 &0 &0 &0 &0 & &0 &0
\end{array} \right]_{m-1,2(m-1)}. \end{align*}
We may reduce this matrix, since we will be taking the permanent modulo $2n+1$. Hence, it reduces to
\begin{align*} \left[
\begin{array}{cccccc|ccccccc}
1&0 &0 & &0 &0 &1 &0 &0 &0 & &0 &1 \\
0&1 &0 & &0 &0 &1 &1 &0 &0 && 0 &1 \\
0&0 &1 & &0 &0 &0 &1 &1 &0 & &0 &1 \\
0&0 &0 & &0 &0 &0 &0 &1 &1 & &0 &1 \\
& & &\ddots & & & & & & &\ddots & & \\
0&0 &0 & &1 &0 &0 &0 &0 &0 & &1 &1 \\
0&0 &0 & &0 &1 &0 &0 &0 &0 & &1 &1
\end{array} \right]. \end{align*}

Label this right block $A$. Then, the matrix used for prime $2n+1$ in the extended graph permanent is $\mathbf{1}_{2n \times n} \otimes \left[ I_{m-1} | A \right]$. Cofactor expansion along the columns in the identity matrix gives $$ \text{Perm} (\mathbf{1}_{2n \times n} \otimes \left[ I_{m-1} | A \right]) = \left( \frac{(2n)!}{n!} \right)^{m-1} \cdot \text{Perm}(\mathbf{1}_n \otimes A) .$$

\begin{figure}[h]
  \centering
      \includegraphics[scale=0.90]{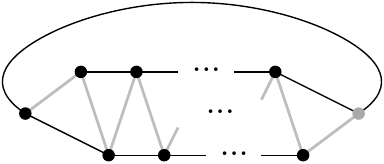}
  \caption{The general zig-zag graph, used to build the signed incidence matrix.}
\label{zarb02}
\end{figure}

What we see with this matrix $A$ is familiar; it is the incidence matrix of an undirected path on $m-1$ vertices, with one additional hyperedge (see \cite{banddagain} for an introduction to hypergraphs and hyperedges) that meets all vertices. In our terms, all edges and vertices receive weight $n$, including the hyperedge. Using cofactor expansion along the column corresponding to the hyperedge followed by our usual tricks;

\begin{align*}
\text{Perm}(I_n \otimes A) &= \sum_{\substack{k_1+\cdots + k_{m-1} = n \\ k_i \geq 0}} \binom{n}{k_1} \cdots \binom{n}{k_{m-1}} n! \left[ \raisebox{-.48\height}{\includegraphics[scale=.7]{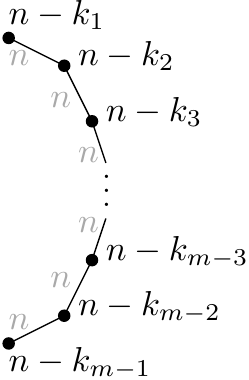}} \right], \end{align*}
where
\begin{align*} \left[ \raisebox{-.48\height}{\includegraphics[scale=.7]{zarb03}} \right] &= \frac{n!}{k_1!} \left[ \raisebox{-.48\height}{\includegraphics[scale=.7]{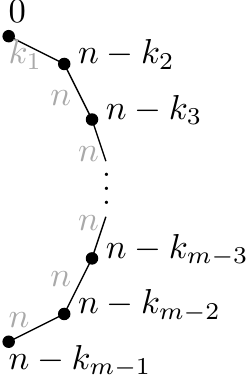}} \right] \\
&\hspace{-.3cm}= \frac{n!}{k_1!} \frac{(n-k_2)!}{(n-k_1-k_2)!} \left[ \raisebox{-.48\height}{\includegraphics[scale=.7]{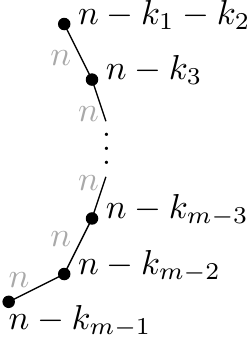}} \right] \\
&\hspace{-.3cm}= \frac{n!}{k_1!} \frac{(n-k_2)!}{(n-k_1-k_2)!} \frac{n!}{(k_1+k_2)!} \left[ \raisebox{-.48\height}{\includegraphics[scale=.7]{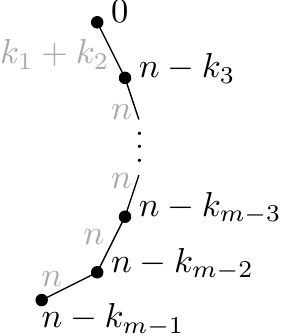}} \right] \\
&\hspace{-.3cm}= \frac{n!}{k_1!} \frac{(n-k_2)!}{(n-k_1-k_2)!} \frac{n!}{(k_1+k_2)!}\frac{(n-k_3)!}{(n-k_1-k_2-k_3)!} \left[ \raisebox{-.48\height}{\includegraphics[scale=.7]{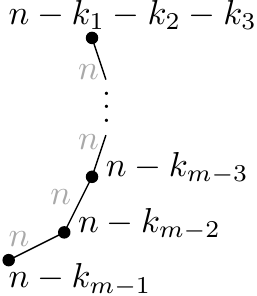}} \right] \\
&\vdots \\
&\hspace{-.3cm}= \frac{n!}{k_1!} \frac{(n-k_2)!}{(n-k_1-k_2)!} \cdots \\
&\hspace{1cm} \frac{n!}{(k_1+ \cdots + k_{m-3})!} \frac{(n-k_{m-2})!}{(n-k_1-\cdots - k_{m-2})!}\left[ \raisebox{-.48\height}{\includegraphics[scale=.6]{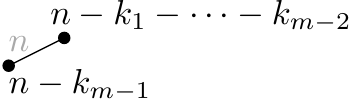}} \right] \\
&\hspace{-.3cm}= \frac{n!}{k_1!} \frac{(n-k_2)!}{(n-k_1-k_2)!} \cdots \frac{n!}{(k_1+ \cdots k_{m-3})!} \frac{(n-k_{m-2})!}{(n-k_1-\cdots - k_{m-2})!}n!\\
&\hspace{-.3cm}= (n!)^{m-2} \binom{n-k_2}{k_1} \binom{n-k_3}{k_1+k_2} \cdots \binom{n-k_{m-2}}{k_1+ \cdots + k_{m-3}}.
\end{align*}

Hence, 
\begin{align*}
 \text{Perm} (\mathbf{1}_{2n \times n} \otimes \left[ I_{m-1} | A \right]) & \\
&\hspace{-3.4cm}= \left( \frac{(2n)!}{n!} \right)^{m-1} \sum_{\substack{k_1+\cdots + k_{m-1} = n \\ k_i \geq 0}} n!^{m-1} \left( \prod_{i=1}^{m-1} \binom{n}{k_i} \prod_{i=1}^{m-3} \binom{n-k_{i+1}}{\sum_{j=1}^i k_j} \right) \\
&\hspace{-3.4cm}= (2n)!^{m-1} \sum_{\substack{k_1+\cdots + k_{m-1} = n \\ k_i \geq 0}}  \left( \prod_{i=1}^{m-1} \binom{n}{k_i} \prod_{i=1}^{m-3} \binom{n-k_{i+1}}{\sum_{j=1}^i k_j} \right)\\
&\hspace{-3.4cm} \equiv (-1)^{m-1} \sum_{\substack{k_1+\cdots + k_{m-1} = n \\ k_i \geq 0}} \left( \prod_{i=1}^{m-1} \binom{n}{k_i} \prod_{i=1}^{m-3} \binom{n-k_{i+1}}{\sum_{j=1}^i k_j} \right) \pmod{2n+1}.
\end{align*}

\section{Matroid $R_{10}$}\label{compr10}

We include now the computation of a permanent from a non-graphic matroid. The matroid $R_{10}$ is of interest as it demonstrates extending the methods to non-graphic objects, because the sequence produced is familiar, and because $R_{10}$ is special in the class of regular matroids, as noted in Chapter~\ref{chmatroid}.

The regular matroid $R_{10}$ can be represented by the matrix \begin{align*}
\left[
\begin{array}{ccccc|ccccc}
1 &0 &0 &0 &0 &-1 &1 &0 &0 &1 \\
0 &1 &0 &0 &0 &1 &-1 &1 &0 &0 \\
0 &0 &1 &0 &0 &0 &1 &-1 &1 &0 \\
0 &0 &0 &1 &0 &0 &0 &1 &-1 &1 \\
0 &0 &0 &0 &1 &1 &0 &0 &1 &-1 
\end{array}
\right]
\end{align*}
over every finite field. Note that this is a $5 \times 10$ matrix, and constructing the fundamental matrix and further block matrix extensions, the permanent is defined, like $4$-point $\phi^4$ graphs, on all odd primes. As before, denote the right block of the previous matrix as $A$. Hence, for prime $2n+1$ we compute $$\text{Perm}( \mathbf{1}_{2n \times n} \otimes \left[ I_5|A \right]) = \left( \frac{(2n)!}{n!} \right)^5 \text{Perm}(\mathbf{1}_n \otimes A).$$

We turn this matrix $A$ into the incidence matrix of a graph; in this case a hypergraph on five vertices with five hyperedges, each vertex and hyperedge receiving weight $n$. Shown in Figure~\ref{r10}, and drawn with vertices distributed in a cycle, note that the value associated to the middle-most vertex of each hyperedge is $-1$, while all others receive value $1$. For notational convenience, we will favour the right drawing. Additionally, the following lemma will be useful.

\begin{figure}[h]
  \centering
      \includegraphics[scale=0.90]{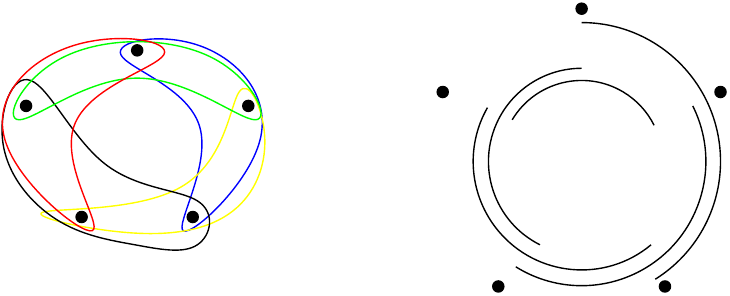}
  \caption{The hypergraph associated to $R_{10}$.}
\label{r10}
\end{figure}

\begin{lemma}\label{oneblocks} If $a+b = c+d$, $$\text{Perm} \left[ \begin{array}{c|c} \mathbf{-1}_{a \times c} & \mathbf{1}_{a \times d} \\ \mathbf{1}_{b \times c} & \mathbf{-1}_{b \times d}  \end{array} \right] = (a+b)!(-1)^{a+d}.$$ \end{lemma}

\begin{proof} Matrix $\left[ \mathbf{1}_{a+b} \right]$ has permanent $(a+b)!$. By Remark~\ref{rowops}, multiplying the first $a$ rows and last $d$ columns by $-1$ produces the desired matrix, and completes the proof. \end{proof}

Whence, turning hyperedges into directed edges when they are incident to precisely two vertices, \begin{align*}
\text{Perm}( \mathbf{1}_{2n \times n} \otimes \left[ I_5|A \right]) 
&= \left( \frac{(2n)!}{n!} \right)^5 \left[ \raisebox{-.48\height}{\includegraphics[scale=.6]{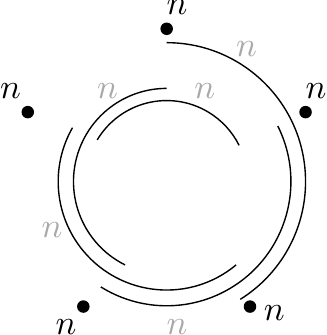}} \right] \\
&= \left( \frac{(2n)!}{n!} \right)^5 \sum_{k_1+k_2+k_3=n} \binom{n}{k_1} \binom{n}{k_2} \binom{n}{k_3} n! (-1)^{k_2} \left[ \raisebox{-.48\height}{\includegraphics[scale=.6]{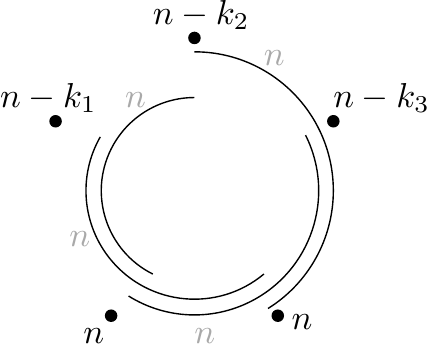}} \right] \\
&= \frac{(2n)!^5}{n!^4} \sum_{k_1+k_2+k_3=n} \left(\prod_{i=1}^3 \binom{n}{k_i}\right) (-1)^{k_2} \cdot \\
&\hspace{2cm} \sum_{j_1+j_2=n-k_2} \binom{n}{j_1} \binom{n}{j_2}(n-k_2)! \left[ \raisebox{-.48\height}{\includegraphics[scale=.6]{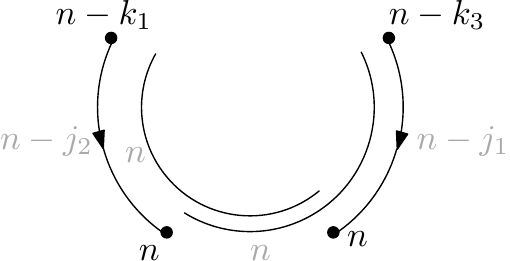}} \right], \\
\left[ \raisebox{-.48\height}{\includegraphics[scale=.6]{r103}} \right] 
&= \sum_{l_1+l_2 = n-k_1} \binom{n-j_2}{l_1} \binom{n}{l_2}(-1)^{l_1}(n-k_1)! \left[ \raisebox{-.48\height}{\includegraphics[scale=.6]{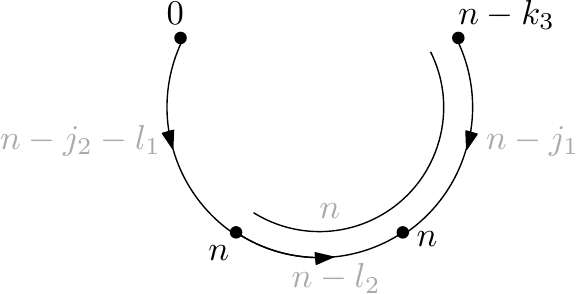}} \right] \\
&= \sum_{l_1+l_2 = n-k_1} \binom{n-j_2}{l_1} \binom{n}{l_2}(-1)^{l_1} \frac{(n-k_1)!n!}{(j_2+l_1)!} \left[ \raisebox{-.48\height}{\includegraphics[scale=.6]{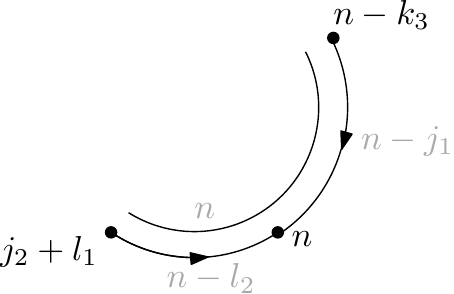}} \right], \\
\left[ \raisebox{-.48\height}{\includegraphics[scale=.6]{r105}} \right] 
&= \sum_{p_1+p_2 = n-k_3} \binom{n-j_1}{p_1} \binom{n}{p_2}(n-k_3)!(-1)^{p_1} \left[ \raisebox{-.48\height}{\includegraphics[scale=.6]{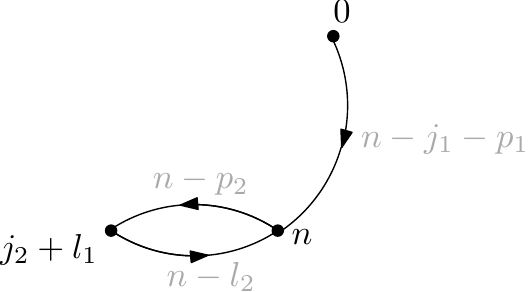}} \right] \\
&= \sum_{p_1+p_2 = n-k_3} \binom{n-j_1}{p_1} \binom{n}{p_2} \frac{(n-k_3)! (-1)^{p_1}n!}{(j_1+p_1)!} \left[ \raisebox{-.48\height}{\includegraphics[scale=.6]{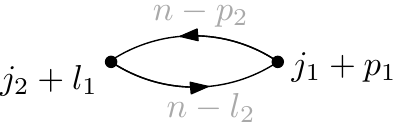}} \right] \\
& \hspace{-1.3cm}= \sum_{p_1+p_2 = n-k_3} \binom{n-j_1}{p_1} \binom{n}{p_2} \frac{(n-k_3)! (-1)^{p_1}n!}{(j_1+p_1)!} (2n-l_2-p_2)!(-1)^{n-l_2+j_1+p_1}.
\end{align*}

Combining and simplifying slightly; \begin{align*}
\text{Perm}( \mathbf{1}_{2n \times n} \otimes \left[ I_5|A \right]) 
&= \frac{(2n)!^5}{n!^2} \sum_{k_1+k_2+k_3=n} \left(\prod_{i=1}^3 \binom{n}{k_i}\right) \\
&\hspace{2cm} \cdot \sum_{j_1+j_2=n-k_2} \binom{n}{j_1} \binom{n}{j_2}(n-k_2)! \\
&\hspace{2cm} \cdot \sum_{l_1+l_2 = n-k_1} \binom{n-j_2}{l_1} \binom{n}{l_2} \frac{(n-k_1)!}{(j_2+l_1)!} \\
&\hspace{2cm} \cdot \sum_{p_1+p_2 = n-k_3} \binom{n-j_1}{p_1} \binom{n}{p_2} \frac{(n-k_3)!}{(j_1+p_1)!} \\
&\hspace{2cm} \cdot (2n-l_2-p_2)!(-1)^{k_1+k_2+j_1}.
\end{align*}

Computing, the first few terms of the extended graph permanent of $R_{10}$ are $$(0,4,0,0,4,1,0,0,5,0,21,4,0,0,9,0,4,0,...).$$ Here, we again see a sequence that appears to be identical to the extended graph permanent of $P_{3,1}^2$, and to the $c_2$ invariant of decompletions of  $P_{9,161}$, $P_{9,170}$, $P_{9,183}$, and $P_{9,185}$. Sequences for $R_{10}$ and $P_{3,1}^2$ have been compared up to prime $p=151$, resulting in the computing of the permanent of a $750 \times 750$ matrix for $R_{10}$. The $c_2$ invariant sequences are unfortunately only readily available up to prime $p=37$.

\section{Computational Simplicity}
\label{niceness} 

In this section, we expressly forbid graphs with loops or vertices with precisely one neighbour.

It is possible to simplify the techniques introduced in this chapter to produce these closed forms in a more algorithmic way for individual graphs. This is due to the cancellation of terms that occurs, and an algorithmic way of gathering the factorials into binomials. We present this method here.

In this section, we will say that we \emph{act on} a vertex when we perform cofactor expansion on the collection of rows corresponding to that vertex. Similarly, we will say we act on an edge by performing cofactor expansion on the associated set of columns. We will similarly say we \emph{move to} a vertex or edge when we act on an incident object.

Let $G$ be a graph. Let $L = \lcm (|V(G)|-1, |E(G)|)$, $\V = \frac{L}{|V(G)|-1}$, and $\E = \frac{L}{|E(G)|}$, so that for prime $\V n+1$ vertices receive weight $\V n$ and edges receive weight $\E n$. Suppose we produce a closed form for $G$ by first acting on the edges incident to the special vertex, and then acting on a set of vertices, and their incident edges, such that the deletion of these vertices produces a tree. Finally, act on vertices of degree one and incident edges until the computation is complete.

For each non-special vertex, we therefore produce a factor of $(\V n)!$ in the numerator. Further, if vertex $v \in V(G)$ is not a vertex that is acted on initially in the production of a tree, most factors produced for this vertex cancel; every time we move to $v$ from an incident edge we produce a factorial in the numerator equal to the current weight, and a factorial in the denominator equal to the weight after this action. Thus, repeated actions telescope, and only the first weight remains in the numerator, and only the last weight remains in the denominator.

For edges, this holds true also. Each edge in the initial set of actions on vertices produces a binomial in $\E n$. All remaining edges will produce a factor of $(\E n)!$ in the numerator, and $(\E n - w_v)!$ in the denominator, where $w_v$ is the weight of the incident vertex. It follows then that these terms may be gathered into further binomials in $\E n$, as there are terms $\frac{(\E n)!}{w_v! (\E n - w_v)!} $.

As such, the graph produces a closed form that contains only factors $(\V n)!^{|V(G)|-1}$, summations from $0$ to $\E n$, a factor of $-1$ to some power, and a set of binomials in $\E n$, the number of these equal to the number of edges not incident to the special vertex. Further, these binomials come in two distinct flavours; those coming from an initial action on a vertex prior to establishing a tree, and those of the form $\binom{\E n}{w_v}$ where $w_v$ is the last weight on a vertex.

\begin{example}  Consider the wheel on four spokes, $W_4$. We saw in Section~\ref{wheels} that setting the apex vertex as the special vertex, we produce closed form $(2n)!^4\sum_{x=0}^n \binom{n}{x}^4$. Using the method described in this section, we can compute this using the following sequence. We include the variable distribution corresponding to acting on the vertex of degree two for clarity, and colour the vertex being acted on grey.
$$\includegraphics[scale=.6]{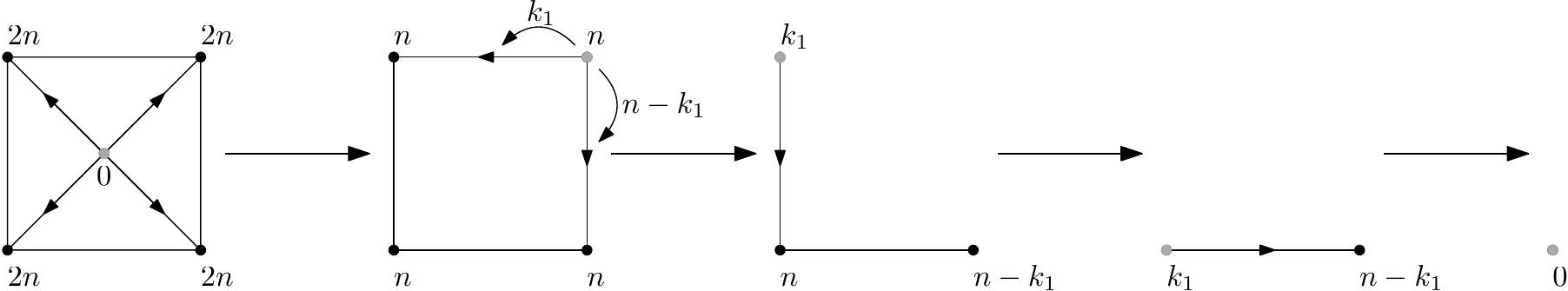}$$
This first action of course produces the term $(2n)!^4$, the second $\sum_{k_1=0}^n \binom{n}{k_1}^2(-1)^n$, and the third and forth steps additional terms $\binom{n}{k_1}^2(-1)^{2k_1}$. This is the desired form, up to a term $(-1)^n$ corresponding to changing the direction of a single edge.

If we instead choose a different vertex as the special vertex, we may use the following sequence.
$$\includegraphics[scale=.6]{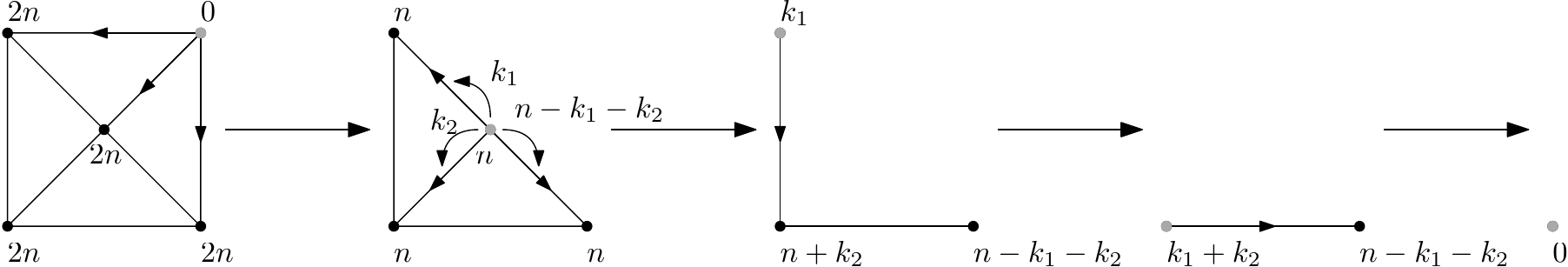}$$
Simplifying, this gives closed form $$(2n)!^4 \sum_{k_1 = 0}^n \sum_{k_2=0}^n \binom{n}{k_1}^2 \binom{n}{k_2} \binom{n}{k_1+k_2}^2 (-1)^{k_2 + n}.$$ While these forms do not appear similar, and certainly are not equal prior to taking residues, we know from the previous work that the residues must agree modulo $2n+1$. Traditional methods, such as the Wilf-Zeilberger algorithm (see for example Section 3.1 in \cite{wz}), generally will not work in explaining equalities such as this, as here we are comparing summations modulo $p$, and the permanent itself is affected by choice of special vertex. \end{example}

The following question is of interest, then, as it may lead to understanding when sequences for non-isomorphic graphs are equal.

\begin{question}\label{equationreconstruction}If we produce a closed form for the extended graph permanent in this manner and perform no non-trivial simplifications, what graphs can be uniquely reconstructed from the equation? If we allow simplification, is there a family of graphs for which this closed form can still be used to uniquely reconstruct the graph?  \end{question}

Further, this method produces computationally easy formulas, though not necessarily the fastest. Computations using methods from this chapter may require careful consideration of the bounds in the summations, as otherwise negative factorial term may appear and break computer calculations. A binomial of the form $\binom{n}{x}$ for $x<0$ will return zero, and the computation will proceed as desired. Hence, closed forms produced using this method are immediately computer ready. Closed forms for primitive $4$-points $\phi^4$ graphs up to seven loops are included in Appendix~\ref{chartofequations}. All were produced using this method.

\chapter{The Extended Graph Permanent as an Affine Hypersurface}
\label{hyperpower}

As mentioned in Chapter~\ref{egpcomp}, the extended graph permanent of $P_{3,1}^2$ appears to be equal to the $c_2$ invariant of graphs $P_{9,161}$, $P_{9,170}$, $P_{9, 183}$, and $P_{9, 185}$, seen in \cite{BSModForms}. As the $c_2$ invariant is constructed from a point count, Brown asked if the extended graph permanent could be expressed as the point count of a polynomial, too\footnote{Special thanks to Dr. Francis Brown for posing this question, as well as directing me to the Chevalley-Warning Theorem and Theorem~\ref{chevwarn}.}. This could potentially open other approaches to understanding the extended graph permanent, and possibly even establish a connection between the $c_2$ invariant and the extended graph permanent.

Let $\mathbb{F}$ be a field. For polynomial $f \in \mathbb{F}[x_1,...,x_n]$, the \emph{affine hypersurface} of $f$ is $$ \left\{(a_1,...,a_n) \in \mathbb{F}^n : f(a_1,...,a_n) = 0 \right\} .$$ In this section, we construct a polynomial from the graph such that, for prime $p$ where the extended graph permanent is defined, the cardinality of the affine hypersurface over $\mathbb{F}_p$ is equal to the extended graph permanent modulo $p$, up to a constant multiplicative factor.

\section{A novel graph polynomial}

\begin{definition} Let $F(x_1,...,x_n)$ be a polynomial with integer coefficients and $q = p^\alpha$ for some prime $p$. We define the \emph{point count} of $F$ over $\mathbb{F}_q$ to be the number of solutions to $F(x_1,...,x_n) =0$ in $\mathbb{F}_q$, and denote it $[F]_q$. Note that the point count is the cardinality of the affine hypersurface over that field. \end{definition} 

We begin with a known method of turning the computation of the permanent into coefficient extraction of a polynomial. For a variable $x$, we denote the coefficient of $x$ in function $f$ as $[x]f$. Multivariate coefficient extraction follows as expected.

\begin{definition} Let $A=[a_{ij}]$ be an $n \times n$ matrix with integer entries. Define $$F_A(x_1, ... , x_n) = \prod_{i = 1}^n\sum_{j= 1}^n a_{ij}x_j. $$Then, $ \text{Perm}(A) = [x_1 \cdots x_n] F_A.$ This follows immediately from the equation for the permanent seen in Definition~\ref{permanentdef}. We will call this the \emph{permanent polynomial}. \end{definition} 

Note that for a graph $G$ with $|E(G)| = (p-1)(|V(G)|-1)$ and an associated fundamental matrix $\overline{M}_G$, this is the same as the polynomial $g$ in the proof of Theorem~\ref{makingorientations}.

This function then gives an alternate method for expressing the permanent as the coefficient, but given our desire to compute permanents for matrices $\overline{M}$ and $\mathbf{1}_k \otimes \overline{M}$, we would require a unique function to compute the permanent of each matrix. The goal of this section is to find one polynomial that can be used to compute the permanent for all such matrices. Given the block matrix construction of these matrices, though, we may construct subsequent functions from the permanent polynomial of the fundamental matrix.

\begin{definition} For polynomial $f = a_1 x_1 + \cdots + a_ n x_n$, define the \emph{$r^\text{th}$ extension of $f$} as $$f^{[r]} = a_1x_1 + \cdots + a_nx_n + a_1x_{n+1} + \cdots + a_nx_{2n} + \cdots + a_nx_{rn}.$$ If $F$ is a polynomial that factors into degree one polynomials with no constant terms, $F= f_1 \cdots f_j$, define the \emph{$r^\text{th}$ extension of $F$} as $F^{[r]} = f_1^{[r]} \cdots f_j^{[r]}$. \end{definition}

\begin{remark}\label{extensionmatrix} From the method of expressing permanents using a coefficient of the permanent polynomial,  $$\text{Perm}(A) = [x_1 \cdots x_n] F_A,$$ and $$ \text{Perm}(\mathbf{1}_r \otimes A) = [x_1 \cdots x_{rn}] \left(F_A^{[r]}\right)^r. $$ By construction, $\left( f^{[r]} \right)^r = \left( f^r \right)^{[r]}$.\end{remark}

\begin{proposition}\label{polyextension} Let $h(x_1,...,x_n)$ be a polynomial that factors into degree one polynomials with no constant term. Then, $[x_1 \cdots x_{rn}] h^{[r]} = r!^n [(x_1 \cdots x_n)^r]h$. \end{proposition}

\begin{proof} Let $S_1$ be the permutations of $x_1,...,x_{rn}$ and $S_2$ the permutations of $r$ distinct but indistinguishable copies each of $x_1$, $x_2$, ..., $x_n$. Then, each permutation in $S_2$ appears $r!^n$ times. For $s \in S_t$, $t \in \{1,2\}$, let $s_i$ be the $i^\text{th}$ value in the permutation $s$. Write $h = h_1 \cdots h_k$ as $h$ factored into degree one polynomials, and note that we may assume that $k=rn$ as otherwise the required coefficient is necessarily zero and proof is trivial. Then $$[x_1 \cdots x_{rn}]h^{[r]} = \sum_{s \in S_1} \prod_{i=1}^k[s_i]h_i^{[r]}, \hspace{2mm}\text{and } [(x_1\cdots x_n)^r]h = \frac{1}{r!^k} \sum_{s \in S_2} \prod_{i=1}^k[s_i]h_i.$$ These equations follow from the fact that $h$ factors into degree one polynomials.  If $s_i = x_{a+bn}$ for $a,b \in \mathbb{N}$, let $\tilde{s_i} = x_a$. By the construction of these extensions, \begin{align*} 
[x_1 \cdots x_{rn}]h^{[r]} 
&= \sum_{s \in S_1} \prod_{i=1}^k[s_i]h_i^{[r]} \\
&= \sum_{s \in S_1} \prod_{i=1}^k [\tilde{s_i}]h_i \\
&= \sum_{s \in S_2} \prod_{i=1}^k [s_i]h_i = r!^k[(x_1 \cdots x_n)^r]h .\end{align*} This completes the proof. \end{proof}

We introduce now the Chevalley-Warning Theorem, which extends to sets of polynomials. Here, we include only the single-polynomial version, as it is sufficient for our needs.

\newtheorem*{chev}{The Chevalley-Warning Theorem}

\begin{chev} Let $\mathbb{F}$ be a finite field and $f \in \mathbb{F}[x_1,...,x_n]$ such that $n > \deg(f)$. The number of solutions to $f(x_1,...,x_n) = 0$  is divisible by the characteristic of $\mathbb{F}$. \end{chev}

A proof of this theorem can be found in \cite{Ax}. The following theorem is a corollary to the proof of the Chevalley-Warning Theorem, and will be useful for us to find the appropriate polynomial for our graphs.

\begin{theorem}\label{chevwarn} Let $F$ be a polynomial of degree $N$ in $N$ variables with integer coefficients. Then, $$[(x_1 \cdots x_N)^{p-1}] F^{p-1}  \equiv [F]_p \pmod{p} $$ for primes $p$. \end{theorem}

\begin{proof}  For $0<i < p-1$, note that $\sum_{x \in \mathbb{F}_p}x^i = 0$. Each element of  $\mathbb{F}_p$ is either $0$ or a $p-1$ root of $1$. Therefore, $$[F]_p = \sum_{x \in \mathbb{F}_p}\left( 1- F^{p-1}(x) \right) = -\sum_{x \in \mathbb{F}_p}F^{p-1}(x).$$ Take any monomial $\mathbf{X}= x_1^{u_1} x_2^{u_2} \cdots x_N^{u_N}$ of $F^{p-1}$. Note that $\mathbf{X}$ has degree at most $N(p-1)$. Then, $$ \sum_{x \in \mathbb{F}_p} \mathbf{X} = \prod_{i=1}^N \sum_{x_i \in \mathbb{F}_p} x_i^{u_i} = \prod_{i=1}^N Y(u_i)  ,$$ where $Y(u_i) = \begin{cases} -1 \text{ if $u_i$ is a positive multiple of $p-1$} \\ 0 \text{ otherwise}  \end{cases}.$ By construction, the only monomial with non-zero contribution is therefore $x_1^{p-1} \cdots x_N^{p-1}$, and hence the coefficient of this monomial is equal to the point count $[F]_p$. \end{proof}

For our purposes, consider a fundamental matrix for a graph $G$, and let $v' \in V(G)$ be the special vertex. Write $\text{lcm}(|E(G)|, |V(G)|-1) = L$. Let $\V= \frac{L}{|V(G)|-1}$, so the fundamental matrix $\overline{M}$ is a $\V$-matrix. To emphasizes the graphic construction, write the permanent polynomial $F_{\overline{M}}$ as $F_{G,v'}$. Then, the permanent polynomial has degree $\frac{L}{\V} \cdot \V = L$ in $L$ variables. 

In order to be able to apply Theorem~\ref{chevwarn} we must correct the exponents. Specifically, suppose that $M$ is a fundamental matrix and we want to compute the permanent of $\mathbf{1}_r \otimes M$ modulo prime $p = r \V+1$. By construction, each factor of $F_{G,v'}$ corresponding to a unique row in the matrix comes with exponent $\V$. Create polynomial $\widetilde{F}_{G,v'}$ from $F_{G,v'}$ by taking the positive $\V^\text{th}$ root of $F_{G,v'}$ and then substituting $y_i^\V$ for all $x_i$. As with $F_{G,v'}$, $\widetilde{F}_{G,v'}$ has degree $L$ in $L$ variables.

\begin{example} For graph $K_4$, we have a reduced signed incidence $2$-matrix
 \begin{align*} \overline{M}_{K_4} = \left( \begin{array}{cccccc} 
1&1&1&0&0&0 \\
-1&0&0&1&1&0 \\
0&-1&0&-1&0&1 \\ \hline
1&1&1&0&0&0 \\
-1&0&0&1&1&0 \\
0&-1&0&-1&0&1   \end{array} \right).  \end{align*} This matrix gives polynomials  \begin{align*}F_{K_4,v'}&= (x_1 + x_2+x_3)^2 (-x_1+x_4+x_5)^2 (-x_2-x_4+x_6)^2, \\  \widetilde{F}_{K_4,v'}&= (y_1^2 + y_2^2+y_3^2) (-y_1^2+y_4^2+y_5^2) (-y_2^2-y_4^2+y_6^2). \end{align*}  \end{example}

\begin{lemma}\label{polyswitch} With variables as defined prior and a graph $G$, $$[(x_1 \cdots x_m)^r] F_{G,v'}^r = [(y_1 \cdots y_m)^{p-1}] \left( \widetilde{F}_{G,v'} \right)^{p-1}.$$ \end{lemma}

\begin{proof} Quickly, \begin{align*} [(x_1 \cdots x_{L})^r] F_{G,v'}^r 
&=  [(x_1 \cdots x_{m})^r] \sqrt[\V]{F_{G,v'}}^{r \V} \\
&=[(y_1^\V \cdots y_{m}^\V)^{r}] (\widetilde{F}_{G,v'})^{r\V}\\
&= [(y_1 \cdots y_{m})^{p-1}] (\widetilde{F}_{G,v'})^{p-1}. \end{align*} \end{proof}

We are therefore in position to demonstrate that this polynomial in fact has point count residues equal to the graph permanent, up to a constant multiplicative factor.

\begin{theorem} Let $G$ be a graph, $L= \text{lcm}(|E(G)|,|V(G)|-1)$, and $\V=\frac{L}{|V(G)|-1}$. Let $F_{G,v'}$ be a permanent polynomial for $G$ with special vertex $v' \in V(G)$. Let $p$ be a prime such that $p \equiv 1 \pmod{\V}$, say $p = r\V+1$. Then $$\text{GPerm}^{[p]} \left(G \right) \equiv r!^L [\widetilde{F}_{G,v'}]_p \pmod{p}.$$ \end{theorem}

\begin{proof} 
\begin{align*}
\text{GPerm}^{[p]}(G) &= \text{Perm}(\mathbf{1}_r \otimes M)&& \\
&= [x_1 \cdots x_{rL}](F_{G,v'}^{[r]})^r && \text{Remark~\ref{extensionmatrix}}\\
&= r!^L [(x_1 \cdots x_L)^r] F_{G,v'}^r && \text{Proposition~\ref{polyextension}}\\
&= r!^L [(y_1 \cdots y_L)^{p-1}] (\widetilde{F}_{G,v'})^{p-1} && \text{Lemma~\ref{polyswitch}}\\
&\equiv r!^L [\widetilde{F}_{G,v'}]_p \pmod{p} && \text{Theorem~\ref{chevwarn}}
\end{align*}
    \end{proof}

Recall Corollary~\ref{wilsoncor}; that for prime $p=2n+1$,  $$n!^2 \equiv \begin{cases} -1\pmod{p} \text{ if $n$ is even} \\ 1 \pmod{p} \text{ if $n$ is odd} \end{cases}.$$ It follows from this that for $4$-point $\phi^4$ graphs, the aforementioned constant multiplicative factor difference between the point count of this polynomial and the extended graph permanent residues are at most a constant sign change.

\begin{corollary}\label{maincor} Let $G$ be a $4$-point $\phi^4$ graph, and all variables as defined prior. Then, $$ \text{GPerm}^{[p]}(G) \equiv \begin{cases} [\widetilde{F}_{G,v'}]_p \pmod{p} \text{ if } |E(G)| \equiv 0 \pmod{4} \\ -[\widetilde{F}_{G,v'}]_p \pmod{p} \text{ otherwise} \end{cases} .$$ \end{corollary}

\begin{proof} For a  $4$-point $\phi^4$ graph $G$, $|E(G)|$ is even. By Corollary~\ref{wilsoncor} the proof is immediate. \end{proof}

Interestingly, while there is no natural way to include the prime $2$ in the extended graph permanent for all $4$-point $\phi^4$ graphs using the permanent construction, it can be extracted from the point count of this polynomial. As each variable is squared in this polynomial for $4$-point $\phi^4$ graphs, then as variables over $\mathbb{F}_2$ we may remove these exponents without loss. For a $\phi^4$ graph $G$, then, $[\widetilde{F}_{G,v'}]_2 \equiv [\sqrt{F_{G,v'}}]_2 \pmod{2}$. Then, $|E(G)| > \deg (\sqrt{F_G})$, and by the Chevalley-Warning Theorem, $[\widetilde{F}_{G,v'}]_2 \equiv 0 \pmod{2}$ for all $\phi^4$ graphs.

It is then of particular interest that the $\phi^4$ banana graph, seen in Section~\ref{treesnshit}, has an extended graph permanent that naturally extends to prime $2$ with value $0 \pmod{2}$. Viewed as a graph produced by doubling the edges of the tree $K_2$, there is a natural way to view this graph as $K_2^{[2]}$, an element in the sequence of graphs used to compute the  extended graph permanent of $K_2$, which has value $1 \pmod{2}$ at prime $2$. This is a unique interpretation for the possible value at prime $2$, since no other graphs of interest in $\phi^4$ theory will reduce in a comparable way.

\section{Modular form coefficients}\label{modformscoeffs}

Some extended graph permanent sequences produced (see Appendix~\ref{chartofgraphs}) are recognizable as Fourier coefficients of modular forms. We include here a very brief introduction to modular forms. Notational conventions are adapted from \cite{ModForms}. 

The \emph{modular group} is $$\text{SL}_2(\mathbb{Z}) = \left\{ \left[ \begin{array}{cc} a&b \\ c&d  \end{array} \right] : a,b,c,d \in \mathbb{Z}, ad-bc=1 \right\}.$$  Let $\widehat{\mathbb{C}} = \mathbb{C} \cup \infty$ and $\mathcal{H} = \{\tau \in \mathbb{C} : \text{Im}(\tau) > 0\}$, the upper half plane. For $m= \left[ \begin{array}{cc} a&b \\ c&d  \end{array} \right] \in \text{SL}_2(\mathbb{Z})$ and $\tau \in \mathcal{H}$, define fractional linear transformation $$m(\tau) = \frac{a \tau +b}{c\tau + d}, \hspace{2mm} m(\infty) = \frac{a}{c}.$$ Important subgroups for our purposes are \begin{align*} \Gamma(N) &= \left\{ \left[ \begin{array}{cc} a&b \\ c&d  \end{array} \right] \in \text{SL}_2(\mathbb{Z}) : a \equiv d \equiv 1, b \equiv c \equiv 0 \pmod{N} \right\}, \\  \Gamma_0(N) &= \left\{ \left[ \begin{array}{cc} a&b \\ c&d  \end{array} \right] \in \text{SL}_2(\mathbb{Z}) : c \equiv 0 \pmod{N} \right\}, \\ \Gamma_1(N) &= \left\{ \left[ \begin{array}{cc} a&b \\ c&d  \end{array} \right] \in \text{SL}_2(\mathbb{Z}) : a \equiv d \equiv 1, c \equiv 0 \pmod{N} \right\}, \end{align*} and hence $\Gamma(N) \subseteq \Gamma_1(N) \subseteq \Gamma_0(N)$. A subgroup $\Gamma$ of $\text{SL}_2(\mathbb{Z})$ is a \emph{congruence subgroup of level $N$} if $\Gamma(N) \subseteq \Gamma$ for some $N \in \mathbb{Z}_{>0}$.

\begin{definition} For integer $k$, a function $f:\mathcal{H} \rightarrow \mathbb{C}$ is a \emph{modular form of weight $k$ and level $N$} if $f$ is holomorphic on $\mathcal{H}$ and at infinity, and there is a $k$ in $\mathbb{Z}_{\geq 0}$ such that $$f\left(\left[ \begin{array}{cc} a&b \\ c&d  \end{array} \right](\tau)\right) = (c\tau+d)^k f(\tau)$$ for all $\left[ \begin{array}{cc} a&b \\ c&d  \end{array} \right] \in \Gamma$ and $\tau \in \mathcal{H}$, where $\Gamma$ is one of $\{ \Gamma(N), \Gamma_0(N), \Gamma_1(N)\}$.  \end{definition}

\noindent  We will specify which congruence subgroup among $\Gamma(N)$, $\Gamma_0(N)$, and $\Gamma_1(N)$ is being considered, as there is number theoretic importance to this.

We are interested in sequences generated from the Fourier expansions, known as $q$-expansions, of these modular forms, where $q = e^{2\pi i z}$. Specifically, let $\mathcal{P}$ be the increasing sequence of all primes. For modular form $f$, build a sequence $\left( ([q^p] f) \pmod{p} \right)_{p \in \mathcal{P}}$.

Modular forms are objects of great mathematical interest (see \cite{ModForms}). We were motivated to look for them here by the appearance of modular forms in $c_2$ sequences; a decompletion of $P_{8,37}$ was demonstrated to match a modular form in \cite{BrS}, and it was further conjectured for a number of graphs in \cite{BSModForms} and proved for a subset of those in \cite{Logan}. Here, we find that the sequences from modular forms occasionally appear to match the extended graph permanents. These apparently matching sequences have been checked up to prime $p=97$, and are listed in Table~\ref{rockingtable}. The extended graph permanent sequences can be found in Appendix~\ref{chartofgraphs}, and the Fourier expansions can be found at \cite{lmfdb}. The modular forms are listed by their weights and levels. Those that are representible as a Dedekind $\eta$-function product, a modular form of weight $1/2$ commonly written $$ \eta(\tau) = (e^{2\pi i \tau})^\frac{1}{24} \prod_{n=1}^\infty (1-(e^{2\pi i \tau})^n) $$ (see \cite{ModForms}), have this product included. These are taken from \cite{lmfdb}. The graph $P_{3,1}^2$ is the unique merging of two copies of $P_{3,1}$ per Theorem~\ref{2vertexcut}.

\begin{table}
\begin{align*}
\begin{array}{cccc}
\text{Graph}&\text{Weight}&\text{Level } (\Gamma_1) & \text{Modular form} \\ \hline
P_{3,1} & 3  & 16 & -\eta(4z)^6 \\
P_{4,1} & 4 & 8 & \eta(2z)^4\eta(4z)^4 \\
P_{3,1}^2 & 5 & 4 &  -\eta(z)^4\eta(2z)^2\eta(4z)^4\\
P_{6,1}, P_{6,4} & 6 & 4 & \eta(2z)^{12}\\
P_{6,3} & 6 & 8 &
\end{array} 
\end{align*}
\caption{ Graphs with extended graph permanents that appear to match Fourier expansions of modular forms, with the weight and level of the modular form listed. The notion of products here refers to the two-vertex joins of graphs established in Theorem~\ref{2vertexcut}.}\label{rockingtable}
\end{table}

Some interesting observations can be made here. First, the loop number of the graph is equal to the weight of the modular form in all cases. Secondly, each graph has a modular form with level a power of two. A third observation requires some new terminology. A cusp form is a modular form that has a Fourier expansion with no constant term. Cusp forms of level $M$ can be embedded into cusp forms of level $N$ for any $N$ that is a multiple of $M$ (see Section 5.6 in \cite{ModForms}). A \emph{newform} is a cusp form that is not directly constructed in this manner. The modular forms in Table~\ref{rockingtable} are all newforms, though that is in part due to the method in which we searched for them. It is interesting, though, that in the Dirichlet character decomposition, a particular  decomposition of this space of newforms, these all fall into subspaces of dimension $1$ (see \cite{lmfdb}).

\chapter{Conclusion}
\label{conclusion}

\section{Summary of the main results}

In Section \ref{graphicc2} a graphic representation of the point count of the $c_2$ invariant was developed. While a number of things were proved regarding how these counts worked for individual flows, it was never developed into a method that could be used to prove the conjectured graphical invariant properties of the $c_2$ invariant. It is of interest, though, as this does establish a connection between the $c_2$ invariant and $\mathbb{F}_p$ flows in the graph. Such a connection is natural and should be expected, given the momentum preservation requirement at vertices in the Feynman integral.

In Chapter \ref{Extended}, a new graph invariant was introduced. With the hope of demonstrating a likely connection to the Feynman period, it was shown that, for $4$-point $\phi^4$ graphs, the extended graph permanent is invariant under completion followed by decompletion (Theorem \ref{egpcompletion}), the Schnetz twist (Proposition \ref{schnetz}), and duality (Corollary \ref{phi4dual}), and that any two graphs made by identifying edges in two $4$-point $\phi^4$ graphs will have equal extended graph permanents (Theorem \ref{2vertexcut}). Further, nontrivial $4$-edge cuts were shown to have a product property also, corresponding to graphs with subdivergences (Theorem \ref{subdivthm}). Conjecture \ref{obviousconjecture} arises naturally; if two graphs have equal Feynman periods, then they have equal extended graph permanents.

Three distinct representations of this invariant and the required computation were developed; as a system of tags on edges of the graph (and associated multigraphs), as an edge and vertex weighting representative of the matrix and computation by cofactor expansion, and as the point count of a polynomial. As a system of tags, we were able to establish both decompletion invariance and invariance under the Schnetz twist, as well as a vanishing property based on structural symmetries of the graph. The system of edge and vertex weights allowed for simplification of complicated sequences, and established a number of classes of graphs for which the invariant will be identically zero for all primes. Further, an algorithm for producing a closed form for the sequence was developed and used to produce the extended graph permanents of primitive $4$-point $\phi^4$ graphs up to eight loops for the first few primes (Appendix \ref{chartofgraphs}). The representation as a point count of a polynomial in itself was not used to establish any new properties of the extended graph permanent, but it hints that the possible appearance of modular forms, as noted in Section \ref{modformscoeffs}, is unlikely to be mere coincidence, and certainly opens up additional tools for future directions of study.

\section{Future research directions}

\subsection{The combinatorial $c_2$ invariant}\label{othercombc2}

While the graphic interpretation of the $c_2$ invariant introduced in Section \ref{graphicc2} presents an interesting counting problem, as of yet it has not been useful for solving any of the open problems of the $c_2$ invariant. Currently, no method of dealing with the overcounting has been found to deal with invariance properties for arbitrary graphs. An alternate graphic interpretation has been used with some success for prime $p=2$, in \cite{karenc2} for circulant graphs, and in \cite{wes} for toroidal grids. We introduce this method now.

Recall from Section \ref{c2introduction} the modified Laplacian matrix of graph $G$, $$K_G = \left[ \begin{array}{c|c} A & M_G^T \\ \hline -M_G & \mathbf{0}   \end{array} \right],$$ where $M_G$ is the reduced signed incidence matrix and $A$ is the diagonal matrix with Schwinger parameter entries indexed to align with $M_G$. Similar to the Kirchhoff polynomial produced as a determinant, we define the Dodgson polynomial using a modified matrix. For $I,J,H \subseteq E(G)$, let $K_G(I,J)_H$ be the matrix derived from $K_G$ by deleting rows indexed by edges in $I$, deleting columns indexed by edges in $J$, and setting $x_e = 0 $ for $e \in H$. For sets such that $|I|=|J|$, we define $$\Psi^{I,J}_{G,H} = \det(K_G(I,J)_H).$$ This is well-defined up to overall sign (\cite{Brbig}), due to the arbitrary orientation and choice of vertex deletion. Dodgson polynomials are of particular use in denominator reduction, a method of computing the Feynman period by integrating one edge variable at a time. These polynomials appear in the denominator at each stage of integration (see \cite{Brbig}). 

A combinatorial interpretation of the Dodgson polynomial is as follows. Let $T_{G,H}^{I,J}$ be the collection of edge sets that simultaneously induce spanning trees in both $$(G \cut I) \contract ((J-I) \cup (H-I)) \text{ and } (G \cut J) \contract ((I-J) \cup (H-J))$$ where $G \cut I$ is normal edge deletion and $G \contract J$ is edge contraction. Then, $$\Psi_{G,H}^{I,J} = \sum_{T \in T_{G,H}^{I,J}} \left( \pm \prod_{e \notin T \cup (I \cup J \cup H)} x_e \right). $$ Here, the sign corresponding to each summand is controlled, and there is a common sign for summands that are produced from the same distribution of vertices in $I \cup J$ in the forests that these edge sets induce in $G$ (Proposition 12 in \cite{BrY}).

Of particular use, Theorem 29 in \cite{BrS} states that we may compute the $c_2$ invariant using the denominators produced in denominator reduction after five or more edge integrations, with some consideration for overall sign, and without dividing by $p^2$. This can further be extended back to Dodgson polynomials constructed from sets of three or four edge variables. In particular, $$ c_2(G)_p \equiv \frac{[\Psi_G]_p}{p^2} \equiv - [ \Psi_G^{ac,bc} \Psi_{G,c}^{a,b}] \pmod{p}   $$ for edges $a,b,c \in E(G)$.

This representation is of particular interest for $4$-point $\phi^4$ graphs. Constructed as it is from trees, the degree of $\Psi_G^{ac,bc}$ is $|V(G)|-3$ and the degree of $\Psi_{G,c}^{a,b}$ is $|V(G)|-2$. Hence, the degree of $\Psi_G^{ac,bc} \Psi_{G,c}^{a,b}$ is $2|V(G)|-5 = |E(G)|-3$ in $|E(G)|-3$ variables. We may therefore apply Theorem \ref{chevwarn}; for prime $p$ and $|E(G)|-3 = n$, $$[(x_1 \cdots x_n)^{p-1}] (\Psi_G^{ac,bc} \Psi_{G,c}^{a,b})^{p-1} \equiv [\Psi_G^{ac,bc} \Psi_{G,c}^{a,b}]_p \pmod{p}.$$ Thus, we may evaluate the point count here as the number of decompositions of  $(G-\{a,b,c\})^{[p-1]}$ into an equal number of edge sets in $T_{G}^{ac,bc}$ and $T_{G,c}^{a,b}$.

\begin{example} Consider a decompletion of $P_{3,1} = K_5$. Label the edges as below. $$\includegraphics[scale=1.0]{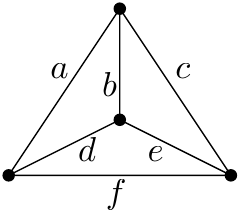}$$ Setting $I = \{a\}$, $J=\{b\}$, and $H = \{c\}$, we compute $\Psi_{K_4}^{ac,bc}$ using edge sets that induce trees in both $G \cut \{a,c\} \contract \{b\}$ and $G \cut \{b,c\} \contract \{a\}$. Hence, edge sets that are spanning trees in both $$\raisebox{-0.4\height}{\includegraphics{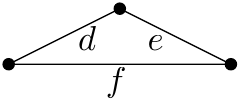}} \hspace{5mm} \text{ and }  \hspace{5mm} \raisebox{-0.4\height}{\includegraphics{c2chevred1}} .$$ As we take the product of edges not in this set, $\Psi_{K_4}^{ac,bc} = x_d + x_e + x_f$. Similarly, $\Psi_{K_4,c}^{a,b}$ is computed using minors $$ \raisebox{-0.4\height}{\includegraphics{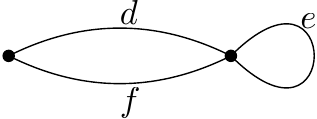}} \hspace{5mm} \text{ and } \hspace{5mm} \raisebox{-0.4\height}{\includegraphics{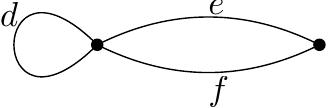}} .$$ As the only set of edges that induces spanning trees in both these graphs is $\{f\}$, $\Psi_{K_4,c}^{a,b} = x_dx_e$. 

Thus, $\Psi_{K_4}^{ac,bc} \Psi_{K_4,c}^{a,b} = x_d^2x_e + x_dx_e^2 + x_dx_ex_f.$ As each term contains $x_d$ and $x_e$, the only way to get an agreement in the powers of all variables in any power of this polynomial is by selecting the $x_dx_ex_f$ term each time. Therefore, $$[(x_dx_ex_f)^{p-1}] (\Psi_{K_4}^{ac,bc} \Psi_{K_4,c}^{a,b})^{p-1} = 1, $$ and the $c_2$ invariant for $P_{3,1}$ is equal to $-1$ at all primes. \end{example}

Again, this interpretation of the $c_2$ invariant has met with some computational success, in particular when coupled with families of graphs with structural symmetries. In \cite{wes}, Chorney and Yeats use this method to show that no such trees could be produced simultaneously in both minors, and hence the $c_2$ invariant vanishes, for an infinite family of graphs known as non-skew toroidal grids at prime $p=2$. In \cite{karenc2}, Yeats showed that such a collection of edge sets inducing trees in these minors force a particular choice of the polynomial in which every variable must appear, producing non-trivial terms in the $c_2$ invariant for certain families of circulants at prime $p=2$.

\subsection{Sequences that appear as both $c_2$ invariants and extended graph permanents}

This interpretation as a decomposition into common spanning trees of particular minors is further suggestive of a connection to the extended graph permanent. We may view the $c_2$ invariant as a decomposition into spanning forests with a particular set of rules as to the distribution of edges based on the edges used in constructing the Dodgson polynomial. The graphic interpretation of the extended graph permanent presented in Section \ref{graphicegp} can be interpreted as a decomposition of a graph $G$ into (not necessarily connected) spanning subgraphs $g$ with $|E(g)| = |V(g)|-1$, where the special vertex must be contained in the connected component that is a tree.

We have seen sequences produced by the $c_2$ invariant and the extended graph permanent that appear to be equal, albeit for different graphs. The graph $P_{3,1}^2$ appears to have extended graph permanent equal to the $c_2$ invariant of decompletions of graphs $P_{9,161}$, $P_{9,170}$, $P_{9,183}$, and $P_{9,185}$ (see \cite{BSModForms}). That these graphs are non-isomorphic opens up a number of immediate questions. Is there a standard, structural mapping between graphs $f(G_1) = G_2$ such  that the extended graph permanent of $G_1$ is equal to the $c_2$ invariant of $G_2$? In particular, is there a structural explanation for these observed sequences? Should such a mapping exist, it may well be complicated; the observed connection between the term-by-term square of $P_{3,1}$ or $P_{3,1}^2$ and the set of graphs $\{P_{9,161},P_{9,170},P_{9,183},P_{9,185}\}$ up to sign ambiguity of the extended graph permanent would suggest such a mapping would connect smaller loop-level extended graph permanents to larger level $c_2$ invariants. Unfortunately, systematic data for $c_2$ invariants at these levels is incredibly difficult to produce. 

The connection between graphs with extended graph permanent zero for all primes and graphs with $c_2$ invariants zero for all primes, and similarly $\pm 1$ for all primes, may also be of use. The meaning of graphs with $c_2$ equal to zero at all primes is important, as it indicates the graph is weight-drop (see \cite{BrS}). Meanwhile, a number of subclasses of graphs with extended graph permanent equal to zero for all primes are introduced here in Theorem \ref{mymaxflow}, though more may exist. Further, graphs with $c_2$ equal to $-1$ for all primes are of particular interest (\cite{BrS}), and the constant $-1$ sequence can be found as the extended graph permanent for all trees with an even number of vertices. An interesting question then; for which graphs does this sequence appears as the extended graph permanent? Trivially, trees are the only connected graphs that have an extended graph permanent value at prime $p=2$, and so describing the sequence by what primes appear force that trees are the only graphs that match this sequence. If we consider the sequence without regarding the sequence of primes used in the construction, though, the banana graph, mentioned in Section \ref{treesnshit}, also produces this sequence. Apart from these, and any graph made by duplicating all edges a fixed number of times in a tree, it is not known if any other graphs do. Matching sequences of interest in the $c_2$ invariant with families of graphs with equal extended graph permanents would be an interesting area of further study.

\subsection{Extended graph permanents as modular form Fourier expansions}

As mentioned prior, it was this possible connection to the $c_2$ invariant that inspired the search for modular form sequences in the extended graph permanent. A number of non-trivial connections between these modular forms and the graphs producing these sequences -- that loop numbers and weights of the modular forms match for equal sequences, that the levels of these modular forms are powers of two, and that the Dirichlet character decompositions all put these sequences in subspaces of dimension $1$ -- have been observed, and suggest that these connections are not coincidental. Some interesting questions then arise. Are there more extended graph permanent sequences that match the Fourier expansions of modular forms? Is there is an explainable reason for the patterns spotted in Section \ref{modformscoeffs}, and do these patterns continue?

\subsection{The extended graph permanent of matroids}

The behaviour of the extended graph permanent is of particular interest with regard to graphic matroids. Should Conjecture \ref{qwhitney} hold, coupled with Remark \ref{connected}, it would mean that graphs with equal graphic matroids necessarily have equal extended graph permanents. 

\begin{conjecture*}[Conjecture \ref{qwhitney}] If two graphs differ by a Whitney flip, then they have equal extended graph permanents.   \end{conjecture*}

\noindent The equal extended graph permanent sequences of matroid $R_{10}$, which is certainly not graphic, and graph $P_{3,1}^2$ make a potential matroid connection much more tantalizing. Every regular matroid has a well-defined extended graph permanent. As noted in Chapter \ref{chmatroid}, the Feynman integral is known to be preserved by the Whitney flip. It follows that, if the extended graph permanent is connected to the Feynman integral in some sense beyond $\phi^4$ graphs, as we would want it to be, this conjecture holds true.

\subsection{Identifying graphs with equal extended graph permanents}

In Chapter \ref{egpcomp}, graphic representation led to closed forms for the permanents of the matrices themselves, as well as the residues. These closed forms are a computational boon, as otherwise permanent computations from the matrices themselves can be oppressively difficult. Further, in Section \ref{niceness}, we ask Question \ref{equationreconstruction}.

\begin{question*}[Question \ref{equationreconstruction}] What graphs can be uniquely reconstructed from the algorithmically generated closed form, if we allow no non-trivial simplification? If we allow simplification, is there a family of graphs for which this closed form can still be used to uniquely reconstruct the graph? \end{question*} 

\noindent If all graphs, or even a family of graphs, can be reconstructed, then it may be possible to understand when two graphs will have equal extended graph permanents in a computational manner, using combinatorial identities to move between these closed forms. Even if they cannot, constructing families of graphs with identical closed forms may explain some identical extended graph permanents in graphs that currently cannot be explained. This may further provide some insight as to additional period preserving graph operations, hinted at by Conjecture \ref{heppconjecture}.

It was a notable result for planar duality in Section \ref{secdual} that the sequence of primes for which the extended graph permanent is defined is equal for planar duals, and a surprising connection between the extended graph permanents of graphs and their duals that do not have the vertex-edge relation seen in $4$-point $\phi^4$ graphs. I suspect this may be a particularly interesting area of future study, to see how other general operations on graphs affect the extended graph permanent, and possibly search for connection that may appear.

Conjecture \ref{obviousconjecture} is likely difficult to solve. 

\begin{conjecture*}[Conjecture \ref{obviousconjecture}] If two $4$-point $\phi^4$ graphs have equal period, then they have equal extended graph permanent. \end{conjecture*}
 
\noindent The results herein only suggest a connection, but as we are trying to relate an integral to an infinite sequence of residues, there is no obvious method of solving this conjecture. The Hepp bound, and in particular Conjecture \ref{heppconjecture} suggests that there are additional period preserving graph operations. 

\begin{conjecture*}[Conjecture \ref{heppconjecture}] Two $4$-point $\phi^4$ graphs have equal Hepp bound if and only if they have equal Feynman periods.  \end{conjecture*}

\noindent A result of this conjecture, as noted prior, is that some graphs have been conjectured to have equal periods that cannot be explained by any of the currently known period preserving graph operations. This is not at odds with Conjecture \ref{obviousconjecture}, though, as the extended graph permanents of these graphs are in fact equal. Thus, Conjecture \ref{heppconjecture} provides further support for Conjecture \ref{obviousconjecture}.

The list of $\phi^4$ graphs up to loop order eight and the extended graph permanents up to prime $p=41$ can be seen in Appendix \ref{chartofgraphs}. While the converse of Conjecture \ref{obviousconjecture} does not appear to hold -- for example, $P_{6,1}$ and $P_{6,4}$ seem to have the same sequence -- the sequences for both pair $P_{8,30}$ and $P_{8,36}$ and pair $P_{8,31}$ and $P_{8,35}$ have at least one graph with unknown period, though both pairs are conjectured to have equal periods, following Conjecture \ref{heppconjecture}. When two graphs will have equal extended graph permanents but non-equal periods is an interesting problem. If such graphs could be identified, it might be possible to create classes of graphs for which any two graphs in the class have equal extended graph permanent if and only if they differ by an operation known to preserve the period.

\subsection{A connection between the extended graph permanent and the Feynman period}

As mentioned prior, Conjecture \ref{obviousconjecture} is likely difficult to solve. However, all available data suggests that it is true: the invariance of the extended graph permanent under all of the known period preserving operations as demonstrated in Chapter \ref{Invariance}, and further the equality in the extended graph permanents of graphs with periods that are conjectured to be equal, per Conjecture \ref{heppconjecture}. A further hint of the truth of Conjecture \ref{obviousconjecture} follows from an alternate proof of Theorem \ref{makingorientations}. The following content is from \cite{pppoaiftposim}.

Let $G$ be a graph such that $|E(G)| = 2|V(G)|-2$, and apply an arbitrary orientation to its edges. Take a spanning tree $T$ of $G$, and apply an arbitrary value in $\mathbb{F}_3$ to edges $e \in E-E(T)$, call it $y_e$. For edges $f \in E(T)$, there exists a unique edge cut that contains $f$ but no other edges of $E(T)$, called the \emph{fundamental cut}. Let $C^+_f$ be the edges in this cut from $E-E(T)$ that are oriented in the same direction as $f$, and $C^-_f$ the edges that are oriented in the opposite direction. Using variables $\{y_e\}_{e \in E-E(T)}$, define the polynomials $$g_f = \sum_{e \in C^-_f} y_e - \sum_{e \in C^+_f} y_e$$ for $f \in E(T)$. As in the proof of Proposition \ref{numberflows}, every assignment of values to the variables $y_e$ extends uniquely to a $\mathbb{F}_3$ flow. 

Importantly, the value that the edge $f$ receives in this flow is $g_f$. As the boundary at every vertex is equal to zero, the sum of the boundaries over vertices on either side of the cut must be equal to zero. It follows that the sum of values on the edges in the cut must be equal traveling in either direction, and, extending the variable notation to edges in the tree $T$, \begin{align*}  y_f + \sum_{e \in C^+_f} y_e &= \sum_{e \in C^-_f} y_e\\ y_f & = \sum_{e \in C^-_f} y_e  -  \sum_{e \in C^+_f} y_e . \end{align*} Define the polynomial $$ g = \prod_{f \in E(T)} g_f.$$ If the coefficient of $\mathbf{y} = \prod_{e \in E-E(T)} y_e$ in the expansion of $g$ is nonzero, then by Theorem \ref{alontarsipoly} there exists an assignment to the variables $y_e$ in $\{\pm1\}$ such that $g$ is nonzero. This creates a nowhere-zero flow, as all of the values assigned to $f \in E(T)$ must be nonzero in this product.

From the proof of Theorem \ref{makingorientations}, the coefficient of $\mathbf{x}$ was equal to $\text{Perm}(\mathbf{1}_{2 \times 1} \otimes M_G)$ for a reduced signed incidence matrix $M_G$. Using the tools introduced in Section \ref{secdual}, if we assume that the first columns in $M_G$ are indexed by edges in $E(T)$, then $M_G$ row reduces to $[I|A]$ for some matrix $A$. Performing cofactor expansion on these columns in the identity matrix, $$ \text{Perm}(\mathbf{1}_{2\times 1} \otimes M_G) = \pm \text{Perm}\left( \left[ \begin{array}{cc} I& A \\ I & A \end{array} \right]  \right)   = \pm \text{Perm}(A) \hspace{5mm} \text{in } \mathbb{F}_3 .$$ Here, coefficient $\mathbf{y}$ is equal to $\text{Perm}(A)$. This is a core property of the fundamental cuts and the row reduced matrix, as we have constructed it. Again, a nonzero value in the extended graph permanent provides a certificate for the existence of a nowhere-zero flow in $\mathbb{F}_3$.

This polynomial $g$ is closely related to the Feynman integral in momentum space (see, for example, \cite{Sphi4}). For a graph $G$, take a basis for the cycle space of the graph by taking tree $T$ and the fundamental cycle in $T \cup \{e\}$ for $e \in E(G)-E(T)$. Assign a variable to each cycle in this basis, $y_e \in \mathbb{R}^4$. These values represented the flow of momentum around the cycles. To all edges of the graph, associate the signed sum of the variables for the cycles running through that edge. This is $g_f$ for $f \in E(T)$ and $y_e$ for $e \notin E(T)$. Define $$\tilde{g} = \prod_{f \in E(T)} g_f^2 \prod_{e \notin E(T)} y_e^2, $$ where vector multiplication is taken to be the dot product. As a formal integral in a momentum space, the Feynman integral in massless scalar field theory is $$\int_{(\mathbb{R}^D)^{h_G}} \frac{1}{\tilde{g}}\prod_{e \notin E(T)} \diff y_e, $$ the massless version of the Feynman integral presented in Section \ref{introbackground} by a change of variables. In both cases, the number of free variables is equal to the dimension of the cycle space. In Section \ref{introbackground}, the variables correspond to the chosen basis of the cycle space, here they correspond to the edges $y_e$ that index a basis of cycles.

Comparing $g$ and $\tilde{g}$, there are two main differences. First, the dot product replaces the normal variables, an understandable change as we move to vector variables. Second, $\tilde{g}$  has a factor for the edges not in $E(T)$. An analogous factor in $g$ would not affect the use of the needed polynomial technique.

\medskip

Again, it is likely difficult to prove Conjecture \ref{obviousconjecture}, but a substantial amount of non-trivial evidence suggests that the connection between the Feynman period and the extended graph permanent is in fact there. As we wanted to create a graph invariant with this connection -- an object that could be used to gain insight into the Feynman period -- this suggests that we have done so. The potential for this connection suggests that further study into the properties of the extended graph permanent is a worthy endeavor.


%
%
%
%
%

\backmatter%
	\addtoToC{Bibliography}
	\bibliographystyle{plain}
	\bibliography{references}

\begin{appendices} 
	\chapter{The extended graph permanent of small $\phi^4$ graphs}
\label{chartofgraphs}

The following charts include the first few primes for all $\phi^4$ graphs up to loop order $8$ where permanent preserving operations have not made the calculation trivial. The naming convention as a family of graphs with equal completions comes from \cite{Sphi4}, and representations of the completed graphs may also be found there. Graphs with alternate common names of prior interest are noted when applicable, and when it is the decompleted graph that has a common name, this will be marked in parenthesis. Grey columns mark values that may be thought of as fixed, while all others are defined collectively up to sign in each row.

Recall from Conjecture~\ref{heppconjecture} in Chapter~\ref{hepp} that it is believed that the periods of $P_{8,31}$ and $P_{8,35}$ are equal, as well as the periods of $P_{8,30}$ and $P_{8,36}$.

Decompletions of graphs $P_{3,1}$, $P_{7,5}$, $P_{7,9}$, $P_{8,18}$, $P_{8,25}$, $P_{8,31}$, and $P_{8,35}$ have zeros in the sequence that correspond to the result of Corollary~\ref{symmetrycor}. It has been verified that each of these graphs has a decompletion with a symmetry that explains this. The twist and dual notion are as described in Section~\ref{introinvariance}, and graphs that differ by these operations have equal extended graph permanents by Proposition~\ref{schnetz} and Corollary~\ref{phi4dual}.

\makeatletter
\newcommand\xleftrightarrow[2][]{%
  \ext@arrow 9999{\longleftrightarrowfill@}{#1}{#2}} 
\newcommand\longleftrightarrowfill@{%
  \arrowfill@\leftarrow\relbar\rightarrow}
\makeatother

\definecolor{Gray}{gray}{0.9}
\newcolumntype{g}{>{\columncolor{Gray}}c}

\begin{align*}
\begin{array}{c|cgccggccgcgg}
\textbf{Graph}&\multicolumn{12}{l}{\textbf{Prime}} \\
&3&5&7&11&13&17&19&23&29&31&37&41 \\ \hline
P_{1,1} &1&4&1&1&12&16&1&1&28&1&36&40 \\ \hline
P_{3,1} = K_5 = (W_3) = C^5_{1,2} &0&1&0&0&3&13&0&0&16&0&33&23 \\ \hline
P_{4,1} = (W_4) = C^6_{1,2} &1&3&4&0&9&16&13&10&24&5&23&7 \\ \hline
P_{5,1} = C^7_{1,2} &1&1&1&5&12&16&11&13&7&1&25&9 \\ \hline
P_{6,1} = C^8_{1,2} &0&4&3&1&11&16&0&13&15&9&35&6\\
P_{6,2} &1&3&5&8&8&15&10&17&27&20&32&1\\
P_{6,3} &1&1&3&8&10&9&15&0&24&24&3&11\\
P_{6,4} = K_{3,4} &0&4&3&1&11&16&0&13&15&9&35&6
\end{array} 
\end{align*}

\begin{align*}
\begin{array}{c|cgccggccgcgg}
\textbf{Graph}&\multicolumn{12}{l}{\textbf{Prime}} \\
&3&5&7&11&13&17&19&23&29&31&37&41 \\ \hline
P_{7,1} = C^9_{1,2} &1&3&3&4&1&15&7&14&13&13&28&0 \\
P_{7,2} &1&2&0&9&9&6&6&12&25&9&0&31 \\
P_{7,3} &0&0&3&8&5&3&2&14&10&18&23&34 \\
P_{7,4} \xleftrightarrow{\text{twist}} P_{7,7} &1&0&4&5&9&1&4&4&4&7&26&0 \\
P_{7,5} \xleftrightarrow{\text{dual}} P_{7,10} &0&3&0&0&1&11&0&0&13&0&26&36\\
P_{7,6}  &1&1&1&8&10&9&7&14&28&16&35&36 \\
P_{7,8} &1&1&2&0&10&16&17&8&4&25&26&33 \\
P_{7,9} &0&0&0&0&10&2&0&0&17&0&1&0 \\
P_{7,11} &0&1&1&1&11&5&0&22&6&25&16&38 \\ \hline
P_{8,1} = C_{1,2}^{10} &1&1&5&10&7&14&17&4&8&11&19&7\\
P_{8,2} &1&0&4&0&10&6&12&12&27&17&34&0\\
P_{8,3} &1&0&1&1&9&10&14&3&8&17&15&22\\
P_{8,4} &1&3&4&0&7&16&3&11&23&23&11&17\\
P_{8,5} &0&2&1&0&0&16&17&9&12&2&33&26\\
P_{8,6} \xleftrightarrow{\text{twist}} P_{8,9} &0&0&3&0&4&5&6&6&3&13&28&24\\
P_{8,7} \xleftrightarrow{\text{twist}} P_{8,8} &1&1&0&2&0&3&13&2&22&7&25&31\\
P_{8,10}\xleftrightarrow{\text{twist}} P_{8,22} &1&1&5&10&7&14&17&4&8&11&19&7\\
P_{8,11} \xleftrightarrow{\text{twist}} P_{8,15} &1&3&1&1&8&14&0&1&13&20&15&24\\
P_{8,12} &1&1&6&0&7&0&6&15&10&29&11&30 \\
P_{8,13} \xleftrightarrow{\text{twist}} P_{8,21} &1&4&4&7&1&12&7&11&28&11&24&26\\
P_{8,14} &0&3&3&2&2&11&12&3&1&27&30&27\\
P_{8,16} &1&3&1&10&3&1&5&16&3&12&23&5\\
P_{8,17}\xleftrightarrow{\text{twist}} P_{8,23} &0&4&2&0&4&0&9&1&27&7&22&17\\
P_{8,18}\xleftrightarrow{\text{twist}} P_{8,25} &0&3&0&0&0&4&0&0&3&0&15&12\\
P_{8,19} \xleftrightarrow{\text{dual}} P_{8,27} &1&4&4&4&10&2&15&6&3&27&28&36\\
P_{8,20} &1&2&3&2&1&15&6&7&14&25&12&38\\
P_{8,24} &1&2&1&6&7&5&3&5&8&5&25&31\\
P_{8,26}\xleftrightarrow{\text{twist}} P_{8,28} &1&1&0&7&1&10&15&16&6&9&2&12\\
P_{8,29} &1&3&5&8&1&15&13&17&8&23&6&15\\
P_{8,30} &1&4&3&4&6&5&2&21&11&5&34&28\\
P_{8,31} &0&3&0&0&3&1&0&0&25&0&35&13\\
P_{8,32}\xleftrightarrow{\text{twist}} P_{8,34} &1&0&1&1&9&10&14&3&8&17&15&22\\
P_{8,33} &0&1&0&0&7&3&7&19&20&29&3&33\\
P_{8,35} &0&3&0&0&3&1&0&0&25&0&35&13\\
P_{8,36} &1&4&3&4 &6 &5   &2   &21  &11&5 &34 &28\\
P_{8,37} &1&1&5&0 &11&5  &13 &7   &13&30   &16 &15\\
P_{8,38} &1&2&0&1 &1  &4  &6   &15 &11&18 &28 &29\\
P_{8,39} &0&0&3&0&4&5&6&6&3&13&28&24\\
P_{8,40} &1&1&5&10&7 &14 &17 &4  &8  &11 &19 &7\\
P_{8,41} &0&3&1&5&12&2&18&15&9&25&27&34 
\end{array} 
\end{align*}
	\chapter{Small equations}
\label{chartofequations}

What follows are the equations for producing the extended graph permanents for primitive $4$-point $\phi^4$ graphs $G$ up to seven loops, using the methods developed in Chapter~\ref{egpcomp}. Naming conventions come from \cite{Sphi4} as a family of decompletions of a $4$-regular graph. Adopting a shorthand, the summation is from $0$ to $n$ for each variable, though in many instances further restrictions are possible and will speed up computations. The extended graph permanent at prime $p = 2n+1$ is then the residue modulo $p$ for each of these equations, up to a factor of $(-1)^n$ corresponding to changing the direction of an edge in the underlying orientation.

As noted prior, the graphs $P_{7,4}$ and $P_{7,7}$ differ by a Schnetz twist, and graph $P_{7,5}$ and $P_{7,10}$ have decompletions that are planar duals and hence equal sequences of residues by Corollary~\ref{phi4dual}. Both pairs of graphs are included here for comparative purposes.

\footnotesize

\begin{align*}
\begin{array}{l|l}
\textbf{G}&\textbf{Equation} \\ \hline
P_{ 1 , 1 } &(2n)!\\
P_{ 3 , 1 } &
(2n)!^ 3
\sum
\binom{n}{ x }^3
(-1)^{ x }\\
P_{ 4 , 1 } &
(2n)!^ 4
\sum
\binom{n}{ x }^4\\
P_{ 5 , 1 } &
(2n)!^ 5
\sum
\binom{n}{ x_0 }^3
\binom{n}{ x_1 }^2
\binom{n}{ x_0 + x_1 }
(-1)^{x_1 }\\
P_{ 6 , 1 } &
(2n)!^ 6
\sum
\binom{n}{ x_0 }^3
\binom{n}{ x_1 }
\binom{n}{ x_2 }
\binom{n}{ 2n - x_1 - x_2 }
\binom{n}{ x_0 + x_1 }
\binom{n}{ -n + x_0 + x_1 + x_2 }
(-1)^{  x_0 +  x_2 }\\
P_{ 6 , 2 } &
(2n)!^ 6
\sum
\binom{n}{ x_0 }^2
\binom{n}{ x_1 }
\binom{n}{ x_2 }
\binom{n}{  x_1 + x_2 }
\binom{n}{ x_0 + x_1 }^2
\binom{n}{ n - x_0 + x_2 }
(-1)^{ x_0  + x_2 }\\
P_{ 6 , 3 } &
(2n)!^ 6
\sum
\binom{n}{ x_0 }^2
\binom{n}{ x_1 }^2
\binom{n}{ 2n - x_0 - x_1 }
\binom{n}{ x_2 }^2
\binom{n}{ -n + x_0 + x_1 + x_2 }
(-1)^{ x_2 }\\
P_{ 6 , 4 } &
(2n)!^ 6
\sum
\binom{n}{ x_0 }^2
\binom{n}{ x_1 }^2
\binom{n}{ x_2 }^2
\binom{n}{ 2n - x_0 - x_1 - x_2 }^2\\
P_{ 7 , 1 } &
(2n)!^ 7
\sum
\binom{n}{ x_0 }^3
\binom{n}{ x_1 }
\binom{n}{ x_2 }
\binom{n}{ x_3 }
\binom{n}{ 2n - x_1 - x_2 - x_3 }
\binom{n}{ x_0 + x_1 }
\binom{n}{ x_0 + x_1 + x_2 }
\binom{n}{ -x_0 + x_3 }
(-1)^{ x_2 + x_3 }\\
P_{ 7 , 2 } &
(2n)!^ 7
\sum
\binom{n}{ x_0 }^3
\binom{n}{ x_1 }
\binom{n}{ 2n - x_0 - x_1 }
\binom{n}{ x_2 }
\binom{n}{ x_3 }
\binom{n}{ n - x_2 - x_3 }
\binom{n}{ x_1 + x_2 }
\binom{n}{ -n + x_0 + x_1 + x_2 + x_3 }
(-1)^{x_0 + x_3 }\\
P_{ 7 , 3 } &
(2n)!^ 7
\sum
\binom{n}{ x_0 }^2
\binom{n}{ x_1 }^2
\binom{n}{ 2n - x_0 - x_1 }
\binom{n}{ x_2 }
\binom{n}{ x_3 }
\binom{n}{ n - x_2 - x_3 }
\binom{n}{ x_0 + x_2 }
\binom{n}{ -n + x_0 + x_1 + x_2 + x_3 }
(-1)^{x_0 +  x_3 }\\
P_{ 7 , 4 } &
(2n)!^ 7
\sum
\binom{n}{ x_0 }
\binom{n}{ x_1 }^2
\binom{n}{ 2n - x_0 - x_1 }
\binom{n}{ x_2 }
\binom{n}{ x_3 }
\binom{n}{ n - x_2 - x_3 }
\binom{n}{ x_0 + x_2 }^2
\binom{n}{ -n + x_0 + x_1 + x_2 + x_3 }
(-1)^{ x_0 +x_2 + x_3 }\\
P_{ 7 , 5 } &
(2n)!^ 7
\sum
\binom{n}{ x_0 }^2
\binom{n}{ x_1 }^2
\binom{n}{ x_2 }
\binom{n}{ x_3 }
\binom{n}{ n - x_2 - x_3 }
\binom{n}{ x_0 + x_2 }^2
\binom{n}{ -x_0 + x_1 + x_3 }
(-1)^{  x_0 + x_1 + x_3 }\\
P_{ 7 , 6 } &
(2n)!^ 7
\sum
\binom{n}{ x_0 }^2
\binom{n}{ x_1 }
\binom{n}{ x_2 }^2
\binom{n}{ x_3 }
\binom{n}{ 2n - x_1 - x_2 - x_3 }
\binom{n}{ x_0 + x_1 }^2
\binom{n}{ -x_0 + x_3 }
(-1)^{ x_0 +  x_2 + x_3 }\\
P_{ 7 , 7 } &
(2n)!^ 7
\sum
\binom{n}{ x_0 }^2
\binom{n}{ x_1 }^2
\binom{n}{ 2n - x_0 - x_1 }
\binom{n}{ x_2 }
\binom{n}{ x_3 }
\binom{n}{ 2n - x_2 - x_3 }
\binom{n}{ -n + x_0 + x_2 }
\binom{n}{ -n + x_1 + x_3 }
(-1)^{x_2 + x_3 }\\
P_{ 7 , 8 } &
(2n)!^ 7
\sum
\binom{n}{ x_0 }^2
\binom{n}{ x_1 }^2
\binom{n}{ x_2 }
\binom{n}{ x_3 }
\binom{n}{ n - x_2 - x_3 }
\binom{n}{ x_0 + x_2 }
\binom{n}{ x_1 + x_3 }
\binom{n}{  x_0 + x_1 + x_2 + x_3 }\\
P_{ 7 , 9 } &
(2n)!^ 7
\sum
\binom{n}{ x_0 }^2
\binom{n}{ x_1 }^2
\binom{n}{ x_2 }
\binom{n}{ x_3 }
\binom{n}{ n - x_2 - x_3 }
\binom{n}{ x_0 + x_2 }
\binom{n}{ x_1 + x_3 }
\binom{n}{ n + x_0 - x_1 - x_3 }
(-1)^{ x_2 }\\
P_{ 7 , 1 0 } &
(2n)!^ 7
\sum
\binom{n}{ x_0 }^2
\binom{n}{ x_1 }
\binom{n}{ x_2 }
\binom{n}{ x_3 }
\binom{n}{ 2n - x_1 - x_2 - x_3 }
\binom{n}{ x_0 + x_1 }
\binom{n}{ -n + x_0 + x_1 + x_3 }^2
\binom{n}{ 2n - x_0 - x_1 - x_2 - x_3 }
(-1)^{x_2 + x_3 }\\
P_{7,11} & (2n)!^7 \sum
\binom{n}{ x_0 }^2
\binom{n}{ x_1 }^2
\binom{n}{ x_2 }
\binom{n}{ x_3 }
\binom{n}{ 2n - x_2 - x_3 }^2
\binom{n}{ -n + x_0 + x_1 + x_2 }
\binom{n}{ -x_0 + x_3 }(-1)^{ x_1  }
\end{array} 
\end{align*}

\normalsize

\end{appendices}
\end{document}